\documentclass[11pt,reqno]{amsart}

\usepackage{amsmath,amssymb,amsthm,graphicx}
\usepackage[dvipsnames]{xcolor}
\usepackage[urlcolor=ForestGreen,citecolor=Maroon,ocgcolorlinks,bookmarksopen=True]{hyperref}
\usepackage{pgf,tikz}
\usepackage{mathrsfs}
\usepackage[all]{xy}
\usepackage{xfrac}
\usetikzlibrary{math,matrix,arrows,backgrounds,shapes.misc,shapes.geometric,patterns,calc,positioning,shapes,decorations.pathmorphing}
\usepackage{ifthen}
\usepackage[left=2.6cm,right=2.6cm,top=2.7cm,bottom=2.8cm]{geometry}
\usepackage{lscape}
\usepackage{caption}
\usepackage{subcaption}
\usepackage{MnSymbol}

\usepackage{enumitem}
\makeatletter
\def\namedlabel#1#2{\begingroup
   \def\@currentlabel{#2}%
   \label{#1}\endgroup
}
\makeatother

\theoremstyle{definition}
\newtheorem{theorem}{Theorem}[section]
\newtheorem{lemma}[theorem]{Lemma}
\newtheorem{coro}[theorem]{Corollary}
\newtheorem*{coro*}{Corollary}
\newtheorem{prop}[theorem]{Proposition}
\newtheorem{defn}[theorem]{Definition}
\newtheorem{nota}[theorem]{Notation}
\newtheorem{ans}[theorem]{Definition}

\newtheorem{rem}[theorem]{Remark}
\newtheorem{que}[theorem]{Question}
\newtheorem{ex}[theorem]{Example}
\newtheorem{thmIntro}{Theorem}    \renewcommand{\thethmIntro}{\Alph{thmIntro}}

\newcommand{\R}{\mathbb{R}} %%% real numbers

\newcommand{\NumCol}{NavyBlue} %%% color of the reference numbers

\newcommand{\CosPiOverFive}{0.80901699437} %%% x*cos(pi/5) 
\newcommand{\SinPiOverFive}{0.58778525229} %%% x*sin(pi/5) 
\newcommand{\CosTwoPiOverFive}{0.30901699437} %%% x*cos(2pi/5) 
\newcommand{\SinTwoPiOverFive}{0.951056516295} %%% x*sin(2pi/5) 

\begin{document}

\date{\today}


\title[Exchange graphs for mutation-finite non-integer quivers of rank $3$]{Exchange graphs for mutation-finite non-integer quivers of rank $3$}

\author{Anna Felikson}
\address{Department of Mathematical Sciences\\
Durham University\\
Science Laboratories\\
South Road\\
Durham, DH1 3LE\\
United Kingdom}
\email{anna.felikson@durham.ac.uk}

\author{Philipp Lampe}
\address{School of Mathematics, Statistics and Actuarial Science\\
University of Kent\\
Sibson Building\\
Parkwood Road\\
Canterbury, CT2 7FS\\
United Kingdom}
\email{P.B.Lampe@kent.ac.uk}

\begin{abstract} Skew-symmetric non-integer matrices with real entries can be viewed as quivers with non-integer weights of arrows. One can mutate such quivers according to usual rules of quiver mutation. Felikson and Tumarkin show that rank 3 mutation-finite non-integer quivers admit geometric realisations by partial reflections. This allows to define a notion of seeds (as Y-seeds), and hence, to define the exchange graphs for mutation classes. In this paper we study exchange graphs of mutation-finite quivers in rank 3. The concept of finite and affine type generalises naturally to non-integer quivers. In particular, exchange graphs of finite type quivers are finite, while exchange graphs of affine quivers are finite modulo the action of a finite-dimensional lattice.
\end{abstract}

\maketitle

\section{Introduction and summary}

Cluster algebras were introduced and studied in a series of four articles by Fomin--Zelevinsky, see \cite{FZ1,FZ2,FZ4}, and Berenstein--Fomin--Zelevinsky, see \cite{BFZ}. Many beautiful properties of cluster algebras are visualised by the exchange graph. For example, cluster algebras attached to Dynkin diagrams have finite exchange graphs, see \cite{FZ2}, and cluster algebras attached to affine Dynkin diagrams have exchange graphs of linear growth, see Felikson--Shapiro--Tumarkin  \cite[Theorem 10.8]{FST} and Felikson--Shapiro--Thomas--Tumarkin \cite[Theorem 1.1]{FSTT}.

In Fomin--Zelevinsky's setting, all quivers are integer-valued, that is, they are given by exchange matrices with integer entries. More generally, one can consider quivers whose multiplicities of arrows are not necessarily integers. Such quivers correspond to exchange matrices with real entries. Those objects were first considered in the context of almost periodicity, see \cite{L}. However, a concept of cluster variables has not yet been developed for quivers with real weights.

Fortunately, we can think about mutation of real quivers geometrically. The article \cite{FT} provides a model for all $3\times 3$ exchange matrices with real weights. More precisely, every mutation-acyclic quiver of rank $3$ admits a geometric realisation based on partial reflections, and every mutation-cyclic quiver of rank $3$ admits a realisation by partial rotations. These geometric models allow to classify real exchange matrices of finite mutation type in rank $3$ (and in general rank, as shown later in \cite{FT2}). In this article we investigate exchange graphs for mutation-finite non-integer mutation classes of rank $3$ quivers. As it is shown in \cite{FT}, any mutation-finite mutation class in rank $3$ (except for the mutation class of the Markov quiver) has a model by partial reflections on the sphere $\mathbb{S}^2$ or in the Euclidean plane $\mathbb{E}^2$. The mutation classes modelled in $\mathbb{S}^2$ have finitely many seeds, and therefore they are said to be of finite type. The other mutation-finite mutation classes are called affine, because they are modelled by Euclidean geometry. The model uses the notion of a $Y$-seed to define an exchange graph.

According to the classification by \cite{FT}, there are exactly $5$ mutation classes of finite type in rank $3$. Two of these classes are given by classical cluster algebras of type $A_3$ and $B_3$, and another three mutation classes are associated with the non-crystallographic root system $H_3$. The exchange graphs of type $A_3$, $B_3$ and $H_3$ are called (generalised) associahedra. Associahedra of type $A$ and $B$ are well-known in the literature; the generalised associahedron of type $H_3$ appeared in the work of Fomin--Reading \cite{FR}. Exchange graphs for the remaining two mutation classes were presented in \cite{FT} and are also related to the root system $H_3$. Note that, in fact, the definition of mutation depends on a certain choice. In this article we show that any definition of mutation on the sphere that respects natural rules will result in one of the five exchange graphs mentioned above. 

More precisely, the definition of mutation depends on a notion of positivity/negativity, and there are several ways to define that. We introduce compatible choices of positivity for which rank $2$ subseeds behave in a way resembling the integer case. We prove that exchange graphs are independent of the choices made.

\renewcommand{\thethmIntro}{A}
\begin{thmIntro}[Theorem \ref{Thm:SphericalExchangeGraph}] Given a seed of finite type (i.e. realised in $\mathbb{S}^2$), all compatible choices of positive vectors yield the same exchange graph.  
\end{thmIntro}

In the affine case mutations can be modelled (by partial reflections of) triangles in $\mathbb{E}^2$. In this context a seed consists of a triangle with angles $r_1\pi$, $r_2\pi$, $r_3\pi$ with $r_1,r_2,r_3\in\mathbb{Q}$ together with a quiver whose vertices correspond to the sides of the triangle. No previous results about the exchange graphs are known to the authors.

\renewcommand{\thethmIntro}{B}
\begin{thmIntro}[Theorem \ref{Thm:ExchangeGraph}, Corollary \ref{Coro:QItoLattice}, Theorem \ref{thm:even}]
\label{Thm:IntroAffine} Given a seed of affine type (i.e. realised in $\mathbb{E}^2$), there exists a lattice $L$ such that the following conditions hold:
\begin{enumerate}
\item \label{Thm:QIIntro} The lattice acts on the exchange graph and the quotient of the group action is finite.  
\item The rank of the lattice is equal to $\operatorname{rk}_{\mathbb{Z}}(L)=\varphi(d)/2$ or $\varphi(d)$ where $\varphi$ is Euler's totient function and $d$ the least common denominator of $r_1,r_2,r_3$.
\end{enumerate}
\end{thmIntro}

To prove Theorem \ref{Thm:IntroAffine} we construct a line $b\subseteq \mathbb{E}^2$ (called the belt line) that intersects every triangle associated to an acyclic seed. The line is related to a billiard problem in Euclidean geometry. The lattice $L$ encodes the number theory of lengths of vectors parallel to $b$ that translate one seed into another seed. After constructing a number of translation vectors geometrically, we use number theory to prove that these vectors are linearly independent. More precisely, we work in the cyclotomic number field and apply a Dedekind determinant and a Verlinde formula. The proof of Theorem~\ref{Thm:IntroAffine} (\ref{Thm:QIIntro}) also uses the Schwarz--Milnor Lemma, see \cite{Sch} and \cite{M}. We observe that the geometry of the belt line $b$ also holds for real rank $3$ quivers of finite type (realised in $\mathbb{S}^2$) and even for mutation-infinite cases represented by partial reflections in $\mathbb{S}^2$, $\mathbb{E}^2$, or the hyperbolic plane~$\mathbb{H}^2$.

As an application of the above results, we compute the growth of the exchange graphs of mutation-finite real quivers. For a natural number $n$ we denote by $gr(n)$ the number of seeds that can be reached from the initial seed by at most $n$ mutations. We are interested in the growth of the function.

\begin{coro*}[Corollary \ref{Theo:GrowthRate}]
The exchange graph $\Gamma$ of a mutation-finite real quiver distinct from the Markov quiver has polynomial growth. In case when the quiver is affine, the polynomial growth rate is equal to the rank of the lattice $L$ from Theorem \ref{Thm:IntroAffine}.
\end{coro*}

For cluster algebras with integer exchange matrices we have a gap between polynomial growth of small degree and exponential growth. More precisely, unless a mutation-finite cluster algebra with an integer exchange has exponential growth, it either has linear growth (in the case of affine cluster algebras) or quadratic or cubic growth (for $3$ sporadic families of cluster algebras), see Fomin--Shapiro--Thurston \cite[Proposition 11.1]{FST08} and Felikson--Shapiro--Thomas--Tumarkin \cite[Theorem 1.1]{FSTT}. In contrast to the integer case, the polynomial growth rate, i.e. the rank of the lattice $L$, is larger than $1$ in most cases.

The article is organised as follows. In Section \ref{Section:GeoReal} we recall from \cite{FT} the notion of geometric realisations. In particular, we use geometry to define a notion of a seed for real exchange matrices. The section also contains a study of mutation-finite seeds that are geometrically realised on the sphere, i.e. quivers of finite type. In Section \ref{Section:GeoMut} we construct initial seeds for affine quivers. In this section we focus on the case where all angles are rational multiples of $\pi$ with an odd least common denominators; we explore the case of even least common denominators in Section \ref{Section:EvenDenominators}. In Section~\ref{Section:NumberTheory} we recall the definition of the cyclotomic field and study number-theoretic properties of sines and cosines of angles in triangles in the exchange graph. Section \ref{Sec:LinInd} is devoted to the proof of the linear independence of a set of translation vectors. In Section \ref{Section:MainThm} we determine the rank of $L$ and present our main result about the structure of the exchange graph, using quasi-isometries. In this section we also introduce growth rates and prove that the exchange graph has polynomial growth. Section~\ref{Section:EvenDenominators} contains an analysis of changes for even common denominators.

\section{Geometric realisations and their mutations}
\label{Section:GeoReal}

\subsection{Real exchange matrices and their mutations}

A crucial notion in the theory of cluster algebras is the mutation of a skew-symmetric matrix. The notion was introduced by Fomin--Zelevinsky \cite{FZ1} for integer matrices, but the same definition makes sense for real matrices. We fix an integer $r\geq 1$ called the \emph{rank}.

\begin{defn}[Matrix mutation \cite{FZ1}] The \emph{mutation} $\mu_k$, $k\in [1,r]$, of a real skew-symmetric $r\times r$ matrix $B=(b_{ij})$ is the skew-symmetric $r\times r$ matrix $\mu_k(B)=B'=(b'_{ij})$ with entries
\begin{align*}
b_{ij}'=\begin{cases}
-b_{ij}&\textrm{if }k\in\{i,j\};\\
b_{ij}+(b_{ik}\lvert b_{kj}\rvert+\lvert b_{ik}\rvert b_{kj})/2&\textrm{otherwise}.
\end{cases}
\end{align*}
\end{defn}

Such a matrix $B$ is called \emph{exchange matrix}. Note that real skew-symmetric $r\times r$ matrices are in correspondence with quivers on $r$ vertices without loops and $2$-cycles (with real-valued edge weights). We say that a skew-symmetric matrix $B$ is \emph{acyclic} if its quiver does not contain oriented cycles. 

Notice that matrix mutation is an involution, that is, $(\mu_k\circ \mu_k)(B)=B$ for all $B$ and all $k$. The following definition is a straightforward generalisation of usual terminology for exchange matrices with integer entries to real entries.

\begin{defn}[Matrix mutation classes] Let $B\in\operatorname{Mat}_{r\times r}(\mathbb{R})$ be an exchange matrix.
\begin{enumerate}
\item A skew-symmetric matrix $B'\in\operatorname{Mat}_{r\times r}(\mathbb{R})$ is called \emph{mutation-equivalent} to $B$ if there exists a sequence $(k_1,\ldots,k_s)$ of indices such that $B'=(\mu_{k_s}\circ\ldots\circ\mu_{k_1})(B)$.
\item The \emph{mutation class} of $B$ is the set of all $B'$ that are mutation-equivalent to $B$.
\item We say that $B$ is \emph{mutation-finite} if the mutation class of $B$ is finite.
\item We say that $B$ is \emph{mutation-acyclic} if it is mutation-equivalent to an acyclic exchange matrix. Otherwise we say it is \emph{mutation-cyclic}.
\end{enumerate}
\end{defn}

In this article we are interested in exchange matrices of rank $3$, that is, we only consider $3\times 3$ exchange matrices in the rest of the paper. Furthermore, our article is devoted to exchange matrices that are mutation-finite. The next proposition summarises \cite[Lemma 6.5, Corollary 6.3, Lemma 6.6, Lemma 6.7/Remark 6.8, Lemma 6.9]{FT}. In this proposition we use a notion introduced by Beineke--Br\"ustle--Hille \cite{BBH} known as the \emph{Markov constant}. For a $3\times 3$ exchange matrix $B$ we put
\begin{align*}
C(B)=\begin{cases}
\lvert b_{12}\rvert^2+\lvert b_{23}\rvert^2+\lvert b_{13}\rvert^2+\lvert b_{12}b_{23}b_{13}\rvert&\textrm{if }B\textrm{ is acyclic};\\
\lvert b_{12}\rvert^2+\lvert b_{23}\rvert^2+\lvert b_{13}\rvert^2-\lvert b_{12}b_{23}b_{13}\rvert&\textrm{if }B\textrm{ is cyclic}.
\end{cases}
\end{align*}

\begin{prop}[\cite{FT}]
\label{Prop:Classification}
Suppose that $B=(b_{ij})\in\operatorname{Mat}_{3\times 3}(\mathbb{R})$ is an exchange matrix. 
\begin{enumerate}
\item \label{Prop:ClassificationRational}
Assume that $B$ is mutation-finite. For any two distinct indices $i,j\in \{1,2,3\}$ there are natural numbers $k\geq 0$ and $l\geq 1$ such that $b_{ij}=2\cos(\pi k/l)$.
\item \label{Prop:ClassificationMarkov}
If $B$ is mutation-finite and mutation-cyclic, then $B$ is mutation-equivalent to the matrix
\begin{align*}
\left(\begin{smallmatrix}0&2&-2\\-2&0&2\\2&-2&0\end{smallmatrix}\right)
\end{align*}
which is known as the \emph{Markov quiver}. The mutation class of this matrix contains only one element (up to reordering of indices).
\item \label{Prop:ClassificationHyperbolic}
If $B$ is mutation-acyclic and $C(B)>4$, then $B$ is not mutation-finite.
\item \label{Prop:ClassificationSpherical}
Suppose that $B$ is mutation-acyclic and $C(B)<4$. Then $B$ is mutation-finite if and only if it is mutation-equivalent to the matrix
\begin{align*}
\left(\begin{smallmatrix}
0&2\cos(\pi t_1)&0\\
-2\cos(\pi t_1)&0&2\cos(\pi t_2)\\
0&-2\cos(\pi t_2)&0
\end{smallmatrix}\right)
\end{align*}
where $(t_1,t_2)$ is one of the pairs
\begin{align*}
\left(\frac{1}{3},\frac{1}{3}\right),\,\left(\frac{1}{3},\frac{1}{4}\right),\,\left(\frac{1}{3},\frac{1}{5}\right),\,\left(\frac{1}{3},\frac{2}{5}\right),\,\left(\frac{1}{5},\frac{2}{5}\right).
\end{align*}
\item \label{Prop:ClassificationEuclidean}
Suppose that $B$ is mutation-acyclic and $C(B)=4$. Then $B$ is mutation-finite if and only it is mutation-equivalent to 
\begin{align}
\label{Eq:InfiniteRegions}
\left(\begin{smallmatrix}
0&2&-2\cos(\pi t)\\
-2&0&2\cos(\pi t)\\
2\cos(\pi t)&-2\cos(\pi t)&0
\end{smallmatrix}\right).
\end{align}
where $t=\pi/d$ for some natural number $d$.
\end{enumerate}
\end{prop} 

The entirety of the statements in Proposition \ref{Prop:Classification} yields a classification of $3\times 3$ exchange matrices of finite mutation type in terms of the Markov constant. In particular, finite mutation type is exhausted by the Markov quiver (case \ref{Prop:ClassificationMarkov}) together with quivers of finite and affine type (cases \ref{Prop:ClassificationSpherical} and \ref{Prop:ClassificationEuclidean}). The various cases admit different geometric models. We will study quivers of finite type (case \ref{Prop:ClassificationSpherical}) in Section \ref{Section:Spherical} via spherical geometry. Quivers of affine type (case \ref{Prop:ClassificationEuclidean}) form the main part of this paper. They admit a model in Euclidean geometry, which is studied in further sections.

\subsection{Geometric realisation of seeds and their mutations}

For a skew-symmetric matrix with integer entries we can construct a cluster algebra by introducing clusters and cluster variables. For a skew-symmetric matrix with non-integer entries it is not easy to define notions of cluster variables and clusters. Building on work of Seven \cite{S} and Barot--Gei{\ss}--Zelevinsky \cite{BGZ}, Felikson--Tumarkin \cite{FT} consider a notion of a \emph{Y-seeds} and realise mutations of Y-seeds by partial reflections. In the rest of the section we give a short overview of the construction.

\begin{rem} We point out that in different contexts we can construct different geometric realisations. In our article we restrict ourselves to the construction below.
\end{rem}

We fix a $3$-dimensional \emph{quadratic space} $V$, that is, a $3$-dimensional real vector space $V$ together with a symmetric bilinear form $(\cdot,\cdot)\,\colon\, V\times V\to \R$. The following definition is motivated by a proposition of Barot--Gei{\ss}--Zelevinsky \cite[Proposition 3.2]{BGZ}.

\begin{defn}[Geometric realisations]
\label{Def:GeoReal}
Suppose that $B$ is a real skew-symmetric $3\times 3$ matrix. An ordered basis $\mathbf{v}=(v_1,v_2,v_3)$ of $V$ is called a \emph{geometric realisation} of $B$ if the following conditions hold:
\begin{enumerate}
\item For every $i\in \{1,2,3\}$ we have $(v_i,v_i)=2$.
\item For two distinct $i,j\in \{1,2,3\}$ we have $\lvert(v_i,v_j)\rvert=\lvert b_{ij}\rvert$.
\item Assume that $(v_i,v_j)\neq 0$ for all $1\leq i<j\leq 3$. Then
\begin{align*}
\# \{\,(i,j)\in\mathbb{N}^2\mid 1\leq i< j\leq 3\textrm{ and } (v_i,v_j)>0\,\} \equiv 
\begin{cases}
0 \textrm{ mod }2&\textrm{if }B\textrm{ is acyclic};\\
1 \textrm{ mod }2&\textrm{otherwise}.
\end{cases}
\end{align*} 
\end{enumerate}
\end{defn}

We identify two geometric realisations that are obtained from each other by a permutation of their entries.

\begin{rem} Let $B=(b_{ij})$ be a skew-symmetric matrix. A \emph{quasi-Cartan companion} of $B$ is a matrix $A=(a_{ij})$ such that $a_{ii}=2$ for all $i$ and $\lvert a_{ij}\rvert=\lvert b_{ij}\rvert$. The \emph{Gram matrix} of $\mathbf{v}\in V^3$ is the matrix $((v_i,v_j)_{ij})$. The conditions (1) and (2) in Definition \ref{Def:GeoReal} imply that the Gram matrix of a geometric realisation  of $B$ is a quasi-Cartan companion of $B$. Condition (3) asserts that the Gram matrix is \emph{compatible} with $B$ in the sense of Barot--Gei{\ss}--Zelevinsky \cite[Definition 3.1]{BGZ}.
\end{rem}

\begin{rem} \begin{itemize}
\item[(a)] 
For better readability we have stated Definition \ref{Def:GeoReal} only for $3\times 3$ exchange matrices. It is possible to define geometric realisations in all ranks.
To define geometric realisations in higher rank we would replace condition (3) with a cycle condition, see the third condition Seven lists in \cite[Theorem 1.4]{S} (with integer entries replaced with real entries).
\item[(b)] A realisation of a $2\times 2$ exchange matrix is a linearly independent pair $(v_1,v_2)$ inside a $2$-dimensional vector space $V$ satisfying conditions (1) and (2) in Definition \ref{Def:GeoReal} for all $i,j\in \{1,2\}$. Given a seed of rank $3$, we can construct a seed of rank $2$ by removing a vector and deleting the corresponding rows and columns of the matrix. This seed is called a \emph{subseed}. 
\end{itemize}
\end{rem}

\begin{defn}[Seeds] A \emph{seed} of rank $r$ is a pair $(\mathbf{v},B)$ where $B$ is a real skew-symmetric $r\times r$ matrix and $\mathbf{v}$ a geometric realisation of $B$. 
\end{defn}

As in Lie theory, where the roots of a root system are divided into positive and negative roots, we declare certain vectors $v\in V\backslash\{0\}$ to be \emph{positive}, and certain vectors $v\in V\backslash\{0\}$ to be \emph{negative}. %This condition will imply that there exists a vector $u\in V$ that does not belong to the radical of $(\cdot,\cdot)\colon V\times V\to\mathbb{R}$ such that $v\in V$ is positive if $(u,v)>0$, and negative if $(u,v)<0$.
We denote the set of positive elements by $V^{+}\subseteq V$ and the set of negative elements by $V^{-}\subseteq V$. We will give a precise definition of positive and negative vectors later in the text, see the Compatibility Condition \ref{Condition:Pos1} in Section \ref{Subsection:GeoRealRank3}. As a rough idea, we demand that the definition be made to recover the classical mutation of seeds when restricting to the case of integer exchange matrices. Moreover, the condition implies that $v\in V^{+}$ if and only if $-v\in V^{-}$. %The vector $u$ is called the \emph{reference vector}.

\begin{defn}[Mutations of seeds]
\label{Def:SeedMutation}
Suppose that $(\mathbf{v},B)$ is a seed of rank $r$. The \emph{mutation} $\mu_k$, $k\in[1,r]$, of $(\mathbf{v},B)$ is $\mu_k(\mathbf{v},B)=(\mu_k(\mathbf{v}),\mu_k(B))$ where $\mu_k(\mathbf{v})=(v_1',\ldots,v_r')$ is defined in the following way. If $v_k\in V^{+}$, then we put
\begin{align*}
v_i'=\begin{cases}
-v_i&\textrm{if }i=k;\\
v_i&\textrm{if }i\neq k\textrm{ and }b_{ik}>0;\\
v_i-(v_i,v_k)v_k&\textrm{if }i\neq k\textrm{ and }b_{ik}<0;
\end{cases}
\end{align*}
otherwise we put
\begin{align*}
v_i'=\begin{cases}
-v_i&\textrm{if }i=k;\\
v_i&\textrm{if }i\neq k\textrm{ and }b_{ik}<0;\\
v_i-(v_i,v_k)v_k&\textrm{if }i\neq k\textrm{ and }b_{ik}>0.
\end{cases}
\end{align*}
\end{defn}

It is easy to check that the mutation $\mu_k$ is well-defined, i.e. $\mu_k(\mathbf{v},B)$ is again a seed. A related statement for exchange matrices with integer entries is due to Barot--Gei{\ss}--Zelevinsky \cite[Proposition 3.2]{BGZ}.

If $v_k$ is positive, then $-v_k$ is negative. We can conclude that $\mu_k$ is an involution, that is, $(\mu_k\circ\mu_k)(\mathbf{v},B)=(\mathbf{v},B)$ for all $k$ and $(\mathbf{v},B)$. The mutation $\mu_k$ is called \emph{positive} if $v_k\in V^{+}$  and it is called \emph{negative} otherwise.

\begin{defn}[Exchange graphs]
The \emph{exchange graph} of a seed $(\mathbf{v},B)$ contains as vertices all seeds $(\mathbf{v}',B')$ obtained from $(\mathbf{v},B)$ by sequences of mutations. Two vertices $(\mathbf{v}',B')$ and $(\mathbf{v}'',B'')$ are connected by an edge if and only if they are related by a single mutation.   
\end{defn}

Two seeds $(\mathbf{v},B)$ and $(\mathbf{v}',B')$ are called \emph{mutation-equivalent} if there exists a sequence of indices $(k_1,\ldots,k_s)$ such that $(\mathbf{v}',B')=(\mu_{k_s}\circ\ldots\circ\mu_{k_1})(\mathbf{v},B)$. By definition, if $(\mathbf{v},B)$ and $(\mathbf{v}',B')$ are mutation-equivalent, then their exchange graphs coincide.

\subsection{Seed mutations in rank 2} 
\label{Section:Rank2}
This subsection is devoted to the mutation of seeds of rank $2$. The aim of the subsection is to study the adequacy of distinct notions of positive and negative vectors.

Here a seed is given by a skew-symmetric $2\times 2$ matrix $B$ together with a realisation $\mathbf{v}=(v_1,v_2)$ in a $2$-dimensional real vector space $V$, which we may identify with the vector space $V=\mathbb{R}^2$. After projectivisation, we may identify the space $V/\mathbb{R}^{+}$ with the circle $\mathbb{S}^1$.

The matrix $B$ is determined by the entry $b_{12}$ thanks to skew-symmetry. In what follows, in regard to Proposition \ref{Prop:Classification} (\ref{Prop:ClassificationRational}), we suppose that $b_{12}=2\cos(\alpha)$ is a rational multiple of $\pi$. To be concrete, let us fix the angle as follows.

\begin{defn}[Fundamental angle] 
\label{Def:AngleRankTwo}
For a rank $2$ seed $(\mathbf{v},B)$ with $b_{12}=2\cos(\alpha)$, we put $\alpha=\frac{a}{b}\pi$ where $a,b\geq 1$ are natural coprime numbers, and call $\alpha$ the \emph{fundamental angle}.
\end{defn} 

Notice that $\lvert b_{12}\rvert \leq 2$. We combine this inequality with condition (2) in Definition \ref{Def:GeoReal}. Sylvester's criterion implies that the bilinear form $(\cdot,\cdot)\,\colon\, V\times V\to\mathbb{R}$ is positive definite. Without loss of generality we may assume that $V=\mathbb{E}^2$ is the Euclidean plane equipped with the standard scalar product $\langle\cdot,\cdot\rangle$.

For $v\in V\backslash\{0\}$ we put $\Pi_v=\{w\in\mathbb{E}^2\mid \langle w,v\rangle<0\}$ and $l_v=\{w\in\mathbb{E}^2\mid \langle w,v\rangle=0\}$. Notice that $\Pi_v$ is an (open) half-plane bounded by the line $l_v$.

The map $v\mapsto \Pi_v$ induces a bijection between the set of vectors $v\in\mathbb{E}^2$ with $\langle v,v\rangle=2$ and the set of (open) half-planes in $\mathbb{E}^2$ whose boundary line passes through the origin. Because of that we identify the pair $\mathbf{v}=(v_1,v_2)$ with the intersection of the half-planes $\Pi_{v_1}\cap\Pi_{v_2}$. This intersection is a sector given by an angle of size $\alpha$ centered at the origin. The size $\alpha$ determines $\vert b_{12}\rvert$. To obtain a geometric model of the seed $(\mathbf{v},B)$ we have to encode the sign of $b_{12}$, which we do by an arrow. More precisely, we orient (the sides of the sector corresponding to) the vectors $v_1 \to v_2$ if $b_{12}>0$, $v_2 \to v_1$ if $b_{12}<0$, and do not connect them if $b_{12}=0$.

The mutation process is now realised by partial reflections, which means the following. Suppose that the seed $(\mathbf{v},B)$ is realised by a sector with an angle whose sides $s_1$ and $s_2$ correspond to $v_1$ and $v_2$ (together with an arrow between $v_1$ and $v_2$). Then the mutation of the seed at $v_2$, compare Definition~ \ref{Def:SeedMutation}, either replaces $s_1$ with the reflection $s_1$ across $s_2$, or leaves it unchanged (depending on the location of $v_2$ and orientation of the arrow).

The following definition is an obvious attempt to define positive and negative vectors. Under some assumptions it recovers the classical notion of positivity in Lie theory when we think if the vectors $v_1$ and $v_2$ as positive simple roots in a finite type root system. 

\begin{ans}[Positivity via reference point]
\label{Attempt:Positivity}
Fix $u\in \mathbb{E}^2\backslash\{0\}$. We  declare $v$ to belong to the set $V^{+}$ if $\langle u,v\rangle>0$, and declare $v$ to belong to the set $V^{-}$ if $\langle u,v\rangle <0$.
\end{ans}

\begin{ex}[Short period]
\label{Ex:A2}
We consider a geometric realisation $\mathbf{v}=(v_1,v_2)$ of the exchange matrix 
\begin{align*}
B=\left(\begin{matrix}0&2\cos(\pi/3)\\-2\cos(\pi/3)&0
\end{matrix}\right)=\left(\begin{matrix}0&1\\-1&0
\end{matrix}\right).
\end{align*}
We draw the seed $(\mathbf{v},B)$ in the top left of corner of Figures \ref{Figure:ShortRankTwo} and \ref{Figure:LongRankTwo}. Here we visualise the sector $\Pi_{v_1}\cap \Pi_{v_2}$ in blue. The vectors $v_1$ and $v_2$ itself are not drawn explicitly, but we can recover them as the two vectors of length $\sqrt{2}$ which are orthogonal to the sides of the sector pointing outside. The reference point $u$ is coloured red. Definition~\ref{Attempt:Positivity} asserts that $v_i$, $i\in\{1,2\}$, is positive if and only if the line $l_{v_i}$ divides the plane into two half-planes such that $u$ lies in the same half-plane than the sector.   

Figure \ref{Figure:ShortRankTwo} shows the mutation process. Here we mutate at sources when passing through the picture in clockwise direction. During the mutation process sectors with angles $\pi/3$ and $2\pi/3$ occur. Sometimes it is convenient not to draw sectors with obtuse angles. Instead we introduce signs. Here a minus sign attached to a half-plane $\Pi_{v_i}$, $i\in\{1,2\}$, indicates that we draw $\{w\in\mathbb{E}^2\mid \langle w,v_i\rangle>0\}$ instead of $\Pi_{v_i}=\{w\in\mathbb{E}^2\mid \langle w,v_i\rangle<0\}$. We obtain a periodic sequence of seeds of period $5$. Figure~\ref{Figure:SeedMutation} shows the $5$-periodic mutation process for the same seed and the same reference point as Figure~\ref{Figure:ShortRankTwo}, but here sectors are drawn entirely without signs. The figure also shows the exchange graph of the seed, a pentagon, which is an associahedron of type $A_2$, compare Fomin--Reading~\cite[Figure~4.4]{FR}.  
\end{ex}

\begin{ex}[Long period]
\label{Ex:A2Long}
Figure \ref{Figure:LongRankTwo} shows the mutation process for the same seed as in Example \ref{Ex:A2} with a different reference point. We obtain a periodic sequence of seeds of period $7$. Notice that the matrix $B$ defines a cluster algebra of type $A_2$, which admits $5$ seeds. Hence, the short periods correspond to the periods that occur in cluster algebras (attached to integer exchange matrices), whereas larger periods yield larger exchange graphs. We draw the locus of the reference points $u$ which yield $5$-periodic sequences of seeds in Figure \ref{Figure:GoodRegions}.  
\end{ex}

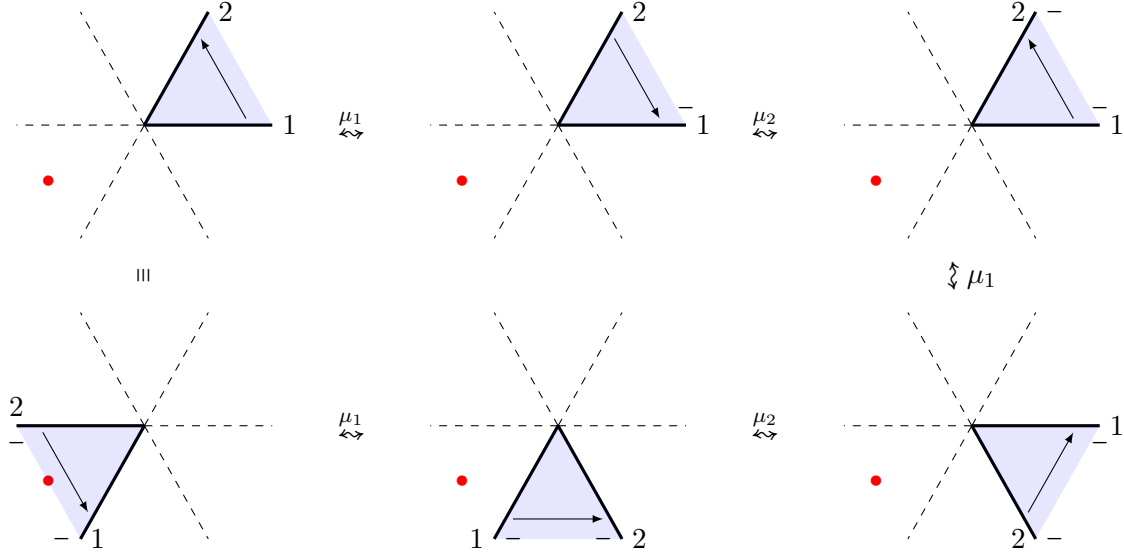
\begin{figure}
\begin{center}
\begin{tikzpicture}

%%% Abbreviations
\renewcommand{\R}{1.7cm} %% length of radius 
\newcommand{\x}{0.3cm} %% distance to label
\newcommand{\s}{5.5cm} %% shift between figures
\newcommand{\ys}{4cm} %% shift between figures
\newcommand{\h}{0.886} %% height of triangles

%% Seed 1
\path[draw,-,very thick] (0,0) to (\R,0);
\path[draw,-,very thick] (0,0) to (0.5*\R,\h*\R);
\path[draw,-,dashed] (0,0) to (-0.5*\R,\h*\R);
\path[draw,-,dashed] (0,0) to (-\R,0);
\path[draw,-,dashed] (0,0) to (-0.5*\R,-\h*\R);
\path[draw,-,dashed] (0,0) to (0.5*\R,-\h*\R);

\node[right] at (\R,0) {$1$};
\node[right] at (0.5*\R,\h*\R) {$2$};

\node[red] at (-0.75*\R,-0.5*\h*\R) {$\bullet$};

\fill[blue,fill opacity=0.1] (0,0) -- (\R,0) -- (0.5*\R,\h*\R); 

%\node at (\R+\x,0) {$1$};
%\node at (0.5*\R+0.5*\x,\h*\R+\h*\x) {$2$};

\draw[-latex,shorten >=0.1cm,shorten <=0.1cm] (\R-\x,0) -- (0.5*\R-0.5*\x,\h*\R-\h*\x);

%%%% Seed 2
\path[draw,-,very thick] (\s,0) to (\s+\R,0);
\path[draw,-,very thick] (\s,0) to (\s+0.5*\R,\h*\R);
\path[draw,-,dashed] (\s,0) to (\s-0.5*\R,\h*\R);
\path[draw,-,dashed] (\s,0) to (\s-\R,0);
\path[draw,-,dashed] (\s,0) to (\s-0.5*\R,-\h*\R);
\path[draw,-,dashed] (\s,0) to (\s+0.5*\R,-\h*\R);

\node[red] at (-0.75*\R+\s,-0.5*\h*\R) {$\bullet$};

\node[right] at (\s+\R,0) {$1$};
\node[right] at (\s+0.5*\R,\h*\R) {$2$};

\node[above] at (\s+\R,0) {$-$}; 

\fill[blue,fill opacity=0.1] (\s,0) -- (\s+\R,0) -- (\s+0.5*\R,\h*\R); 

%\node at (\s+\R+\x,0) {$1$};
%\node at (\s+0.5*\R+0.5*\x,\h*\R+\h*\x) {$2$};

\draw[-latex,shorten >=0.1cm,shorten <=0.1cm] (\s+0.5*\R-0.5*\x,\h*\R-\h*\x) -- (\s+\R-\x,0);

%%%% Seed 3
\path[draw,-,very thick] (2*\s,0) to (2*\s+\R,0);
\path[draw,-,very thick] (2*\s,0) to (2*\s+0.5*\R,\h*\R);
\path[draw,-,dashed] (2*\s,0) to (2*\s-0.5*\R,\h*\R);
\path[draw,-,dashed] (2*\s,0) to (2*\s-\R,0);
\path[draw,-,dashed] (2*\s,0) to (2*\s-0.5*\R,-\h*\R);
\path[draw,-,dashed] (2*\s,0) to (2*\s+0.5*\R,-\h*\R);

\node[red] at (-0.75*\R+2*\s,-0.5*\h*\R) {$\bullet$};

\node[above] at (2*\s+\R,0) {$-$}; 
\node[right] at (2*\s+0.5*\R,\h*\R) {$-$}; 

\node[right] at (2*\s+\R,0) {$1$};
\node[left] at (2*\s+0.5*\R,\h*\R) {$2$};

\fill[blue,fill opacity=0.1] (2*\s,0) -- (2*\s+\R,0) -- (2*\s+0.5*\R,\h*\R); 

%\node at (\s+\R+\x,0) {$1$};
%\node at (\s+0.5*\R+0.5*\x,\h*\R+\h*\x) {$2$};

\draw[-latex,shorten >=0.1cm,shorten <=0.1cm] (2*\s+\R-\x,0) -- (2*\s+0.5*\R-0.5*\x,\h*\R-\h*\x);

%%%% Seed 4
\path[draw,-,very thick] (2*\s,-\ys) to (2*\s+\R,-\ys);
\path[draw,-,dashed] (2*\s,-\ys) to (2*\s+0.5*\R,\h*\R-\ys);
\path[draw,-,dashed] (2*\s,-\ys) to (2*\s-0.5*\R,\h*\R-\ys);
\path[draw,-,dashed] (2*\s,-\ys) to (2*\s-\R,-\ys);
\path[draw,-,dashed] (2*\s,-\ys) to (2*\s-0.5*\R,-\h*\R-\ys);
\path[draw,-,very thick] (2*\s,-\ys) to (2*\s+0.5*\R,-\h*\R-\ys);

\node[red] at (-0.75*\R+2*\s,-0.5*\h*\R-\ys) {$\bullet$};

\node[below] at (2*\s+\R,-\ys) {$-$}; 
\node[right] at (2*\s+0.5*\R,-\h*\R-\ys) {$-$}; 

\node[right] at (2*\s+\R,-\ys) {$1$};
\node[left] at (2*\s+0.5*\R,-\h*\R-\ys) {$2$};

\fill[blue,fill opacity=0.1] (2*\s,-\ys) -- (2*\s+\R,-\ys) -- (2*\s+0.5*\R,-\h*\R-\ys); 

%\node at (\s+\R+\x,0) {$1$};
%\node at (\s+0.5*\R+0.5*\x,\h*\R+\h*\x) {$2$};

\draw[-latex,shorten >=0.1cm,shorten <=0.1cm] (2*\s+0.5*\R-0.5*\x,-\h*\R+\x*\h-\ys) -- (2*\s+\R-\x,-\ys);

%%%% Seed 5
\path[draw,-,dashed] (\s,-\ys) to (\s+\R,-\ys);
\path[draw,-,dashed] (\s,-\ys) to (\s+0.5*\R,\h*\R-\ys);
\path[draw,-,dashed] (\s,-\ys) to (\s-0.5*\R,\h*\R-\ys);
\path[draw,-,dashed] (\s,-\ys) to (\s-\R,-\ys);
\path[draw,-,very thick] (\s,-\ys) to (\s-0.5*\R,-\h*\R-\ys);
\path[draw,-,very thick] (\s,-\ys) to (\s+0.5*\R,-\h*\R-\ys);

\node[red] at (-0.75*\R+\s,-0.5*\h*\R-\ys) {$\bullet$};

\node[left] at (\s-0.5*\R,-\h*\R-\ys) {$1$};
\node[right] at (\s+0.5*\R,-\h*\R-\ys) {$2$};

\node[right] at (\s-0.5*\R,-\h*\R-\ys) {$-$}; 
\node[left] at (\s+0.5*\R,-\h*\R-\ys) {$-$}; 

\fill[blue,fill opacity=0.1] (\s,-\ys) -- (\s+0.5*\R,-\h*\R-\ys) -- (\s-0.5*\R,-\h*\R-\ys); 

%\node at (\s+\R+\x,0) {$1$};
%\node at (\s+0.5*\R+0.5*\x,\h*\R+\h*\x) {$2$};

\draw[-latex,shorten >=0.1cm,shorten <=0.1cm] (\s-0.5*\R+0.5*\x,-\h*\R+\x*\h-\ys) -- (\s+0.5*\R-0.5*\x,-\h*\R+\x*\h-\ys);

%%%% Seed 6
\path[draw,-,dashed] (0,-\ys) to (\R,-\ys);
\path[draw,-,dashed] (0,-\ys) to (0.5*\R,\h*\R-\ys);
\path[draw,-,dashed] (0,-\ys) to (-0.5*\R,\h*\R-\ys);
\path[draw,-,very thick] (0,-\ys) to (-\R,-\ys);
\path[draw,-,very thick] (0,-\ys) to (-0.5*\R,-\h*\R-\ys);
\path[draw,-,dashed] (0,-\ys) to (0.5*\R,-\h*\R-\ys);

\node[red] at (-0.75*\R,-0.5*\h*\R-\ys) {$\bullet$};

\node[right] at (-0.5*\R,-\h*\R-\ys) {$1$};
\node[above] at (-\R,-\ys) {$2$};

\node[below] at (-\R,-\ys) {$-$}; 
\node[left] at (-0.5*\R,-\h*\R-\ys) {$-$}; 

\fill[blue,fill opacity=0.1] (0,-\ys) -- (-\R,-\ys) -- (-0.5*\R,-\h*\R-\ys); 

%\node at (\s+\R+\x,0) {$1$};
%\node at (\s+0.5*\R+0.5*\x,\h*\R+\h*\x) {$2$};

\draw[-latex,shorten >=0.1cm,shorten <=0.1cm] (-\R+\x,-\ys) -- (-0.5*\R+0.5*\x,-\h*\R+\x*\h-\ys);

%%% Mutations
\node at (0.5*\s,0) {$\overset{\mu_1}{\leftrightsquigarrow}$};
\node at (1.5*\s,0) {$\overset{\mu_2}{\leftrightsquigarrow}$};
\node at (2*\s,-0.5*\ys) {$\squigarrowupdown \mu_1$};
\node at (1.5*\s,-\ys) {$\overset{\mu_2}{\leftrightsquigarrow}$};
\node at (0.5*\s,-\ys) {$\overset{\mu_1}{\leftrightsquigarrow}$};
\node at (0,-0.5*\ys) {\rotatebox[origin=c]{90}{$\equiv$}};

\end{tikzpicture}
\caption{Geometric description of seeds and mutations in rank 2}
\label{Figure:ShortRankTwo}
\end{center}
\end{figure}

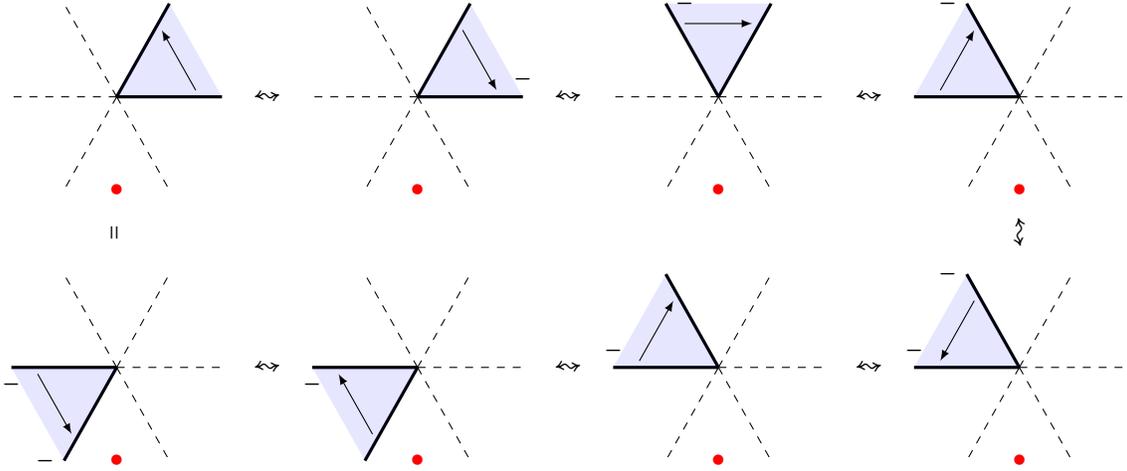
\begin{figure}
\begin{center}
\begin{tikzpicture}

%%% Abbreviations
\renewcommand{\R}{1.4cm} %% length of radius 
\newcommand{\x}{0.3cm} %% distance to label
\newcommand{\s}{4cm} %% shift between figures
\newcommand{\ys}{3.6cm} %% shift between figures
\newcommand{\h}{0.886} %% height of triangles

%% Seed 1
\path[draw,-,very thick] (0,0) to (\R,0);
\path[draw,-,very thick] (0,0) to (0.5*\R,\h*\R);
\path[draw,-,dashed] (0,0) to (-0.5*\R,\h*\R);
\path[draw,-,dashed] (0,0) to (-\R,0);
\path[draw,-,dashed] (0,0) to (-0.5*\R,-\h*\R);
\path[draw,-,dashed] (0,0) to (0.5*\R,-\h*\R);

\node[red] at (0,-\h*\R) {$\bullet$};

\fill[blue,fill opacity=0.1] (0,0) -- (\R,0) -- (0.5*\R,\h*\R); 

%\node at (\R+\x,0) {$1$};
%\node at (0.5*\R+0.5*\x,\h*\R+\h*\x) {$2$};

\draw[-latex,shorten >=0.1cm,shorten <=0.1cm] (\R-\x,0) -- (0.5*\R-0.5*\x,\h*\R-\h*\x);

%%%% Seed 2
\path[draw,-,very thick] (\s,0) to (\s+\R,0);
\path[draw,-,very thick] (\s,0) to (\s+0.5*\R,\h*\R);
\path[draw,-,dashed] (\s,0) to (\s-0.5*\R,\h*\R);
\path[draw,-,dashed] (\s,0) to (\s-\R,0);
\path[draw,-,dashed] (\s,0) to (\s-0.5*\R,-\h*\R);
\path[draw,-,dashed] (\s,0) to (\s+0.5*\R,-\h*\R);

\node[red] at (\s,-\h*\R) {$\bullet$};

\node[above] at (\s+\R,0) {$-$}; 

\fill[blue,fill opacity=0.1] (\s,0) -- (\s+\R,0) -- (\s+0.5*\R,\h*\R); 

%\node at (\s+\R+\x,0) {$1$};
%\node at (\s+0.5*\R+0.5*\x,\h*\R+\h*\x) {$2$};

\draw[-latex,shorten >=0.1cm,shorten <=0.1cm] (\s+0.5*\R-0.5*\x,\h*\R-\h*\x) -- (\s+\R-\x,0);

%%%% Seed 3
\path[draw,-,dashed] (2*\s,0) to (2*\s+\R,0);
\path[draw,-,very thick] (2*\s,0) to (2*\s+0.5*\R,\h*\R);
\path[draw,-,very thick] (2*\s,0) to (2*\s-0.5*\R,\h*\R);
\path[draw,-,dashed] (2*\s,0) to (2*\s-\R,0);
\path[draw,-,dashed] (2*\s,0) to (2*\s-0.5*\R,-\h*\R);
\path[draw,-,dashed] (2*\s,0) to (2*\s+0.5*\R,-\h*\R);

\node[red] at (2*\s,-\h*\R) {$\bullet$};

\node[right] at (2*\s-0.5*\R,\h*\R) {$-$}; 

\fill[blue,fill opacity=0.1] (2*\s,0) -- (2*\s+0.5*\R,\h*\R) -- (2*\s-0.5*\R,\h*\R); 

%\node at (\s+\R+\x,0) {$1$};
%\node at (\s+0.5*\R+0.5*\x,\h*\R+\h*\x) {$2$};

\draw[-latex,shorten >=0.1cm,shorten <=0.1cm] (2*\s-0.5*\R+0.5*\x,\h*\R-\h*\x) -- (2*\s+0.5*\R-0.5*\x,\h*\R-\h*\x);

%%%% Seed 4
\path[draw,-,dashed] (3*\s,0) to (3*\s+\R,0);
\path[draw,-,dashed] (3*\s,0) to (3*\s+0.5*\R,\h*\R);
\path[draw,-,very thick] (3*\s,0) to (3*\s-0.5*\R,\h*\R);
\path[draw,-,very thick] (3*\s,0) to (3*\s-\R,0);
\path[draw,-,dashed] (3*\s,0) to (3*\s-0.5*\R,-\h*\R);
\path[draw,-,dashed] (3*\s,0) to (3*\s+0.5*\R,-\h*\R);

\node[red] at (3*\s,-\h*\R) {$\bullet$};

\node[left] at (3*\s-0.5*\R,\h*\R) {$-$}; 

\fill[blue,fill opacity=0.1] (3*\s,0) -- (3*\s-\R,0) -- (3*\s-0.5*\R,\h*\R); 

%\node at (\s+\R+\x,0) {$1$};
%\node at (\s+0.5*\R+0.5*\x,\h*\R+\h*\x) {$2$};

\draw[-latex,shorten >=0.1cm,shorten <=0.1cm] (3*\s-\R+\x,0) -- (3*\s-0.5*\R+0.5*\x,\h*\R-\h*\x);

%%%% Seed 5
\path[draw,-,dashed] (3*\s,-\ys) to (3*\s+\R,-\ys);
\path[draw,-,dashed] (3*\s,-\ys) to (3*\s+0.5*\R,\h*\R-\ys);
\path[draw,-,very thick] (3*\s,-\ys) to (3*\s-0.5*\R,\h*\R-\ys);
\path[draw,-,very thick] (3*\s,-\ys) to (3*\s-\R,-\ys);
\path[draw,-,dashed] (3*\s,-\ys) to (3*\s-0.5*\R,-\h*\R-\ys);
\path[draw,-,dashed] (3*\s,-\ys) to (3*\s+0.5*\R,-\h*\R-\ys);

\node[red] at (3*\s,-\h*\R-\ys) {$\bullet$};

\node[left] at (3*\s-0.5*\R,\h*\R-\ys) {$-$}; 
\node[above] at (3*\s-\R,-\ys) {$-$}; 

\fill[blue,fill opacity=0.1] (3*\s,-\ys) -- (3*\s-\R,-\ys) -- (3*\s-0.5*\R,\h*\R-\ys); 

%\node at (\s+\R+\x,0) {$1$};
%\node at (\s+0.5*\R+0.5*\x,\h*\R+\h*\x) {$2$};

\draw[-latex,shorten >=0.1cm,shorten <=0.1cm] (3*\s-0.5*\R+0.5*\x,\h*\R-\h*\x-\ys) -- (3*\s-\R+\x,-\ys);

%%%% Seed 6
\path[draw,-,dashed] (2*\s,-\ys) to (2*\s+\R,-\ys);
\path[draw,-,dashed] (2*\s,-\ys) to (2*\s+0.5*\R,\h*\R-\ys);
\path[draw,-,very thick] (2*\s,-\ys) to (2*\s-0.5*\R,\h*\R-\ys);
\path[draw,-,very thick] (2*\s,-\ys) to (2*\s-\R,-\ys);
\path[draw,-,dashed] (2*\s,-\ys) to (2*\s-0.5*\R,-\h*\R-\ys);
\path[draw,-,dashed] (2*\s,-\ys) to (2*\s+0.5*\R,-\h*\R-\ys);

\node[red] at (2*\s,-\h*\R-\ys) {$\bullet$};

\node[above] at (2*\s-\R,-\ys) {$-$}; 

\fill[blue,fill opacity=0.1] (2*\s,-\ys) -- (2*\s-\R,-\ys) -- (2*\s-0.5*\R,\h*\R-\ys); 

%\node at (\s+\R+\x,0) {$1$};
%\node at (\s+0.5*\R+0.5*\x,\h*\R+\h*\x) {$2$};

\draw[-latex,shorten >=0.1cm,shorten <=0.1cm] (2*\s-\R+\x,-\ys) -- (2*\s-0.5*\R+0.5*\x,\h*\R-\h*\x-\ys);

%%%% Seed 7
\path[draw,-,dashed] (\s,-\ys) to (\s+\R,-\ys);
\path[draw,-,dashed] (\s,-\ys) to (\s+0.5*\R,\h*\R-\ys);
\path[draw,-,dashed] (\s,-\ys) to (\s-0.5*\R,\h*\R-\ys);
\path[draw,-,very thick] (\s,-\ys) to (\s-\R,-\ys);
\path[draw,-,very thick] (\s,-\ys) to (\s-0.5*\R,-\h*\R-\ys);
\path[draw,-,dashed] (\s,-\ys) to (\s+0.5*\R,-\h*\R-\ys);

\node[red] at (\s,-\h*\R-\ys) {$\bullet$};

\node[below] at (\s-\R,-\ys) {$-$}; 

\fill[blue,fill opacity=0.1] (\s,-\ys) -- (\s-\R,-\ys) -- (\s-0.5*\R,-\h*\R-\ys); 

%\node at (\s+\R+\x,0) {$1$};
%\node at (\s+0.5*\R+0.5*\x,\h*\R+\h*\x) {$2$};

\draw[-latex,shorten >=0.1cm,shorten <=0.1cm] (\s-0.5*\R+0.5*\x,-\h*\R+\h*\x-\ys) --  (\s-\R+\x,-\ys);

%%%% Seed 8
\path[draw,-,dashed] (0,-\ys) to (\R,-\ys);
\path[draw,-,dashed] (0,-\ys) to (0.5*\R,\h*\R-\ys);
\path[draw,-,dashed] (0,-\ys) to (-0.5*\R,\h*\R-\ys);
\path[draw,-,very thick] (0,-\ys) to (-\R,-\ys);
\path[draw,-,very thick] (0,-\ys) to (-0.5*\R,-\h*\R-\ys);
\path[draw,-,dashed] (0,-\ys) to (0.5*\R,-\h*\R-\ys);

\node[red] at (0,-\h*\R-\ys) {$\bullet$};

\node[below] at (-\R,-\ys) {$-$}; 
\node[left] at(-0.5*\R,-\h*\R-\ys) {$-$};

\fill[blue,fill opacity=0.1] (0,-\ys) -- (-\R,-\ys) -- (-0.5*\R,-\h*\R-\ys); 

%\node at (\s+\R+\x,0) {$1$};
%\node at (\s+0.5*\R+0.5*\x,\h*\R+\h*\x) {$2$};

\draw[-latex,shorten >=0.1cm,shorten <=0.1cm] (-\R+\x,-\ys) -- (-0.5*\R+0.5*\x,-\h*\R+\h*\x-\ys);

%%% Mutations
\node at (0.5*\s,0) {$\leftrightsquigarrow$};
\node at (1.5*\s,0) {$\leftrightsquigarrow$};
\node at (2.5*\s,0) {$\leftrightsquigarrow$};
\node at (3*\s,-0.5*\ys) {$\squigarrowupdown$};
\node at (2.5*\s,-\ys) {$\leftrightsquigarrow$};
\node at (1.5*\s,-\ys) {$\leftrightsquigarrow$};
\node at (0.5*\s,-\ys) {$\leftrightsquigarrow$};
\node at (0,-0.5*\ys) {\rotatebox[origin=c]{90}{$=$}};

\end{tikzpicture}
\caption{Seed mutation with period 7}
\label{Figure:LongRankTwo}
\end{center}
\end{figure}

\begin{figure}
\centering
\begin{minipage}{.44\textwidth}
\centering
\begin{tikzpicture}

%%%% Abbreviations
\renewcommand{\R}{2cm} %% length of radius 
\newcommand{\x}{0.3cm}
\newcommand{\h}{0.886} %% height of triangles

\path[draw,-,very thick] (0,0) to (\R,0);
\path[draw,-,very thick] (0,0) to (0.5*\R,\h*\R);
\path[draw,-,very thick] (0,0) to (-0.5*\R,\h*\R);
\path[draw,-,very thick] (0,0) to (-\R,0);
\path[draw,-,very thick] (0,0) to (-0.5*\R,-\h*\R);
\path[draw,-,very thick] (0,0) to (0.5*\R,-\h*\R);

\fill[blue,fill opacity=0.1] (0,0) -- (-\R,0) -- (-0.5*\R,-\h*\R); 
\fill[blue,fill opacity=0.1] (0,0) -- (-\R,0) -- (\R,0) -- (0.5*\R,\h*\R) -- (-0.5*\R,\h*\R) -- (-\R,0); 

\draw[-latex,shorten >=0.1cm,shorten <=0.1cm] (\R-\x,0) -- (0.5*\R-0.5*\x,\h*\R-\h*\x);

\end{tikzpicture}
\caption{The locus of reference points (shaded) for which mutation has a short period}
\label{Figure:GoodRegions}
\end{minipage}
\begin{minipage}{.54\textwidth}
\centering
\begin{tikzpicture}

%%%% Abbreviations
\renewcommand{\R}{2cm} %% length of radius 
\newcommand{\h}{0.886} %% height of triangles
\newcommand{\x}{0.3cm}
\newcommand{\as}{0.8}

\path[draw,-,very thick] (0,0) to (\R,0);
\path[draw,-,very thick] (0,0) to (0.5*\R,\h*\R);
\path[draw,-,very thick] (0,0) to (-0.5*\R,\h*\R);
\path[draw,-,very thick] (0,0) to (-\R,0);
\path[draw,-,very thick] (0,0) to (-0.5*\R,-\h*\R);
%\path[draw,-,dashed] (0,0) to (0.5*\R,-\h*\R);

%\node[right] at (\R,0) {$-\alpha_1$};
%\node[right] at (0.5*\R,\h*\R) {$-\alpha_2$};
%\node[left] at (-0.5*\R,\h*\R) {$\alpha_1+\alpha_2$};
%\node[left] at (-\R,0) {$\alpha_1$};
%\node[left] at (-0.5*\R,-\h*\R) {$\alpha_2$};

\fill[pattern color=blue!50!,pattern=north west lines,opacity=0.5] (0,0) -- (-\R,0) -- (-0.5*\R,-\h*\R); 
\fill[pattern color=blue!50!,pattern=north west lines,opacity=0.5] (0,0) -- (\R,0) -- (0.5*\R,\h*\R); 
\fill[pattern color=blue!50!,pattern=vertical lines,opacity=0.5] (0,0) -- (0.5*\R,\h*\R) -- (-0.5*\R,\h*\R); 
\fill[pattern color=blue!50!,pattern=north east lines,opacity=0.5] (0,0) -- (-0.5*\R,\h*\R) -- (-\R,0);
\fill[pattern color=blue!50!,pattern=north east lines,opacity=0.5] (0,0) -- (-0.5*\R,-\h*\R) -- (0.5*\R,-\h*\R) -- (\R,0);

\draw[-latex,shorten >=0.1cm,shorten <=0.1cm] (\R-\x,0) -- (0.5*\R-0.5*\x,\h*\R-\h*\x);
\draw[-latex,shorten >=0.1cm,shorten <=0.1cm]  (0.5*\R-0.5*\x,\h*\R-\h*\x) -- (-0.5*\R+0.5*\x,\h*\R-\h*\x);
\draw[-latex,shorten >=0.1cm,shorten <=0.1cm]  (-0.5*\R+0.5*\x,\h*\R-\h*\x) -- (-\R+\x,0);
\draw[-latex,shorten >=0.1cm,shorten <=0.1cm]  (-\R+\x,0) -- (-0.5*\R+0.5*\x,-\h*\R+\h*\x);
\draw[-latex,shorten >=0.1cm,shorten <=0.1cm]  (-0.5*\R+0.5*\x,-\h*\R+\h*\x) to [bend right=50] (\R-\x,0);

\node at (0,\as) {$\bullet$};
\node at (-\h*\as,0.5*\as) {$\bullet$};
\node at (-\h*\as,-0.5*\as) {$\bullet$};
\node at (\h*\as,-1.5*\as) {$\bullet$};
\node at (\h*\as,0.5*\as) {$\bullet$};

\draw[-,color=black!80!,very thick] (0,\as) -- (-\h*\as,0.5*\as);
\draw[-,color=black!80!,very thick] (-\h*\as,0.5*\as) -- (-\h*\as,-0.5*\as);
\draw[-,color=black!80!,very thick] (-\h*\as,-0.5*\as) -- (\h*\as,-1.5*\as);
\draw[-,color=black!80!,very thick]  (\h*\as,-1.5*\as) -- (\h*\as,0.5*\as);
\draw[-,color=black!80!,very thick]  (0,\as) -- (\h*\as,0.5*\as);

\end{tikzpicture}
\caption{The construction of the associahedron of type $A_2$ via geometric mutations}
\label{Figure:SeedMutation}
\end{minipage}
\end{figure}

Our aim is to describe which choices result in short periods, and which ones result in long periods. To do this, we define the \emph{swap} of a seed $((v_1,v_2),B)$ to be the seed $((v_2,v_1),\tau(B))$ where $\tau(B)$ is obtained from $B$ by a simultaneous reordering of the rows and columns given by the permutation $(12)$. The swap is denoted by $\tau$.

Let us fix an initial seed $(\mathbf{w},B)$, given by a sector with an angle of size $\alpha$ as in Definition \ref{Def:AngleRankTwo}. For technical reasons it will be convenient to write the pair of vectors as $\mathbf{w}=(w_0,-w_1)$. Without loss of generality we may assume that $b_{12}>0$. We consider the map $\tau\circ \mu_2$. Note that 
\begin{align*}
(\tau\circ \mu_2)((w_0,-w_1,B)=((w_1,-w_2),B')
\end{align*}
for some $w_2\in V$ and a matrix $B'=(b_{ij}')$. Notice that $B'=(\tau\circ\mu_2)(B)=\tau(-B)=B$. In particular, $b'_{12}>0$. Iteration yields seeds $(\tau\circ\mu_2)^{n}(\mathbf{w},B)=((w_n,-w_{n+1}),B)$ parametrised by natural numbers $n\in\mathbb{N}$ all of which have the same exchange matrix. Explicitly, we have
\begin{align*}
w_{n+1}=\begin{cases}
-w_{n-1}+\langle w_{n-1},w_n\rangle w_n&\textrm{if }w_{n}\textrm{ is positive};\\
-w_{n-1}&\textrm{if }w_{n}\textrm{ is negative}.\\
\end{cases}
\end{align*}    

\begin{defn}[Lazy mutations]
We say that $n\in\mathbb{N}$ is \emph{lazy} if $w_n$ is negative. If this is the case, then we also call the mutation $(w_{n-1},-w_n)\leftrightsquigarrow (w_n,-w_{n+1})$ lazy.  
\end{defn}

\begin{rem}
\label{Rem:LazyMutations}
\begin{enumerate} 
\item If $n$ is lazy, then $\langle w_{n+1},w_n\rangle=-\langle w_n,w_{n-1}\rangle$, else $\langle w_{n+1},w_n\rangle=\langle w_n,w_{n-1}\rangle$. 
\item Lazy numbers come in pairs. Suppose that $n\in\mathbb{Z}$ is lazy. Then $w_{n+1}=-w_{n-1}$ so that exactly one of the numbers $n-1$ and $n+1$ is lazy.
\end{enumerate}
\end{rem}

For every $n$ the angle between $w_n$ and $w_0$ is an integer multiple of $\pi/b$. It follows that the sequence $(\tau\circ\mu_2)^{n}(\mathbf{w},B)$ must be periodic.  

\begin{lemma}
\label{Lemma:Period}
Let $p$ be the period of $\tau\circ\mu_2$. Suppose that $2a<b$ (i.e. the angle between the initial vectors $w_0$ and $-w_1$ is acute). Then $p=b+2a$ unless $w_0$ is negative and $w_1$ is positive. In the latter case $p=3b-2a>b+2a$.
\end{lemma}

\begin{proof}
For every $n\in\mathbb{Z}$ we consider the linear map $f_n\colon V\to V$ defined by $w_{n-1}\mapsto w_n$ and $w_n\mapsto w_{n+1}$. If $n$ is not lazy, then $f_n$ is a rotation around the origin by the angle $\psi\in[0,\pi]$ for which $\langle w_{n-1},w_n\rangle =2\cos(\psi)$. Since lazy numbers come in pairs, the value of $\langle w_{n-1},w_n\rangle$ is the same for all non-lazy numbers $n$ by Remark \ref{Rem:LazyMutations} (1). In other words, for every non-lazy number $n$ the map $f_n$ is a rotation by the same angle $\psi$ in the same direction. We have $\psi=\alpha$ or $\psi=\pi-\alpha$, depending on the sign of $w_0$ and $w_1$.

To find the period $p$, we consider several cases. First, suppose that $w_1$ is negative. In this case, the number $1$ is not lazy, and hence $f_1$ is a rotation so that $\psi=\alpha$. Second, suppose that $w_0$ is positive. In this case $0$ is not lazy, and hence $f_0$ is a rotation so that $\psi=\alpha$. Now suppose that $w_1$ is positive and $w_0$ negative. In this case, $0$ and $1$ are both lazy. It follows that $2$ is not lazy and hence $f_2$ is a rotation so that $\psi=\pi-\alpha$. 

Let us write $\psi=c\pi/b$ with $c\in\{a,b-a\}$. In the case $c=a$, we have to show that $p=b+2a$; in the case $c=b-a$, we have to show that $p=3b-2a$. Therefore, we have to show $p=b+2c$ in both cases. To this end, we show that there are exactly $b$ non-lazy mutations and exactly $c$ pairs of lazy mutations along a fundamental period of seeds. We denote the number of non-lazy mutations along a period by $b'$, and the number of pairs of lazy mutations by $c'$. We claim $b'=b$ and $c'=c$.

The composition $f_{n+1}\circ f_n$ is a rotation by $\pi$ when $(n,n+1)$ is a pair of lazy numbers. Since $f_n$ is a rotation by $\psi$ for every non-lazy number $n$, we must have $c'\pi+b'\psi=2\pi m$ for some integer $m$. The equation implies that $c'+b'\cdot\tfrac{c}{b}=2m$ is an integer. Because of the coprimality of $b$ and $c$, we conclude that $b'$ must be a multiple of $b$. It cannot be zero because there must be non-lazy mutations. Hence $b'\geq b$.

There are $b$ lines through the origin intersecting the sides of the initial sector under an angle that is a multiple of $\pi/b$. These lines divide the plane into $2b$ sectors with angles $\pi/b$, which we can label consecutively by elements in $\mathbb{Z}/b\mathbb{Z}$ when we identify opposite sectors. Now suppose we perform $b$ non-lazy mutations. The occurring seeds are given by sectors made of $c$ smaller sectors whose labels are given by 
\begin{align*}
[k+1,k+c],[k+c+1,k+2c],\ldots,[k+(b-1)c+1,k+bc],[k+1,k+c].
\end{align*}   
for some $k\in \mathbb{Z}/b\mathbb{Z}$. (Recall that $[r,s]=[r,r+1,\ldots,r+s]$.) The non-lazy mutations are interrupted by pairs of lazy mutations. Lazy mutations happen between intervals $[k+(l-1)c+1,k+lc]$ and $[k+lc+1,k+l(c+1)]$ if $[k+(l-1)c+1,k+lc]$ contains a fixed sector $t\in \mathbb{Z}/b\mathbb{Z}$ (determined by the reference point $u$). Hence the number of pairs of lazy mutations is equal to the number of elements in the sequence $k+1,k+2,\ldots,k+bc$ that are equal to $t$. This number is equal to $c$, and hence $p=b+2c$.
\end{proof}

There are three types of cluster algebras of rank $2$ with finitely many cluster variables, namely $A_2$, $B_2$, and $G_2$. The entries of the exchange matrix $B=(b_{ij})$ fulfil the relation $\vert b_{ij}b_{ji}\rvert=1=4\cos^2(\pi/3)$ in type $A_2$, $\vert b_{ij}b_{ji}\rvert=2=4\cos^2(\pi/4)$ in type $B_2$, and $\vert b_{ij}b_{ji}\rvert=3=4\cos^2(\pi/6)$. The orbits of the initial seed under the transformation $\tau\circ\mu_2$ have lengths $5$, $6$, and $8$, respectively. We observe that the natural number $b$ such that $\vert b_{ij}b_{ji}\rvert=4\cos^2(\pi/b)$ and the period $p$ obey the relation $p=b+2$, that is, the first case occurs in Lemma \ref{Lemma:Period}.

\begin{defn}[Short and long periods] We refer to a period of the form $p=b+2a$ as in Lemma~ \ref{Lemma:Period} as a \emph{short} period and to a period of the form $p=3b-2a$ as a \emph{long} period.
\end{defn}

\begin{defn}[Compatibility of the reference point]
\label{Def:Positivity}
We say that the choice $u \in V$ is \emph{compatible} with the seed $((w_0,-w_1),B)$ if mutation of this seed (with respect to the reference point $u$) has a short period.
\end{defn}

In other words, $u \in V$ is compatible with the seed $((w_0,-w_1),B)$ if $w_0$ is positive or $w_1$ is negative.

\begin{rem}
\label{Rem:ForbiddenRegions}
Let us describe the locus of compatible reference points geometrically. We fix a sector with angle $\angle (A,O,B)$ and orient its sides $OA \to OB $, and denote by $(\mathbf{v},B)$ the associated seed. Then the reference point $u$ is compatible with $(\mathbf{v},B)$ unless $u$ lies inside the supplementary angle $\angle (-B,O,A)$.   
\end{rem}

\subsection{Geometric realisations of seeds in rank 3} 
\label{Subsection:GeoRealRank3}

For the rest of the article we concern ourselves with mutation-finite exchange matrices of rank $3$. We consider a skew-symmetric matrix $B\in\operatorname{Mat}_{3\times 3}(\mathbb{R})$ and assume it is mutation-finite. Proposition \ref{Prop:Classification} (\ref{Prop:ClassificationMarkov}) says that if $B$ is mutation-cyclic, then $B$ is the exchange matrix of the Markov cluster algebra. The exchange graph of this cluster algebra is an infinite $3$-regular tree (of exponential growth). From now on, we will suppose without loss of generality that $B$ is mutation-acyclic. Part $(3)$ of Definition~\ref{Def:GeoReal} implies that a geometric realisation is represented by an acute-angled triangle if and only if its quiver is acyclic, see \cite{FT}.

\begin{rem}[\cite{FT}]
\label{Rem:FTProp}
Suppose that $B$ is a mutation-acyclic exchange matrix of rank $3$. Then it admits a geometric realisation by partial reflections.
\begin{enumerate}
\item \label{Rem:FTPropSph} If the Markov constant satisfies $C(B)<4$, then the geometric realisation of $B$ lies on the sphere $\mathbb{S}^2$. An initial seed can be realised by a triangle in $\mathbb{S}^2$ whose interior angles are as in Proposition \ref{Prop:Classification} (\ref{Prop:ClassificationSpherical}).
\item  \label{Rem:FTPropEuc} If the Markov constant satisfies $C(B)=4$, then the geometric realisation of $B$ lies in the Euclidean plane $\mathbb{E}^2$. An initial seed can be realised by a triangle in $\mathbb{E}^2$ whose interior angles are as in Remark \ref{Rem:classificationEuclidean}.
\end{enumerate}   
\end{rem}

The Definition \ref{Def:SeedMutation} of a seed mutation requires a choice of positive and negative vectors. Like in the situation in rank $2$, as discussed in Section \ref{Section:Rank2}, we define positive and negative vectors through a reference point, compare also Macdonald \cite[Section 1.2]{Mac}. To construct the set $V^{+}\subseteq V$,
\begin{enumerate}
\item we fix a point $u\in \mathbb{S}^2$ on the sphere if the exchange matrix $B$ falls in the spherical case (\ref{Rem:FTPropSph}) of Remark \ref{Rem:FTProp},
\item or we fix a point $u\in \mathbb{E}^2$ in the Euclidean plane (or in a projective plane containing $\mathbb{E}^2$) if the exchange matrix $B$ falls in the Euclidean case (\ref{Rem:FTPropEuc}) of Remark \ref{Rem:FTProp}.
\end{enumerate}
Every vector $v\in V$ defines a half-sphere $\Pi_{v}$ (in case (\ref{Rem:FTPropSph})) or a half-plane $\Pi_v$ (in case (\ref{Rem:FTPropEuc})). The details of this construction will be explained in Sections \ref{Section:Spherical} and \ref{Sec:GeomRealE}.

\begin{defn}[Positive vectors]
We say that $v\in V$ is \emph{positive} if $u$ belongs to $\Pi_{v}$. In this case we write $v\in V^{+}$. We say $v$ belongs to $V^{-}$ if $-v\in V^{+}$. 
\end{defn}

\begin{defn}[Compatibility in rank 3]
We say that the reference point $u$ is \emph{compatible} with the seed $(\mathbf{v},B)$ if $u$ is compatible with every rank $2$ subseed of $(\mathbf{v},B)$ in the sense of Definition~\ref{Def:Positivity}. We say that $u$ is \emph{compatible} if it is compatible with every seed in the mutation class of $(\mathbf{v},B)$.
\end{defn}

As we have seen in Section \ref{Section:Rank2}, not every notion of positivity will recover the exchange graphs for classical cluster algebras. Moreover, the following condition \ref{Condition:Pos1} is a necessary condition for exchange graphs to agree with classical exchange graphs. Hence it is natural to impose this condition in general. We will see in Section \ref{Section:Spherical} that it is also a sufficient condition to recover classical exchange graphs. We fix an initial seed $(\mathbf{v}_0,B)$ and a reference point $u$, and denote the resulting exchange graph by $\Gamma$.

\begin{description}[style=multiline]
    \item[(C)\namedlabel{Condition:Pos1}{C}] (\textbf{Compatibility Condition}) We say that the compatibility condition holds if the reference point $u$ is compatible with every seed $(\mathbf{v},B)$ in the exchange graph $\Gamma$.
\end{description} 

\phantomsection

\begin{defn}[Finite and affine seeds]
\label{Def:FinAff}
Suppose that $(\mathbf{v},B)$ is a seed of rank $3$. 
\begin{enumerate}
\item We say that $(\mathbf{v},B)$ is of \emph{finite type} if its exchange graph is finite, or, equivalently, if its geometric realisation lies on the sphere $\mathbb{S}^2$. We say it is \emph{of infinite type} otherwise.
\item We say that $(\mathbf{v},B)$ is of \emph{affine} if its geometric realisation lies in the Euclidean plane $\mathbb{E}^2$.
\end{enumerate} 
\end{defn}

\begin{rem}[Finite and affine quivers] As the exchange matrix is the essential part from which we can determine the seed $(\mathbf{v},B)$, we also use the terminology from Definition \ref{Def:FinAff} when considering exchange matrices only. In particular,
\label{Def:Affine}
\begin{enumerate}
\item an exchange matrix $B$ is of \emph{finite type} if its geometric realisation lies on the sphere $\mathbb{S}^2$;
\item an exchange matrix $B$ is \emph{affine} if its geometric realisation lies in the Euclidean plane $\mathbb{E}^2$.
\end{enumerate}
\end{rem}

\begin{rem} In Definition \ref{Def:Affine} the term \emph{affine} is justified because the condition is satisfied for all affine quivers in the usual setting with integer exchange matrices. In particular, quivers of type $\tilde{A}_2$, $\tilde{B_2}$, etc. satisfy the condition.
\end{rem}

\subsection{Exchange graphs for quivers of finite type}
\label{Section:Spherical}

Consider a mutation-finite exchange matrix $B$ of rank $3$. This subsection concerns case (\ref{Rem:FTPropSph}) of Remark \ref{Rem:FTProp}, that is, we assume that $B$ is mutation-acyclic with $C(B)<4$. We choose a geometric realisation $\mathbf{v}=(v_1,v_2,v_3)$ of $B$ inside a quadratic space $V=\mathbb{R}^3$, which we may realise on the sphere $\mathbb{S}^2$ after projectivisation.

More precisely, the intersection of the half-space $\Pi_{v_i}=\{\,w\in\mathbb{R}^3\mid (v_i,w)<0\,\}$, $i\in\{1,2,3\}$, with $\mathbb{S}^2$ defines a half-sphere. Hence the intersection $\Pi_{v_1}\cap\Pi_{v_2}\cap\Pi_{v_3}$ defines a spherical triangle. The triangle is bounded by $3$ great circles, given by intersecting $\mathbb{S}^2$ with the planes $l_{v_i}=\{\,w\in\mathbb{R}^3\mid (v_i,w)=0\,\}$, $i\in\{1,2,3\}$. The concept of \emph{geometric realisation} now means that if $\lvert b_{ij}\rvert=\lvert (v_i,v_j)\rvert=2\cos(\pi t_{ij})$ for two indices $i$ and $j$, then $\pi t_{ij}$ is the angle between the great circles associated to $v_i$ and $v_j$. The number $t_{ij}$ is rational by Proposition \ref{Prop:Classification} (\ref{Prop:ClassificationRational}). A \emph{seed} is now given by spherical triangle together with a quiver whose vertices are associated to the sides of the triangle. The quiver encodes the placements of the signs inside the skew-symmetric matrix $B$.

Thanks to Proposition \ref{Prop:Classification} (\ref{Prop:ClassificationSpherical}) we may assume that
\begin{align}
\label{Eqn:SphericalB}
B=\left(\begin{matrix}
0&2\cos(\pi t_1)&0\\
-2\cos(\pi t_1)&0&2\cos(\pi t_2)\\
0&-2\cos(\pi t_2)&0
\end{matrix}\right)
\end{align}
where $(t_1,t_2)$ belongs to the set
\begin{align*}
Sph=\left\lbrace\,\left(\frac{1}{3},\frac{1}{3}\right),\,\left(\frac{1}{3},\frac{1}{4}\right),\,\left(\frac{1}{3},\frac{1}{5}\right),\,\left(\frac{1}{3},\frac{2}{5}\right),\,\left(\frac{1}{5},\frac{2}{5}\right)\right\rbrace.
\end{align*}

All five choices of the pair $(t_1,t_2)$ yield finite exchange graphs. The choices $(t_1,t_2)=\left(\tfrac13,\tfrac13\right)$ and $(t_1,t_2)=\left(\tfrac13,\tfrac14\right)$ yield cluster algebras of type $A_3$ and $B_3$. Their exchange graphs are (generalised) associahedra. The choice $(t_1,t_2)=\left(\tfrac13,\tfrac15\right)$ yields a generalised associadron of type $H_3$, which was given explicitly in the work of Fomin--Reading \cite[Figure 5.14]{FR}. For a particular choice of a reference point, Felikson--Tumarkin drew the exchange graph for an exchange matrix associated to $(t_1,t_2)=\left(\tfrac{1}{3},\tfrac{2}{5}\right)$ and $(t_1,t_2)=\left(\tfrac{1}{5},\tfrac{2}{5}\right)$ explicitly in \cite[Figure 6.4]{FT}. The exchange graphs contain $40$ and $48$ seeds, respectively.

Since the number of fundamental regions is finite we can use hand and computer calculations to prove the following.

\begin{theorem}
\phantomsection\label{Thm:SphericalExchangeGraph}
Fix a pair $(t_1,t_2)\in Sph$, and consider the exchange matrix $B$ from equation (\ref{Eqn:SphericalB}).
\begin{enumerate}
\item There are geometric realisations $(\mathbf{v},B)$ with the reference point in a compatible position;
\item all such realisations result in the same exchange graph;
\item in all such realisations, the compatible reference point belongs to some acute-angled domain (i.e. a domain associated to an acyclic seed);
\item every choice of a reference point in an acute-angled domain (i.e. a domain associated to an acyclic seed) gives such a realisation.
\end{enumerate} 
\end{theorem}

The exchange graphs in Theorem \ref{Thm:SphericalExchangeGraph} (2) agree with the exchange graphs of the ordinary cluster algebras of type $A_3$ or $B_3$ if $(t_1,t_2)=\left(\tfrac{1}{3},\frac{1}{3}\right)$ or $(t_1,t_2)=\left(\tfrac{1}{3},\frac{1}{4}\right)$; they agree with Fomin--Reading's exchange graph of type $H_3$ if $(t_1,t_2)=\left(\tfrac{1}{3},\frac{1}{5}\right)$; and they agree with Felikson--Tumarkin's exchange graphs if $(t_1,t_2)=\left(\tfrac{1}{3},\tfrac{2}{5}\right)$ or $(t_1,t_2)=\left(\tfrac{1}{5},\tfrac{2}{5}\right)$.

An example of a geometric picture of the exchange graph of type $A_3$ is shown in Figure \ref{Figure:ExchangeGraphSphere} where the reference point is represented by a star. Here a seed is given by a spherical triangle together with a quiver.

\begin{rem}[Belt line $b$] In Figure \ref{Figure:ExchangeGraphSphere}, a spherical line $b$ of the sphere is drawn in red. The line $b$ has formidable properties. First, it is a line which intersects precisely the acyclic seeds. Second, the spherical triangle of every cyclic seed admits two vertices such that the orthogonal projection of the vertex to the opposite side lies on the spherical line $b$. One can check that a spherical line with the same properties exists for every compatible realisation in Theorem \ref{Thm:SphericalExchangeGraph}. We will see later in Section \ref{Sec:InitialAcyclicBelt} that a line with the same properties exists for exchange graphs in the Euclidean plane.
\end{rem}

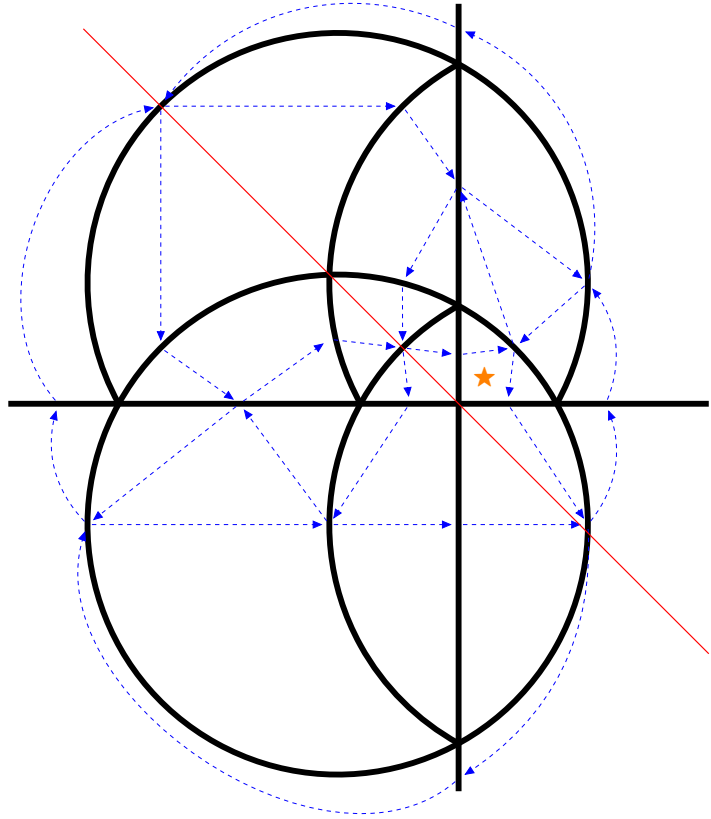
\begin{figure}
\resizebox{0.6\textwidth}{!}{
\begin{tikzpicture}

\newcommand{\AxisLength}{6}
\newcommand{\Radius}{6}
\newcommand{\CentreCircleCoord}{2.9}

\tikzmath{     
   \xaxis = sqrt(\Radius*\Radius-\CentreCircleCoord*\CentreCircleCoord)-\CentreCircleCoord;
   \sangle = asin(\CentreCircleCoord/\Radius);
   \he = sqrt(0.5*\Radius*\Radius-\CentreCircleCoord*\CentreCircleCoord);
};

\draw [black,line width=4pt] (0,\xaxis) arc [radius=\Radius, start angle=\sangle+90, end angle=270-\sangle];
\draw [black,line width=4pt] (0,\xaxis) arc [radius=\Radius, start angle=90-\sangle, end angle=450-\sangle];

\draw [black,line width=4pt] (\xaxis,0) arc [radius=\Radius, start angle=-\sangle, end angle=180+\sangle];

\draw [black,line width=4pt] (0,1*\xaxis+2*\CentreCircleCoord) arc [radius=\Radius, start angle=90+\sangle, end angle=180+\sangle];

\draw [-,black,line width=4pt] (-1.8*\AxisLength,0) to (\AxisLength,0);
\draw [-,black,line width=4pt] (0,-1.55*\AxisLength) to (0,1.6*\AxisLength);

\draw[shorten >=5pt,shorten <=3pt,->,blue,dashed,>=triangle 45] (0,0.5*\xaxis) -- (0.7071*\Radius-\CentreCircleCoord,0.7071*\Radius-\CentreCircleCoord);

\draw[shorten >=5pt,shorten <=3pt,->,blue,dashed,>=triangle 45] (0.7071*\Radius-\CentreCircleCoord,0.7071*\Radius-\CentreCircleCoord) -- (0.5*\xaxis,0);

\draw[shorten >=5pt,shorten <=3pt,->,blue,dashed,>=triangle 45] (-0.7071*\Radius+\CentreCircleCoord,0.7071*\Radius-\CentreCircleCoord) -- (0,0.5*\xaxis);

\draw[shorten >=5pt,shorten <=3pt,->,blue,dashed,>=triangle 45] (-0.7071*\Radius+\CentreCircleCoord,0.7071*\Radius-\CentreCircleCoord) -- (-0.5*\xaxis,0);

\draw[shorten >=5pt,shorten <=3pt,->,blue,dashed,>=triangle 45] (0.5*\xaxis,0) -- (\Radius-\CentreCircleCoord,-\CentreCircleCoord);

\draw[shorten >=5pt,shorten <=3pt,->,blue,dashed,>=triangle 45] (0,-\CentreCircleCoord) -- (\Radius-\CentreCircleCoord,-\CentreCircleCoord);

\draw[shorten >=5pt,shorten <=3pt,->,blue,dashed,>=triangle 45] (-0.5*\xaxis,0) -- (-\Radius+\CentreCircleCoord,-\CentreCircleCoord);

\draw[shorten >=5pt,shorten <=3pt,->,blue,dashed,>=triangle 45] (-\Radius+\CentreCircleCoord,-\CentreCircleCoord) -- (0,-\CentreCircleCoord);

\draw[shorten >=5pt,shorten <=3pt,->,blue,dashed,>=triangle 45] (-\CentreCircleCoord-\Radius,-\CentreCircleCoord) -- (-\Radius+\CentreCircleCoord,-\CentreCircleCoord);

\draw[shorten >=5pt,shorten <=3pt,->,blue,dashed,>=triangle 45] (-\CentreCircleCoord-\xaxis,0) -- (-\CentreCircleCoord-\Radius,-\CentreCircleCoord);

\draw[shorten >=5pt,shorten <=3pt,->,blue,dashed,>=triangle 45] (-\Radius+\CentreCircleCoord,-\CentreCircleCoord) -- (-\CentreCircleCoord-\xaxis,0);

\draw[shorten >=5pt,shorten <=3pt,->,blue,dashed,>=triangle 45] (-\CentreCircleCoord-\xaxis,0) -- (-0.5*\he+0.5*\CentreCircleCoord-0.5*\Radius,0.5*\he);

\draw[shorten >=5pt,shorten <=3pt,->,blue,dashed,>=triangle 45] (-\CentreCircleCoord-0.7071*\Radius,-\CentreCircleCoord+0.7071*\Radius) -- (-\CentreCircleCoord-\xaxis,0);

\draw[shorten >=5pt,shorten <=3pt,->,blue,dashed,>=triangle 45] (-0.5*\he+0.5*\CentreCircleCoord-0.5*\Radius,0.5*\he) -- (-0.7071*\Radius+\CentreCircleCoord,0.7071*\Radius-\CentreCircleCoord);

\draw[shorten >=5pt,shorten <=3pt,->,blue,dashed,>=triangle 45] (-\CentreCircleCoord+0.2588*\Radius,-\CentreCircleCoord+0.966*\Radius) -- (-0.7071*\Radius+\CentreCircleCoord,0.7071*\Radius-\CentreCircleCoord);

\draw[shorten >=5pt,shorten <=3pt,->,blue,dashed,>=triangle 45] (0,\CentreCircleCoord+\xaxis) -- (-\CentreCircleCoord+0.2588*\Radius,-\CentreCircleCoord+0.966*\Radius);

\draw[shorten >=5pt,shorten <=3pt,->,blue,dashed,>=triangle 45] (\CentreCircleCoord-0.7071*\Radius,\CentreCircleCoord+0.7071*\Radius) -- (0,\CentreCircleCoord+\xaxis);

\draw[shorten >=5pt,shorten <=3pt,->,blue,dashed,>=triangle 45] (0.7071*\Radius-\CentreCircleCoord,0.7071*\Radius-\CentreCircleCoord) -- (0,\CentreCircleCoord+\xaxis);

\draw[shorten >=5pt,shorten <=3pt,->,blue,dashed,>=triangle 45] (0,\CentreCircleCoord+\xaxis) -- (\Radius-\CentreCircleCoord,\CentreCircleCoord);

\draw[shorten >=5pt,shorten <=3pt,->,blue,dashed,>=triangle 45] (\Radius-\CentreCircleCoord,\CentreCircleCoord) -- (0.7071*\Radius-\CentreCircleCoord,0.7071*\Radius-\CentreCircleCoord);

\draw[shorten >=5pt,shorten <=3pt,->,blue,dashed,>=triangle 45] (-\CentreCircleCoord-0.7071*\Radius,\CentreCircleCoord+0.7071*\Radius) -- (\CentreCircleCoord-0.7071*\Radius,\CentreCircleCoord+0.7071*\Radius);

\draw[shorten >=5pt,shorten <=3pt,->,blue,dashed,>=triangle 45] (-\CentreCircleCoord-0.7071*\Radius,\CentreCircleCoord+0.7071*\Radius) -- (-\CentreCircleCoord-0.7071*\Radius,-\CentreCircleCoord+0.7071*\Radius);

\draw[shorten >=5pt,shorten <=3pt,->,blue,dashed,>=triangle 45] (-1.6*\AxisLength,0) to [bend left=60] (-\CentreCircleCoord-0.7071*\Radius,\CentreCircleCoord+0.7071*\Radius);

\draw[shorten >=5pt,shorten <=3pt,->,blue,dashed,>=triangle 45] (0,1.5*\AxisLength) to [bend right=40] (-\CentreCircleCoord-0.7071*\Radius,\CentreCircleCoord+0.7071*\Radius);

\draw[shorten >=5pt,shorten <=3pt,->,blue,dashed,>=triangle 45] (\Radius-\CentreCircleCoord,\CentreCircleCoord) to [bend right=40] (0,1.5*\AxisLength);

\draw[shorten >=5pt,shorten <=3pt,->,blue,dashed,>=triangle 45] (1.5*\xaxis,0) to [bend right=30] (\Radius-\CentreCircleCoord,\CentreCircleCoord);

\draw[shorten >=5pt,shorten <=3pt,->,blue,dashed,>=triangle 45] (\Radius-\CentreCircleCoord,-\CentreCircleCoord) to [bend right=30] (1.5*\xaxis,0);

\draw[shorten >=5pt,shorten <=3pt,->,blue,dashed,>=triangle 45] (\Radius-\CentreCircleCoord,-\CentreCircleCoord) to [bend left=30] (0,-1.5*\AxisLength);

\draw[shorten >=5pt,shorten <=3pt,->,blue,dashed,>=triangle 45] (0,-1.5*\AxisLength) to [bend left=75] (-\CentreCircleCoord-\Radius,-\CentreCircleCoord);

\draw[shorten >=5pt,shorten <=3pt,->,blue,dashed,>=triangle 45] (-\CentreCircleCoord-\Radius,-\CentreCircleCoord) to [bend left=30] (-1.6*\AxisLength,0);

\node[orange] at (0.2*\he,0.2*\he) {\textbf{\Large{$\bigstar$}}};

\draw[-,red,thick] (-1.5*\AxisLength,1.5*\AxisLength) -- (1*\AxisLength,-1*\AxisLength); 

\end{tikzpicture}}
\caption{Realising the exchange graph of type $A_3$ on the sphere}
\label{Figure:ExchangeGraphSphere}
\end{figure}

\subsection{Realisations in the Euclidean plane} 
\label{Sec:GeomRealE}

Consider a mutation-finite exchange matrix $B$ of rank $3$. The rest of the paper concerns case (\ref{Rem:FTPropEuc}) of Remark \ref{Rem:FTProp}, that is, we assume that $B$ is mutation-acyclic with $C(B)=4$. We choose a geometric realisation $\mathbf{v}=(v_1,v_2,v_3)$ of $B$ inside a quadratic space $V=\mathbb{R}^3$. Now we describe how to realise the seed $(\mathbf{v},B)$ in the Euclidean plane $\mathbb{E}^2$ after projectivisation.

There are two bilinear forms $V\times V\to\mathbb{R}$. First, we denote the standard Euclidean form by angular brackets $\langle\cdot,\cdot\rangle\colon V\times V\to\mathbb{R}$. Moreover, as in Definition \ref{Def:GeoReal}, we denote the bilinear form defining the quadratic space $V$ by round brackets $(\cdot,\cdot)\colon V\times V\to\mathbb{R}$. The condition $C(B)=4$ implies $\operatorname{det}(B)=0$. So the bilinear form $(\cdot,\cdot)\colon V\times V\to \mathbb{R}$ is degenerate. Recall that the \emph{radical} of a symmetric bilinear form $b\colon V\times V\to \mathbb{R}$ is
\begin{align*}
\operatorname{rad}(b)=\{\,v\in V\mid b(v,w)=0\textrm{ for all }w\in V\,\}.
\end{align*}
Using Remark \ref{Rem:classificationEuclidean} it is not hard to show that $(\cdot,\cdot)$ is positive semi-definite and that its radical has dimension $1$.

We choose a non-zero vector $e_1\in V$ in the radical of $(\cdot,\cdot)$. Extend $e_1$ to an ordered basis $\mathbf{e}=(e_1,e_2,e_3)$ of $V$ by choosing vectors $e_2$ and $e_3$. We pick a hyperplane $P\subseteq \mathbb{R}^3$ that is parallel to $e_2$ and $e_3$ and does not contain the origin. To be concrete, let us put $P=\{\,(w_1,w_2,w_3)\in \mathbb{R}^3\mid w_1=1\,\}$. Here coordinates are written with respect to the basis $\mathbf{e}$.

Let $v = (v_1,v_2,v_3)\in V$, and suppose that $v$ is not parallel to $e_1$, that is, $(v_2,v_3)\neq (0,0)$. Then the orthogonal complement $v^{\perp}=\{\,w\in V\mid \langle v,w\rangle=0\,\}$ is not parallel to $P$. Hence the intersection of $l_{v}=v^{\perp}\cap P$ defines a line in $P$. The line bounds the half-plane $\Pi_v=\{\,w\in V\mid \langle v,w\rangle<0\,\}\cap P$, and the intersection $\Pi_{v_1}\cap \Pi_{v_2}\cap \Pi_{v_3}$ defines a planar triangle (or an infinite region bounded by three lines). The concept of \emph{geometric realisation} now means that if $\lvert b_{ij}\rvert=\lvert (v_i,v_j)\rvert=2\cos(\pi t_{ij})$ for two indices $i$ and $j$, then $\pi t_{ij}$ is the angle between the sides of the triangle associated to $v_i$ and $v_j$. The number $t_{ij}$ is rational by Proposition \ref{Prop:Classification} (\ref{Prop:ClassificationRational}). A \emph{seed} is now given by a triangle in $P$ together with a quiver whose vertices are the sides of the triangle. The quiver encodes the placements of the signs inside the skew-symmetric matrix $B$.    

\begin{rem} Another geometric description can be constructed as follows. Consider the Euclidean plane $\mathbb{E}^2=\mathbb{R}^2$. Recall that a map $f\colon \mathbb{E}^2\to\mathbb{R}$ is \emph{affine linear} if there exists a row vector $a\in\mathbb{R}^2$ and a real number $b$ such that $f(x)=ax+b$ for all $x$. We denote by $F$ the set of all affine linear maps $\mathbb{E}^2\to\mathbb{R}$. Clearly, $F$ is a $3$-dimensional vector space. We define a symmetric bilinear form on $F$ by setting $(f,g)=(a,c)$ when $f(x)=ax+b$ and $g(x)=cx+d$ for all $x$. Note that the bilinear form is degenerate and its radical consists of all constant maps.

For every $f\in F$ define $l_f=f^{-1}(0)$ and $\Pi_f=f^{-1}([0,\infty))$. If $f$ is non-constant, then $l_f\subseteq\mathbb{E}^2$ is a line and $\Pi_f\subseteq \mathbb{E}^2$ is a half-plane bounded by $l_f$. Notice that $l_f$ and $\Pi_f$ do not change when we multiply $f$ by a positive scalar. In this way, a seed $(\mathbf{v},B)$ defines a triangle $\Delta$ in the Euclidean plane whose angles encode the entries of $B$. 

Both constructions are dual to each other. Vectors $v = (v_1,v_2,v_3)\in V$ are in bijection with linear maps $f_v\colon V\to \mathbb{R}$ defined by $f_v(w_1,w_2,w_3) = v_1w_1+v_2w_2+v_3w_3$. The hyperplane $P=\mathbb{R}^2$ embeds in $\mathbb{R}^3$ via $(w_2,w_3)\to (1,w_2,w_3)$. The restriction of $f_v$ to $P$ is an affine linear map $P\to \mathbb{R}$. Conversely, every affine linear map $P\to\mathbb{R}$ extends uniquely to a linear map $f_v$ from $V$ to $\mathbb{R}$. The zero set of the restriction of $f_v$ to $P$ defines the same line $l_{v}$ as above. The viewpoint of affine linear maps is used in the context of affine root systems, see Macdonald \cite[Section 1.2]{Mac}. Angles measured in $P=\mathbb{R}^2$ using the Euclidean bilinear form on $\mathbb{R}^2$ agree with angles measured in $F$ using the bilinear form above.
\end{rem}

\section{Geometric mutations of Euclidean triangles with rational angles}
\label{Section:GeoMut}

\subsection{The initial configuration}

We fix a mutation-finite exchange matrix $B_0$ of rank $3$ which is mutation-acyclic with $C(B_0)=4$ as in case (\ref{Rem:FTPropEuc}) of Remark \ref{Rem:FTProp}. We refer to $B_0$ as the \emph{initial} exchange matrix. We consider a geometric realisation $(\mathbf{v}_0,B_0)$ in the Euclidean plane $\mathbb{E}^2$ as constructed in Section \ref{Sec:GeomRealE}. It is given by a triangle $\Delta_0$ and a quiver $Q_0$. In view of Proposition~ \ref{Prop:Classification} (\ref{Prop:ClassificationRational}) the angles of $\Delta_0$ are rational multiples of $\pi$. 

Proposition \ref{Prop:Classification} (\ref{Prop:ClassificationEuclidean}) asserts that $(\mathbf{v}_0,B_0)$ is then mutation-equivalent to a seed formed by two parallel lines and a third line intersecting the parallel lines under the angle $\pi/d$ for some positive integer $d$, see Figure \ref{Figure:Parallel}.

\begin{rem}
\label{Rem:classificationEuclidean}
We can construct other matrices in the same mutation class using considerations as in the proof of \cite[Theorem 6.11]{FT}. When we mutate the matrix in equation (\ref{Eq:InfiniteRegions}) at indices $3$ and $2$ alternately, we obtain matrices 
\begin{align*}
\left(\begin{smallmatrix}
0&2\cos(\pi t_3)&2\cos(\pi t_2)\\
-2\cos(\pi t_3)&0&2\cos(\pi t_1)\\
-2\cos(\pi t_2)&-2\cos(\pi t_1)&0
\end{smallmatrix}\right).
\end{align*}
where $(t_1,t_2,t_3)$ is a permutation of $(\tfrac{1}{d}\pi ,\tfrac{k}{d}\pi,\tfrac{d-k-1}{d}\pi)$ for some natural number $k\in [0,\tfrac{d}{2}]$. Note that the sum of these angles is equal to $\pi$, so that they form a Euclidean triangle. We plug in $k=\tfrac{d-1}{2}$ if $d$ is odd, and $k=\tfrac{d}{2}$ if $d$ is even. After writing $d=2n+1$ or $d=2n$ for some $n\geq 1$, we can rephrase Proposition \ref{Prop:Classification} (\ref{Prop:ClassificationEuclidean}) as follows. Suppose that $B$ is mutation-acyclic and $C(B)=4$. Then $B$ is mutation-finite if and only it is mutation-equivalent to 
\begin{align*}
\left(\begin{smallmatrix}
0&2\cos(\pi t_3)&2\cos(\pi t_2)\\
-2\cos(\pi t_3)&0&2\cos(\pi t_1)\\
-2\cos(\pi t_2)&-2\cos(\pi t_1)&0
\end{smallmatrix}\right).
\end{align*}
where $(t_1,t_2,t_3)$ is one of the triples
\begin{align*}
\left(\frac{n}{2n+1}\pi,\frac{n}{2n+1}\pi,\frac{1}{2n+1}\pi\right),\quad\left(\frac{n-1}{2n}\pi,\frac{1}{2}\pi,\frac{1}{2n}\pi\right)&&(n\geq 1).
\end{align*}
\end{rem}

In fact, a stronger statement is true. 

\begin{prop}
\label{Prop:PossibleTriangles}
Let $n_1,n_2,n_3\in [1,d]$ with $n_1+n_2+n_3=d$. The mutation class of $(\mathbf{v}_0,B_0)$ contains a triangle with angles $\tfrac{n_1}{d}\pi$, $\tfrac{n_2}{d}\pi$ and $\tfrac{n_3}{d}\pi$ if and only $\operatorname{gcd}(n_1,n_2,n_3)=1$. 
\end{prop}

\begin{proof}
Suppose that $\tfrac{n_1}{d}\pi$, $\tfrac{n_2}{d}\pi$ and $\tfrac{n_3}{d}\pi$ are the angles of a triangle $\Delta$ and $n_1,n_2,n_3$ have a non-trivial common divisor $e$. Then the angles of every triangle in the mutation class of $\Delta$ are multiples of $e\pi/(d)$. This mutation class cannot contain $\Delta_0$.

For the converse direction, notice that the matrix $B$ is mutation-finite by the $\pi$-rationality of the angles. Hence, we can apply Felikson--Tumarkin's classification of mutation classes for mutation-finite exchange matrices in rank $3$, see \cite[Section 6]{FT}. In particular, a seed whose triangle has angles $\tfrac{n_1}{d}\pi$, $\tfrac{n_2}{d}\pi$ and $\tfrac{n_3}{d}\pi$ is mutation-equivalent to $(\mathbf{v}_0,B_0)$.
\end{proof}

\begin{rem}
Note that the number $d$ is the least common denominator of the three rational multiples of $\pi$ in any geometric realisation obtained from $(\mathbf{v}_0,B_0)$ by sequences of mutations. We consider separately the cases where $d$ is even or odd. In this Section \ref{Section:GeoMut} and in Sections  \ref{Section:NumberTheory}, \ref{Sec:LinInd}, \ref{Section:MainThm} we study the case where $d$ is odd. Section \ref{Section:FurtherExchangeGraphs} is devoted to the case where $d$ is even.
\end{rem}

We fix an integer $n\geq 1$, and put $d=2n+1$.  We construct an initial seed as follows. 

\begin{defn}[Initial triangle]
\label{Def:InitialTriangle}
\begin{itemize}
\item[(a)] We define the \emph{fundamental angle} to be 
\begin{align*}
\alpha=\frac{\pi}{2n+1}.
\end{align*}
\item[(b)] The \emph{initial triangle} $\Delta_0\subseteq \mathbb{E}^2$ is an isosceles triangle with angles $\alpha$, $n\alpha$ and $n\alpha$.
\item[(c)] The \emph{initial seed} $(\mathbf{v}_0,B_0)$ is given by the initial triangle $\Delta_0$ together with an acyclic quiver as shown in Figure \ref{Figure:Initial}. 
\end{itemize}
\end{defn}

\begin{rem} By virtue of Remark~\ref{Rem:classificationEuclidean} the choice of an initial seed in Definition \ref{Def:InitialTriangle} is not a loss of generality.
\end{rem}

\begin{figure}
\centering
\begin{minipage}{.4\textwidth}
\centering
\begin{tikzpicture}

\newcommand{\x}{1cm} %%% half the length of the base of the initial triangle
\newcommand{\h}{\x*\SinTwoPiOverFive/\CosTwoPiOverFive} %%% height of the initial triangle, numerical approx. of \sqrt{5+2\sqrt{5}}

\path[draw] (-\x,0) to (\x,0);
\path[draw] (-\x,0) to (0,\h);
\path[draw] (\x,0) to (0,\h);

%\node[above left] (A) at (0.5*\x,0.5*\h) {\tiny{$+$}};
%\node[above right] (B) at (-0.5*\x,0.5*\h) {\tiny{$+$}};
%\node[above left] (C) at (0,0) {\tiny{$+$}};

\node[above left] (B) at (-0.5*\x,0.5*\h) {\tiny{$d_1$}};
\node[above right] (B) at (0.5*\x,0.5*\h) {\tiny{$d_1$}};
%\node[below] (C) at (0,0) {\tiny{$x_1$}};

\path[->,draw,shorten >= 2pt,shorten <=2pt,>=stealth] (-0.5*\x,0.5*\h) to (0.5*\x,0.5*\h);
\path[->,draw,shorten >= 2pt,shorten <=2pt,>=stealth] (0,0) to (0.5*\x,0.5*\h) ;
\path[->,draw,shorten >= 2pt,shorten <=2pt,>=stealth] (0,0) to (-0.5*\x,0.5*\h) ;

%\node at (0,0) {\textcolor{\RefCol}{$\bullet$}};

\end{tikzpicture}
\caption{The initial seed}
\label{Figure:Initial}
\vspace{2\baselineskip}
\begin{tikzpicture}[scale=0.7]

\newcommand{\rParallel}{5cm}
\newcommand{\sParallel}{1cm}
\newcommand{\tParallel}{1.2cm}

\path[draw] (0,0) to (\rParallel,0);
\path[draw] (0,0) to (\sParallel,\tParallel);
\path[draw] (\sParallel,\tParallel) to (\rParallel+\sParallel,\tParallel);

\node[above right] (A) at (0.2cm,0) {\tiny{$\pi/m$}};

\end{tikzpicture}
\caption{A seed with parallel sides}
\label{Figure:Parallel}

\end{minipage}%
\begin{minipage}{.5\textwidth}
  \centering
\definecolor{ffqqqq}{rgb}{1,0,0}
\definecolor{uuuuuu}{rgb}{0.26666666666666666,0.26666666666666666,0.26666666666666666}\definecolor{ududff}{rgb}{0.30196078431372547,0.30196078431372547,1}
\begin{tikzpicture}[line cap=round,line join=round,>=triangle 45,x=1cm,y=1cm,scale=0.6]
\draw [line width=1pt] (5.2,6.04)-- (3.68,1.52);
\draw [line width=1pt] (6.66,1.5)-- (3.68,1.52);
\draw [line width=1pt] (5.2,6.04)-- (6.66,1.5);
\draw [line width=1pt] (3.68,1.52)-- (5.1391946671471045,-3.0199945950815232);
\draw [line width=1pt] (5.1391946671471045,-3.0199945950815232)-- (6.66,1.5);
\draw [line width=1pt] (3.68,1.52)-- (1.269694730739412,-0.23249208528160037);
\draw [line width=1pt] (5.1391946671471045,-3.0199945950815232)-- (1.269694730739412,-0.23249208528160037);
\draw [line width=1pt] (1.269694730739412,-0.23249208528160037)-- (0.3707661294541681,-3.0737466268162663);
\draw [line width=1pt] (0.3707661294541681,-3.0737466268162663)-- (5.1391946671471045,-3.0199945950815232);
\draw [line width=1pt] (1.269694730739412,-0.23249208528160037)-- (-3.4989767387438,-0.2870119544746417);
\draw [line width=1pt] (-3.4989767387438,-0.2870119544746417)-- (0.3707661294541681,-3.0737466268162663);
\draw [line width=1pt] (-3.4989767387438,-0.2870119544746417)-- (-2.038976738743802,-4.827011954474642);
\draw [line width=1pt] (-2.038976738743802,-4.827011954474642)-- (0.3707661294541681,-3.0737466268162663);
\draw [line width=1pt] (-3.4989767387438,-0.2870119544746417)-- (-5.018976738743804,-4.807011954474642);
\draw [line width=1pt] (-5.018976738743804,-4.807011954474642)-- (-2.038976738743802,-4.827011954474642);
\draw [line width=0.3pt,dash pattern=on 1pt off 1pt on 1pt off 4pt] (3.68,1.52)-- (6.374871434098983,2.386632663829189);
\draw [line width=0.3pt,dash pattern=on 1pt off 1pt on 1pt off 4pt] (6.66,1.5)-- (3.976717089988039,2.402342925490747);
\draw [line width=0.3pt,dash pattern=on 1pt off 1pt on 1pt off 4pt] (5.2,6.04)-- (5.1695973335735514,1.510002702459238);
\draw [line width=1.6pt,color=ffqqqq] (5.1695973335735514,1.510002702459238)-- (2.5*1.205+5.1696,2.5*0.877+1.51);
\node [above] at (2.5*1.205+5.1696,2.5*0.877+1.51) {$b$};
\draw [line width=0.6pt,dashed] (5.1695973335735514,1.510002702459238)-- (3.976717089988039,2.402342925490747);
\draw [line width=0.6pt,dashed] (3.976717089988039,2.402342925490747)-- (6.374871434098983,2.386632663829189);
\draw [line width=1.6pt,color=ffqqqq] (5.1695973335735514,1.510002702459238)-- (3.964847365369707,0.6337539573592014);
\draw [line width=0.6pt,dashed] (6.36299654549934,0.6172742624822333)-- (3.964847365369707,0.6337539573592014);
\draw [line width=0.6pt,dashed] (5.1695973335735514,1.510002702459238)-- (6.36299654549934,0.6172742624822333);
\draw [line width=0.6pt,dashed] (2.475173052956761,0.6439907592804615)-- (2.0253830647270834,-0.776873313408131);
\draw [line width=0.6pt,dashed] (3.964847365369707,0.6337539573592014)-- (2.475173052956761,0.6439907592804615);
\draw [line width=1.2pt,color=ffqqqq] (3.964847365369707,0.6337539573592014)-- (2.0253830647270834,-0.776873313408131);
\draw [line width=1.6pt,color=ffqqqq] (2.0253830647270834,-0.776873313408131)-- (0.8201089642016517,-1.653503274778085);
\draw [line width=0.6pt,dashed] (0.8201089642016517,-1.653503274778085)-- (1.301604444927962,-3.063253767318059);
\draw [line width=0.6pt,dashed] (1.301604444927962,-3.063253767318059)-- (2.0253830647270834,-0.776873313408131);
\draw [line width=0.6pt,dashed] (0.33840400652924657,-0.24313946344332726)-- (-0.38464100400219414,-2.5297520198781243);
\draw [line width=1.6pt,color=ffqqqq] (-0.38464100400219414,-2.5297520198781243)-- (0.8201089642016517,-1.653503274778085);
\draw [line width=0.6pt,dashed] (0.8201089642016517,-1.653503274778085)-- (0.33840400652924657,-0.24313946344332726);
\draw [line width=1.2pt,color=ffqqqq] (-2*1.939-0.385,-2*1.411-2.53)-- (-0.38464100400219414,-2.5297520198781243);
\draw [line width=0.6pt,dashed] (-2.324105304644819,-3.9403792906454536)-- (-0.8337796930510591,-3.950142384241392);
\draw [line width=0.6pt,dashed] (-0.8337796930510591,-3.950142384241392)-- (-0.38464100400219414,-2.5297520198781243);
\begin{scriptsize}
\draw [fill=ududff] (3.68,1.52) circle (2.5pt);
\draw [fill=ududff] (6.66,1.5) circle (2.5pt);
\draw [fill=ududff] (5.2,6.04) circle (2.5pt);
\draw [fill=ududff] (5.1391946671471045,-3.0199945950815232) circle (2.5pt);
\draw [fill=ududff] (1.269694730739412,-0.23249208528160037) circle (2.5pt);
\draw [fill=ududff] (0.3707661294541681,-3.0737466268162663) circle (2.5pt);
\draw [fill=ududff] (-3.4989767387438,-0.2870119544746417) circle (2.5pt);
\draw [fill=ududff] (-2.038976738743802,-4.827011954474642) circle (2.5pt);
\draw [fill=ududff] (-5.018976738743804,-4.807011954474642) circle (2.5pt);
\draw [fill=uuuuuu] (6.374871434098983,2.386632663829189) circle (2pt);
\draw [fill=uuuuuu] (5.1695973335735514,1.510002702459238) circle (2pt);
\draw [fill=uuuuuu] (3.976717089988039,2.402342925490747) circle (2pt);
\draw [fill=uuuuuu] (6.36299654549934,0.6172742624822333) circle (2pt);
\draw [fill=uuuuuu] (3.964847365369707,0.6337539573592014) circle (2pt);
\draw [fill=uuuuuu] (2.475173052956761,0.6439907592804615) circle (2pt);
\draw [fill=uuuuuu] (2.0253830647270834,-0.776873313408131) circle (2pt);
\draw [fill=uuuuuu] (0.8201089642016517,-1.653503274778085) circle (2pt);
\draw [fill=uuuuuu] (1.301604444927962,-3.063253767318059) circle (2pt);
\draw [fill=uuuuuu] (0.33840400652924657,-0.24313946344332726) circle (2pt);
\draw [fill=uuuuuu] (-0.38464100400219414,-2.5297520198781243) circle (2pt);
\draw [fill=uuuuuu] (-2.324105304644819,-3.9403792906454536) circle (2pt);
\draw [fill=uuuuuu] (-0.8337796930510591,-3.950142384241392) circle (2pt);
\end{scriptsize}
\end{tikzpicture}
\caption{The line $b$}
\label{Figure:Line}
\end{minipage}
\end{figure}

\subsection{The initial acyclic belt}
\label{Sec:InitialAcyclicBelt}

\begin{defn}[Acyclic belts]
\begin{enumerate}
\item An \emph{acyclic belt} is a maximal connected full subgraph of the exchange graph containing only vertices $(\mathbf{v},B)$ with acyclic exchange matrices $B$. 
\item The acyclic belt obtained from $(\mathbf{v}_0,B_0)$ by sequences of mutations at sinks or sources is called the \emph{initial acyclic belt}. It is denoted by $I$ and we label the vertices of $I$ so that $I_0=(\mathbf{v}_0,B_0)$ and that $I_{n+1}=(\mathbf{v}_{n+1},B_{n+1})$ is obtained from $I_n=(\mathbf{v}_n,B_n)$ by a mutation at a source for all $n\in\mathbb{Z}$. 
\end{enumerate}
\end{defn}

\begin{ex}[Type $\widetilde{A}_2$]
Let us look at the case $n=1$. In this case, the matrix $B$ has integer entries and it defines a cluster algebra of type $\tilde{A}_2$. Its exchange graph is called the \emph{brick wall} and it is shown in Figure \ref{Figure:Classic}. Moreover, it is known how to visualise the seeds by equilateral triangles in the Euclidean plane, see Figure \ref{Figure:ClassicDual}. Note that there is exactly one acyclic belt, and this acyclic belt consists of equilateral triangles.  
\end{ex}

\begin{defn}[Belt line $b$] Define $b\subseteq\mathbb{E}^2$ to be the line connecting the two feet of the altitudes in the initial triangle corresponding to the sink and the source in the initial quiver $Q_0$. The line is shown in Figure \ref{Figure:Line}.
\end{defn}

\begin{rem}[Billiard]
The line $b$ is well-studied in Euclidean geometry. It plays an important role in the solution of Fagnano's problem, which is related to triangular billiards. The problem asks to find three points on the three sides of a given triangle such that the perimeter of the triangle formed by the three points is as small as possible. For an acute-angled triangle, the minimum is attained when the points are the feet of the altitudes of the triangle, and the perimeter of the minimal triangle forms a $3$-periodic billiard orbit, see Tabachnikov~\cite[Chapter 7]{T}. The name $b$ stands both for acyclic \emph{belt} and \emph{billiard}.   
\end{rem}

Now, we need to find a compatible choice of a reference point.

\begin{prop}
\label{Prop:ReferencePoint}
The only reference point $u$ satisfying the compatibility condition \ref{Condition:Pos1} lies infinitely far away on line $b$, in the direction of the arrow from the source to the sink in the initial triangle. 
\end{prop}

\begin{figure}
\centering
\begin{minipage}{.45\textwidth}
\begin{center}
\begin{tikzpicture}
\newcommand{\XX}{1.3cm};
\newcommand{\HH}{1cm};
\newcommand{\GG}{0.5cm};

\node at (-\XX,0) {$\bullet$};
\node at (\XX,0) {$\bullet$};
\node at (0,3.078*\XX) {$\bullet$};

\path[draw,-] (-\XX,0) to (0,3.078*\XX);
\path[draw,-] (\XX,0) to (0,3.078*\XX);
\path[draw,-] (-\XX,0) to (\XX,0);

\path[,draw,-latex,shorten >=0.1cm,shorten <=0.1cm] (0,0) to (0.5*\XX,1.539*\XX);
\path[,draw,-latex,shorten >=0.1cm,shorten <=0.1cm] (0,0) to (-0.5*\XX,1.539*\XX);
\path[,draw,-latex,shorten >=0.1cm,shorten <=0.1cm] (-0.5*\XX,1.539*\XX) to (0.5*\XX,1.539*\XX);

\fill[pattern color=red!50!,pattern=vertical lines] (-\XX,0) -- (\XX+\HH,0) -- (\XX+\HH-\GG,-3.078*\GG) -- (-\XX-\GG,-3.078*\GG); 
\draw[opacity=.4] (-\XX-\GG,-3.078*\GG) -- (-\XX,0) -- (\XX+\HH,0);

\fill[pattern color=blue!50!,pattern=horizontal lines] (\XX,0) -- (-\XX-\HH,0) -- (-\XX-\HH+\GG,-3.078*\GG) -- (\XX+\GG,-3.078*\GG); 
\draw[opacity=.4] (\XX+\GG,-3.078*\GG) -- (\XX,0) -- (-\XX-\HH,0);

\fill[pattern color=green!50!,pattern=north east lines] (-\XX-\GG,-3.078*\GG) -- (0,3.078*\XX) -- (-\GG,3.078*\XX+3.078*\GG) -- (-\XX-2*\GG,0) ; 
\draw[opacity=.4] (0,3.078*\XX) -- (-\GG,3.078*\XX+3.078*\GG);

\end{tikzpicture}
\caption{Locus of references points not compatible with the initial seed}
\label{Figure:Forbidden}
\end{center}
\end{minipage}
\begin{minipage}{.5\textwidth}
\centering
\begin{tikzpicture}

\newcommand{\n}{4} %% number of bricks
\newcommand{\s}{0.7} 
\newcommand{\y}{0.5} 
\renewcommand{\d}{0.2} 

\path[draw,-] (-\d,0) to (2*\n*\s+\d,0);
\path[draw,-] (-\d,\y) to (2*\n*\s+\d,\y);
\path[draw,-] (-\d,2*\y) to (2*\n*\s+\d,2*\y);

\foreach \x in {0,...,\n}
{
\node at (2*\x*\s,0) {$\bullet$};
\node at (2*\x*\s,\y) {$\bullet$};
\path[draw,-] (2*\x*\s,0) to (2*\x*\s,\y); 
}

\foreach \x in {1,...,\n}
{
\node at (2*\x*\s-\s,\y) {$\bullet$};
\node at (2*\x*\s-\s,2*\y) {$\bullet$};
\path[draw,-] (2*\x*\s-\s,\y) to (2*\x*\s-\s,2*\y); 
}

\end{tikzpicture}
\caption{The exchange graph of a cluster algebra of type $A_2^{(1)}$}
\label{Figure:Classic}
\vspace{0.5\baselineskip}
\begin{tikzpicture}

\newcommand{\numbertriangles}{8} %% number of triangles
\newcommand{\sidelength}{0.5} %% side length triangle
\newcommand{\lengthinf}{1.5} %% length of (vertical) infinity-gons
\newcommand{\lengthinfhorizontal}{3.5} %% length of horizontal infinity-gons

\newcommand{\height}{0.886*\sidelength} %% height of triangles

\path[draw,-] (0,0) to (0.5*\height*\numbertriangles,\height*\numbertriangles);
\path[draw,-] (-\sidelength,0) to (0.5*\height*\numbertriangles-\sidelength,\height*\numbertriangles);

\foreach \x in {0,...,\numbertriangles}
{
\path[draw,-] (0.5*\height*\x,\height*\x) to (-\lengthinfhorizontal+0.5*\height*\x-\sidelength,\height*\x);
\path[draw,-] (0.5*\height*\x-\sidelength,\height*\x) to (0.5*\height*\x-\sidelength+0.5*\sidelength+0.5*\lengthinf,\height*\x-0.866*\sidelength-0.866*\lengthinf);
}

\path[draw,-] (0.5*\height*\numbertriangles,\height*\numbertriangles) to (0.5*\height*\numbertriangles+0.5*\lengthinf,\height*\numbertriangles-0.866*\lengthinf);

\end{tikzpicture}
\caption{Geometric description of the seeds of type $A_2^{(1)}$}
\label{Figure:ClassicDual}
\end{minipage}
\end{figure}

\begin{proof}
We consider the initial seed from Figure \ref{Figure:Initial}. Condition \ref{Condition:Pos1} requires $u$ to be compatible with the seed at every index, that is, at all angles of the initial triangle $\Delta_0$. Remark \ref{Rem:ForbiddenRegions} allows us to determine the loci of compatibility for each angle explicitly. The loci are given by sectors around the $3$ angles; the sectors of non-compatible points are coloured green, blue, and red in Figure \ref{Figure:Forbidden}.

Starting with the initial seed we perform a sequence of mutations at sources. We label the initial seed $0$ and occurring further seeds by $-1,-2,\ldots$, see Figure \ref{Figure:InitialBelt}. The discussion in the previous paragraph about the location of $u$ implies that each of the first $6$ source mutations is a reflection of a triangle across a side.

The composition of two reflections across two lines intersecting each other at an angle $\varphi$ is a rotation with angle $2\varphi$. The composition of three rotations with angles $2\alpha$, $2n\alpha$ and $2n\alpha$ is a translation due to $2\alpha+2n\alpha+2n\alpha=2\pi$. Hence, after $6$ steps we obtain a triangle $\Delta_{-6}$ which can be obtained from $\Delta_0$ by a translation with a vector $w$. We obtain regions of compatibility for $\Delta_{-6}$ similar to those $\Delta_0$ in Figure \ref{Figure:Forbidden} but shifted by vector $w$. Analogously, we can continue to mutate at sources, and obtain similar regions of compatibility for every sixth triangle. 

It is well-known in the literature that the line $b$ intersects the occuring triangles in the feet of two altitudes, see for example Section 4.5 about Fagnano's theorem in Coxeter--Greitzer's book \cite{CG}. From this we can conclude that the translation vector $w$ is parallel to $b$.

A tracking of angles shows that the base sides of the triangles $\Delta_{-1}$ and $\Delta_{-4}$ in seeds $-1$ and $-4$ are parallel to $b$. The compatibility of the reference point of the seeds $\Delta_{-1-6k}$ and $\Delta_{-4-6k}$ for all $k\geq 0$ implies that $u$ must lie between the two lines defined by the base sides of the triangles $\Delta_{-1}$ and $\Delta_{-4}$. This implies that $u$ lies infinitely far away on the line $b$.
\end{proof}

The initial acyclic belt is shown in Figure \ref{Figure:InitialBelt} for the example $n=2$ (that is, $\alpha=\tfrac{\pi}{5}$).

\begin{rem} We can describe the notion of positivity given by the reference point $u$ from Proposition \ref{Prop:ReferencePoint} in geometric terms. Suppose that a seed $(\mathbf{v},B)$ is given by a triangle $\Delta$. Let $s$ be a side of $\Delta$, and without loss of generality let us assume that $s$ corresponds to the component $v_1$ of $\mathbf{v}=(v_1,v_2,v_3)$. The extension of $s$ divides $\mathbb{E}^2$ into two half-planes. Let $z\in\mathbb{E}^2$ be a vector perpendicular to $s$ such that lies in a different half-plane than $\Delta$. Then it is easy to see that there exists a vector $b^{\perp}\in\mathbb{E}^2$ having the property that $v_1$ (and hence the mutation of $(\mathbf{v},B)$ at $v_1$) is positive if and only if $(b^{\perp},z)>0$. Here $(\cdot,\cdot)\, \colon\, \mathbb{E}^2\times \mathbb{E}^2\to \mathbb{R}$ denotes the standard scalar product. By construction, the line $b^{\perp}$ is orthogonal to the line $b$.  
\end{rem}

%The initial triangle $\Delta_0$ intersects the line $b$ in two points. We denote by $P$ the intersection of $b$ and the base the triangle; we denote by $Q$ the intersection of $b$ and the on of the legs of the triangle. Suppose that a half-plane $\Pi\subseteq \mathbb{E}^2$ is bounded by a $l\subseteq \mathbb{E}^2$ that is not parallel to $b$. Then $\Pi\cap b$ is a ray $r$. We say that $\Pi$ is \emph{positive} if the points on $r$ that are far away from the origin are closer to $Q$ than to $P$. We say it is \emph{negative} otherwise. In other words, when $\Pi=\{\,v\in \mathbb{E}^2\mid (v,w)<0\,\}$ for some $w\in\mathbb{E}^2$ (orthogonal to $l$) and $u=P-Q$ (parallel to $b$), then $\Pi$ is positive if and only if $(u,w)>0$. Furthermore, if $l$ is parallel to $b$ we say that $\Pi$ is positive if and only if $b\subseteq \Pi$.
%
%\begin{prop} Suppose that the seed $(\mathbf{v},B)$ is given by a triangle or an infinite region $\Delta$. Moreover, assume that $\Pi$ is one of its defining half-planes. Then the mutation $\mu_H(\Delta)$ is positive if and only if $\Pi$ is positive. 
%\end{prop}
%
%\begin{proof} To be written.
%\end{proof}

The discussion in the proof of Proposition \ref{Prop:ReferencePoint} shows the following.

\begin{prop} 
\label{Prop:InitialBelt}
The initial acyclic belt has the following geometric structure.
\begin{enumerate}
\item The domain of every geometric realisation $I_n$ with $n\in\mathbb{Z}$ inside the initial acyclic belt is an acute-angled triangle. The line $b$ intersects every triangle $I_n$ in two points. These points are the feet of two altitudes in $I_n$. 
\item Suppose that $n\in\mathbb{Z}$. The domains of the geometric realisation $I_n$ and $I_{n+6}$ are obtained from each other by a translation by a vector parallel to $b$.
\end{enumerate}
\end{prop}

\begin{rem}
Proposition \ref{Prop:InitialBelt} is true for every other choice of an acyclic seed (given by an acute-angled initial triangle and an acyclic quiver) when we choose a reference point infinitely far away on the line through the feet of the altitudes corresponding to the sink and the source. Hence it is true for every acyclic belt. \emph{A priori}, two different acyclic belts could yield two distinct lines $b$ and $b'$. Moreover, two different directions of the lines could result in two conflicting definitions of the reference point (and thus, conflicting notions of positivity), and \emph{a priori} the length of the translation vectors in Proposition~\ref{Prop:InitialBelt} (2) could depend on the acyclic seed. However, we will show in the next section that the direction of the line $b$, the line $b$ itself, and the length of the translation vector are independent of the choice of the acyclic belt. 
\end{rem}

\begin{figure}
\begin{center}
\begin{tikzpicture}

%%%% Preamble
\newcommand{\x}{1cm} %%% half the length of the base of the initial triangle
\newcommand{\dist}{0cm} %%% distance between triangles at lazy mutations

\newcommand{\h}{\x*\SinTwoPiOverFive/\CosTwoPiOverFive} %%% height of the initial triangle
\newcommand{\xa}{-2*\x} %%% x-coordinate of vector a
\newcommand{\ya}{0} %%% y-coordinate of vector a
\newcommand{\xb}{-\x*2*\CosPiOverFive} %%% x-coordinate of vector b
\newcommand{\yb}{-\x*2*\SinPiOverFive} %%% y-coordinate of vector b
\newcommand{\xc}{-\x*2*\CosTwoPiOverFive} %%% x-coordinate of vector c
\newcommand{\yc}{-\x*2*\SinTwoPiOverFive} %%% y-coordinate of vector c

%%%% Initial triangle
\path[draw,thick] (-\x,0) to (\x,0);
\path[draw,thick] (-\x,0) to (0,\h);
\path[draw,thick] (\x,0) to (0,\h);

%\node[above left] (A) at (0.5*\x,0.5*\h) {\tiny{$+$}};
%\node[above right] (B) at (-0.5*\x,0.5*\h) {\tiny{$+$}};
%\node[above left] (C) at (0,0) {\tiny{$+$}};

\path[->,draw,shorten >= 2pt,shorten <=2pt,>=stealth] (-0.5*\x,0.5*\h) to (0.5*\x,0.5*\h);
\path[->,draw,shorten >= 2pt,shorten <=2pt,>=stealth] (0,0) to (0.5*\x,0.5*\h);
\path[->,draw,shorten >= 2pt,shorten <=2pt,>=stealth] (0,0) to (-0.5*\x,0.5*\h);

%\node at (0,0) {\textcolor{\RefCol}{$\bullet$}};
\node[above right] at (-\x,0) {\textcolor{\NumCol}{\tiny{$0$}}};

%%%% Triangle 1
\path[draw,thick] (-\x,0) to (0,-\h);
\path[draw,thick] (\x,0) to (0,-\h);

%\node[above left] (A) at (0.5*\x,-0.5*\h) {\tiny{$+$}};
%\node[above right] (B) at (-0.5*\x,-0.5*\h) {\tiny{$+$}};
%\node[below left] (C) at (0,0) {\tiny{$+$}};

\path[->,draw,shorten >= 2pt,shorten <=2pt,>=stealth] (-0.5*\x,-0.5*\h) to (0.5*\x,-0.5*\h);
\path[->,draw,shorten >= 2pt,shorten <=2pt,>=stealth] (0.5*\x,-0.5*\h) to (0,0);
\path[->,draw,shorten >= 2pt,shorten <=2pt,>=stealth] (-0.5*\x,-0.5*\h) to (0,0);

\node[below right] at (-\x,0) {\textcolor{\NumCol}{\tiny{$1$}}};

%%%% Triangle 2
\path[draw,thick] (-\x,0) to (-\x+\xb,\yb);
\path[draw,thick]  (-\x+\xb,\yb) to (0,-\h);

%\node[right] (A) at (-\x+0.5*\xb,0.5*\yb) {\tiny{$+$}};
%\node[right] (B) at  (-0.5*\x+0.5*\xb,0.5*\yb-0.5*\h) {\tiny{$+$}};
%\node[left] (C) at (-0.5*\x,-0.5*\h) {\tiny{$+$}};

\path[->,draw,shorten >= 2pt,shorten <=2pt,>=stealth] (-\x+0.5*\xb,0.5*\yb) to (-0.5*\x,-0.5*\h);
\path[->,draw,shorten >= 2pt,shorten <=2pt,>=stealth] (-0.5*\x+0.5*\xb,0.5*\yb-0.5*\h) to (-0.5*\x,-0.5*\h);
\path[->,draw,shorten >= 2pt,shorten <=2pt,>=stealth] (-0.5*\x+0.5*\xb,0.5*\yb-0.5*\h) to (-\x+0.5*\xb,0.5*\yb);

\node[right] at (-\x+\xb,\yb) {\textcolor{\NumCol}{\tiny{$2$}}};

%%%% Triangle 3
\path[draw,thick] (-\x+\xb,\yb) to (-\x+\xb+\xc,\yb+\yc);
\path[draw,thick]  (0,-\h) to (-\x+\xb+\xc,\yb+\yc);

%\node[below left] (A) at (-0.5*\x+0.5*\xb,-0.5*\h+0.5*\yb) {\tiny{$+$}};
%\node[above right] (B) at  (-\x+\xb+0.5*\xc,\yb+0.5*\yc) {\tiny{$+$}};
%\node[above right] (C) at (-0.5*\x+0.5*\xb+0.5*\xc,-0.5*\h+0.5*\yb+0.5*\yc) {\tiny{$+$}};

\path[->,draw,shorten >= 2pt,shorten <=2pt,>=stealth] (-\x+\xb+0.5*\xc,\yb+0.5*\yc) to (-0.5*\x+0.5*\xb,-0.5*\h+0.5*\yb);
\path[->,draw,shorten >= 2pt,shorten <=2pt,>=stealth] (-0.5*\x+0.5*\xb+0.5*\xc,-0.5*\h+0.5*\yb+0.5*\yc) to (-0.5*\x+0.5*\xb,-0.5*\h+0.5*\yb);
\path[->,draw,shorten >= 2pt,shorten <=2pt,>=stealth] (-\x+\xb+0.5*\xc,\yb+0.5*\yc) to (-0.5*\x+0.5*\xb+0.5*\xc,-0.5*\h+0.5*\yb+0.5*\yc);

\node[above right] at (-\x+\xb+\xc,\yb+\yc) {\textcolor{\NumCol}{\tiny{$3$}}};

%%%% Triangle 4
\path[draw,thick] (-\x+\xb,\yb) to (-2*\x+2*\xb+\xc,\yb);
\path[draw,thick] (-2*\x+2*\xb+\xc,\yb) to (-\x+\xb+\xc,\yb+\yc);

%\node[below] (A) at (-1.5*\x+1.5*\xb-0.5*\xc,\yb) {\tiny{$+$}};
%\node[below left] (B) at  (-\x+\xb+0.5*\xc,\yb+0.5*\yc) {\tiny{$+$}};
%\node[above] (C) at (-1.5*\x+1.5*\xb+\xc,\yb+0.5*\yc) {\tiny{$+$}};

\path[->,draw,shorten >= 2pt,shorten <=2pt,>=stealth] (-1.5*\x+1.5*\xb+\xc,\yb+0.5*\yc) to (-1.5*\x+1.5*\xb-0.5*\xc,\yb);
\path[->,draw,shorten >= 2pt,shorten <=2pt,>=stealth] (-1.5*\x+1.5*\xb+\xc,\yb+0.5*\yc) to (-\x+\xb+0.5*\xc,\yb+0.5*\yc);
\path[->,draw,shorten >= 2pt,shorten <=2pt,>=stealth] (-1.5*\x+1.5*\xb-0.5*\xc,\yb) to (-\x+\xb+0.5*\xc,\yb+0.5*\yc);

\node[below left] at (-\x+\xb,\yb) {\textcolor{\NumCol}{\tiny{$4$}}};

%%%% Triangle 5
\path[draw,thick] (-\x+2*\xb+\xc,2*\yb+\yc) to (-2*\x+\xc+2*\xb,\yb);
\path[draw,thick] (-\x+2*\xb+\xc,2*\yb+\yc) to (-\x+\xb+\xc,\yb+\yc);

%\node[right] (A) at (-1.5*\x+2*\xb+\xc,1.5*\yb+0.5*\yc) {\tiny{$+$}};
%\node[left] (B) at  (-\x+1.5*\xb+\xc,1.5*\yb+\yc) {\tiny{$+$}};
%\node[left] (C) at (-1.5*\x+1.5*\xb+\xc,\yb+0.5*\yc) {\tiny{$+$}};

\path[->,draw,shorten >= 2pt,shorten <=2pt,>=stealth]  (-1.5*\x+2*\xb+\xc,1.5*\yb+0.5*\yc) to (-\x+1.5*\xb+\xc,1.5*\yb+\yc);
\path[->,draw,shorten >= 2pt,shorten <=2pt,>=stealth]  (-1.5*\x+2*\xb+\xc,1.5*\yb+0.5*\yc) to (-1.5*\x+1.5*\xb+\xc,\yb+0.5*\yc);
\path[->,draw,shorten >= 2pt,shorten <=2pt,>=stealth] (-\x+1.5*\xb+\xc,1.5*\yb+\yc) to (-1.5*\x+1.5*\xb+\xc,\yb+0.5*\yc);

\node[left] at (-\x+\xb+\xc,\yb+\yc) {\textcolor{\NumCol}{\tiny{$5$}}};

%%%% Triangle 6
\path[draw,thick] (-\x+\xa+2*\xb+\xc,0+\ya+2*\yb+\yc) to (\x+\xa+2*\xb+\xc,0+\ya+2*\yb+\yc);
\path[draw,thick] (-\x+\xa+2*\xb+\xc,0+\ya+2*\yb+\yc) to (0+\xa+2*\xb+\xc,\h+\ya+2*\yb+\yc);

%\node[above left] (A) at (0.5*\x+\xa+2*\xb+\xc,0.5*\h+\ya+2*\yb+\yc) {\tiny{$+$}};
%\node[above right] (B) at (-0.5*\x+\xa+2*\xb+\xc,0.5*\h+\ya+2*\yb+\yc) {\tiny{$+$}};
%\node[above left] (C) at (0+\xa+2*\xb+\xc,0+\ya+2*\yb+\yc) {\tiny{$+$}};

\path[->,draw,shorten >= 2pt,shorten <=2pt,>=stealth] (-0.5*\x+\xa+2*\xb+\xc,0.5*\h+\ya+2*\yb+\yc) to (0.5*\x+\xa+2*\xb+\xc,0.5*\h+\ya+2*\yb+\yc);
\path[->,draw,shorten >= 2pt,shorten <=2pt,>=stealth] (0+\xa+2*\xb+\xc,0+\ya+2*\yb+\yc) to (0.5*\x+\xa+2*\xb+\xc,0.5*\h+\ya+2*\yb+\yc);
\path[->,draw,shorten >= 2pt,shorten <=2pt,>=stealth] (0+\xa+2*\xb+\xc,0+\ya+2*\yb+\yc) to (-0.5*\x+\xa+2*\xb+\xc,0.5*\h+\ya+2*\yb+\yc);

\node[above right] at (-\x+\xa+2*\xb+\xc,0+\ya+2*\yb+\yc) {\textcolor{\NumCol}{\tiny{$6$}}};

\path[draw,thick] (-\x-\xb-\xa+3*\dist,-\yb-\ya) to (-2*\x-\xa+3*\dist,-\yb-\ya-\yc);
\path[draw,thick] (-\x-\xa+3*\dist,-\ya) to (-\x-\xb-\xa+3*\dist,-\yb-\ya);

%\node[right] (A) at (-1.5*\x-\xa+3*\dist,\ya-0.5*\yb-0.5*\yc) {\tiny{$-$}};
%\node[left] (B) at  (-\x-0.5*\xb-\xa+3*\dist,-0.5*\yb-\ya) {\tiny{$-$}};
%\node[left] (C) at (-1.5*\x-0.5*\xb-\xa+3*\dist,-\ya-\yb-0.5*\yc) {\tiny{$-$}};

\path[->,draw,shorten >= 2pt,shorten <=2pt,>=stealth] (-\x-0.5*\xb-\xa+3*\dist,-0.5*\yb-\ya) to (-1.5*\x-0.5*\xb-\xa+3*\dist,-\ya-\yb-0.5*\yc);
\path[->,draw,shorten >= 2pt,shorten <=2pt,>=stealth] (-1.5*\x-\xa+3*\dist,\ya-0.5*\yb-0.5*\yc) to (-1.5*\x-0.5*\xb-\xa+3*\dist,-\ya-\yb-0.5*\yc);
\path[->,draw,shorten >= 2pt,shorten <=2pt,>=stealth] (-1.5*\x-\xa+3*\dist,\ya-0.5*\yb-0.5*\yc) to (-\x-0.5*\xb-\xa+3*\dist,-0.5*\yb-\ya);

\node[left] at (-\x-0.5*\xb-\xa+3*\dist,-0.5*\yb-\ya) {\textcolor{\NumCol}{\tiny{$-1$}}};

%%%% Triangle -5
\path[draw,thick] (-\x+\xb-\xa-2*\xb-\xc+3*\dist,\yb-\ya-2*\yb-\yc) to (-\x+\xb+\xc-\xa-2*\xb-\xc+3*\dist,\yb+\yc-\ya-2*\yb-\yc);
\path[draw,thick] (-2*\x+2*\xb+\xc-\xa-2*\xb-\xc+3*\dist,\yb-\ya-2*\yb-\yc) to (-\x+\xb-\xa-2*\xb-\xc+3*\dist,\yb-\ya-2*\yb-\yc);

%\node[below] (A) at (-1.5*\x+1.5*\xb-0.5*\xc-\xa-2*\xb-\xc+3*\dist,\yb-\ya-2*\yb-\yc) {\tiny{$-$}};
%\node[below left] (B) at  (-\x+\xb+0.5*\xc-\xa-2*\xb-\xc+3*\dist,\yb+0.5*\yc-\ya-2*\yb-\yc) {\tiny{$-$}};
%\node[above] (C) at (-1.5*\x+1.5*\xb+\xc-\xa-2*\xb-\xc+3*\dist,\yb+0.5*\yc-\ya-2*\yb-\yc) {\tiny{$-$}};

\path[->,draw,shorten >= 2pt,shorten <=2pt,>=stealth] (-1.5*\x+1.5*\xb+\xc-\xa-2*\xb-\xc+3*\dist,\yb+0.5*\yc-\ya-2*\yb-\yc) to (-1.5*\x+1.5*\xb-0.5*\xc-\xa-2*\xb-\xc+3*\dist,\yb-\ya-2*\yb-\yc);
\path[->,draw,shorten >= 2pt,shorten <=2pt,>=stealth] (-1.5*\x+1.5*\xb+\xc-\xa-2*\xb-\xc+3*\dist,\yb+0.5*\yc-\ya-2*\yb-\yc) to (-\x+\xb+0.5*\xc-\xa-2*\xb-\xc+3*\dist,\yb+0.5*\yc-\ya-2*\yb-\yc);
\path[->,draw,shorten >= 2pt,shorten <=2pt,>=stealth] (-1.5*\x+1.5*\xb-0.5*\xc-\xa-2*\xb-\xc+3*\dist,\yb-\ya-2*\yb-\yc) to (-\x+\xb+0.5*\xc-\xa-2*\xb-\xc+3*\dist,\yb+0.5*\yc-\ya-2*\yb-\yc);

\node[below left] at (-\x+\xb-\xa-2*\xb-\xc+3*\dist,\yb-\ya-2*\yb-\yc) {\textcolor{\NumCol}{\tiny{$-2$}}};

%%%% Triangle -6
\path[draw,thick] (0-\xa-2*\xb-\xc+3*\dist,-\h-\ya-2*\yb-\yc) to (-\x+\xb+\xc-\xa-2*\xb-\xc+3*\dist,\yb+\yc-\ya-2*\yb-\yc);
\path[draw,thick]  (0-\xa-2*\xb-\xc+3*\dist,-\h-\ya-2*\yb-\yc) to (-\x+\xb-\xa-2*\xb-\xc+3*\dist,\yb-\ya-2*\yb-\yc);

%\node[below left] (A) at (-0.5*\x+0.5*\xb-\xa-2*\xb-\xc+3*\dist,-0.5*\h+0.5*\yb-\ya-2*\yb-\yc) {\tiny{$-$}};
%\node[above right] (B) at  (-\x+\xb+0.5*\xc-\xa-2*\xb-\xc+3*\dist,\yb+0.5*\yc-\ya-2*\yb-\yc) {\tiny{$-$}};
%\node[above right] (C) at (-0.5*\x+0.5*\xb+0.5*\xc-\xa-2*\xb-\xc+3*\dist,-0.5*\h+0.5*\yb+0.5*\yc-\ya-2*\yb-\yc) {\tiny{$-$}};

\path[->,draw,shorten >= 2pt,shorten <=2pt,>=stealth] (-\x+\xb+0.5*\xc-\xa-2*\xb-\xc+3*\dist,\yb+0.5*\yc-\ya-2*\yb-\yc) to (-0.5*\x+0.5*\xb-\xa-2*\xb-\xc+3*\dist,-0.5*\h+0.5*\yb-\ya-2*\yb-\yc);
\path[->,draw,shorten >= 2pt,shorten <=2pt,>=stealth] (-0.5*\x+0.5*\xb+0.5*\xc-\xa-2*\xb-\xc+3*\dist,-0.5*\h+0.5*\yb+0.5*\yc-\ya-2*\yb-\yc) to (-0.5*\x+0.5*\xb-\xa-2*\xb-\xc+3*\dist,-0.5*\h+0.5*\yb-\ya-2*\yb-\yc);
\path[->,draw,shorten >= 2pt,shorten <=2pt,>=stealth] (-\x+\xb+0.5*\xc-\xa-2*\xb-\xc+3*\dist,\yb+0.5*\yc-\ya-2*\yb-\yc) to (-0.5*\x+0.5*\xb+0.5*\xc-\xa-2*\xb-\xc+3*\dist,-0.5*\h+0.5*\yb+0.5*\yc-\ya-2*\yb-\yc);

\node[above right] at (-\x+\xb+\xc-\xa-2*\xb-\xc+3*\dist,\yb+\yc-\ya-2*\yb-\yc) {\textcolor{\NumCol}{\tiny{$-3$}}};

%%%% Triangle -7
\path[draw,thick] (-\x-\xa-2*\xb-\xc+3*\dist,0-\ya-2*\yb-\yc) to (0-\xa-2*\xb-\xc+3*\dist,-\h-\ya-2*\yb-\yc);
\path[draw,thick]  (-\x+\xb-\xa-2*\xb-\xc+3*\dist,\yb-\ya-2*\yb-\yc) to (-\x-\xa-2*\xb-\xc+3*\dist,0-\ya-2*\yb-\yc);

%\node[right] (A) at (-\x+0.5*\xb-\xa-2*\xb-\xc+3*\dist,0.5*\yb-\ya-2*\yb-\yc) {\tiny{$-$}};
%\node[right] (B) at  (-0.5*\x+0.5*\xb-\xa-2*\xb-\xc+3*\dist,0.5*\yb-0.5*\h-\ya-2*\yb-\yc) {\tiny{$-$}};
%\node[left] (C) at (-0.5*\x-\xa-2*\xb-\xc+3*\dist,-0.5*\h-\ya-2*\yb-\yc) {\tiny{$-$}};

\path[->,draw,shorten >= 2pt,shorten <=2pt,>=stealth] (-\x+0.5*\xb-\xa-2*\xb-\xc+3*\dist,0.5*\yb-\ya-2*\yb-\yc) to (-0.5*\x-\xa-2*\xb-\xc+3*\dist,-0.5*\h-\ya-2*\yb-\yc);
\path[->,draw,shorten >= 2pt,shorten <=2pt,>=stealth] (-0.5*\x+0.5*\xb-\xa-2*\xb-\xc+3*\dist,0.5*\yb-0.5*\h-\ya-2*\yb-\yc) to (-0.5*\x-\xa-2*\xb-\xc+3*\dist,-0.5*\h-\ya-2*\yb-\yc);
\path[->,draw,shorten >= 2pt,shorten <=2pt,>=stealth] (-0.5*\x+0.5*\xb-\xa-2*\xb-\xc+3*\dist,0.5*\yb-0.5*\h-\ya-2*\yb-\yc) to (-\x+0.5*\xb-\xa-2*\xb-\xc+3*\dist,0.5*\yb-\ya-2*\yb-\yc);

\node[right] at (-\x+\xb-\xa-2*\xb-\xc+3*\dist,\yb-\ya-2*\yb-\yc) {\textcolor{\NumCol}{\tiny{$-4$}}};

%%%% Triangle -8
\path[draw,thick] (-\x-\xa-2*\xb-\xc+3*\dist,0-\ya-2*\yb-\yc) to (\x-\xa-2*\xb-\xc+3*\dist,0-\ya-2*\yb-\yc);
\path[draw,thick] (\x-\xa-2*\xb-\xc+3*\dist,0-\ya-2*\yb-\yc) to (0-\xa-2*\xb-\xc+3*\dist,-\h-\ya-2*\yb-\yc);

%\node[above left] (A) at (0.5*\x-\xa-2*\xb-\xc+3*\dist,-0.5*\h-\ya-2*\yb-\yc) {\tiny{$-$}};
%\node[above right] (B) at (-0.5*\x-\xa-2*\xb-\xc+3*\dist,-0.5*\h-\ya-2*\yb-\yc) {\tiny{$-$}};
%\node[below left] (C) at (0-\xa-2*\xb-\xc+3*\dist,0-\ya-2*\yb-\yc) {\tiny{$-$}};

\path[->,draw,shorten >= 2pt,shorten <=2pt,>=stealth] (-0.5*\x-\xa-2*\xb-\xc+3*\dist,-0.5*\h-\ya-2*\yb-\yc) to (0.5*\x-\xa-2*\xb-\xc+3*\dist,-0.5*\h-\ya-2*\yb-\yc);
\path[->,draw,shorten >= 2pt,shorten <=2pt,>=stealth] (0.5*\x-\xa-2*\xb-\xc+3*\dist,-0.5*\h-\ya-2*\yb-\yc) to (0-\xa-2*\xb-\xc+3*\dist,0-\ya-2*\yb-\yc);
\path[->,draw,shorten >= 2pt,shorten <=2pt,>=stealth] (-0.5*\x-\xa-2*\xb-\xc+3*\dist,-0.5*\h-\ya-2*\yb-\yc) to (0-\xa-2*\xb-\xc+3*\dist,0-\ya-2*\yb-\yc);

\node[below right] at (-\x-\xa-2*\xb-\xc+3*\dist,0-\ya-2*\yb-\yc) {\textcolor{\NumCol}{\tiny{$-5$}}};

%%%% Triangle -9
\path[draw,thick] (-\x-\xa-2*\xb-\xc+3*\dist,0-\ya-2*\yb-\yc) to (0-\xa-2*\xb-\xc+3*\dist,\h-\ya-2*\yb-\yc);
\path[draw,thick] (\x-\xa-2*\xb-\xc+3*\dist,0-\ya-2*\yb-\yc) to (0-\xa-2*\xb-\xc+3*\dist,\h-\ya-2*\yb-\yc);

%\node[above left] (A) at (0.5*\x-\xa-2*\xb-\xc+3*\dist,0.5*\h-\ya-2*\yb-\yc) {\tiny{$-$}};
%\node[above right] (B) at (-0.5*\x-\xa-2*\xb-\xc+3*\dist,0.5*\h-\ya-2*\yb-\yc) {\tiny{$-$}};
%\node[above left] (C) at (0-\xa-2*\xb-\xc+3*\dist,0-\ya-2*\yb-\yc) {\tiny{$-$}};

\path[->,draw,shorten >= 2pt,shorten <=2pt,>=stealth] (-0.5*\x-\xa-2*\xb-\xc+3*\dist,0.5*\h-\ya-2*\yb-\yc) to (0.5*\x-\xa-2*\xb-\xc+3*\dist,0.5*\h-\ya-2*\yb-\yc);
\path[->,draw,shorten >= 2pt,shorten <=2pt,>=stealth] (0-\xa-2*\xb-\xc+3*\dist,0-\ya-2*\yb-\yc) to (0.5*\x-\xa-2*\xb-\xc+3*\dist,0.5*\h-\ya-2*\yb-\yc);
\path[->,draw,shorten >= 2pt,shorten <=2pt,>=stealth] (0-\xa-2*\xb-\xc+3*\dist,0-\ya-2*\yb-\yc) to (-0.5*\x-\xa-2*\xb-\xc+3*\dist,0.5*\h-\ya-2*\yb-\yc);

\node[above right] at (-\x-\xa-2*\xb-\xc+3*\dist,0-\ya-2*\yb-\yc) {\textcolor{\NumCol}{\tiny{$-6$}}};

\end{tikzpicture}
\end{center}
\caption{The initial acyclic belt $I$}
\label{Figure:InitialBelt}
\end{figure}
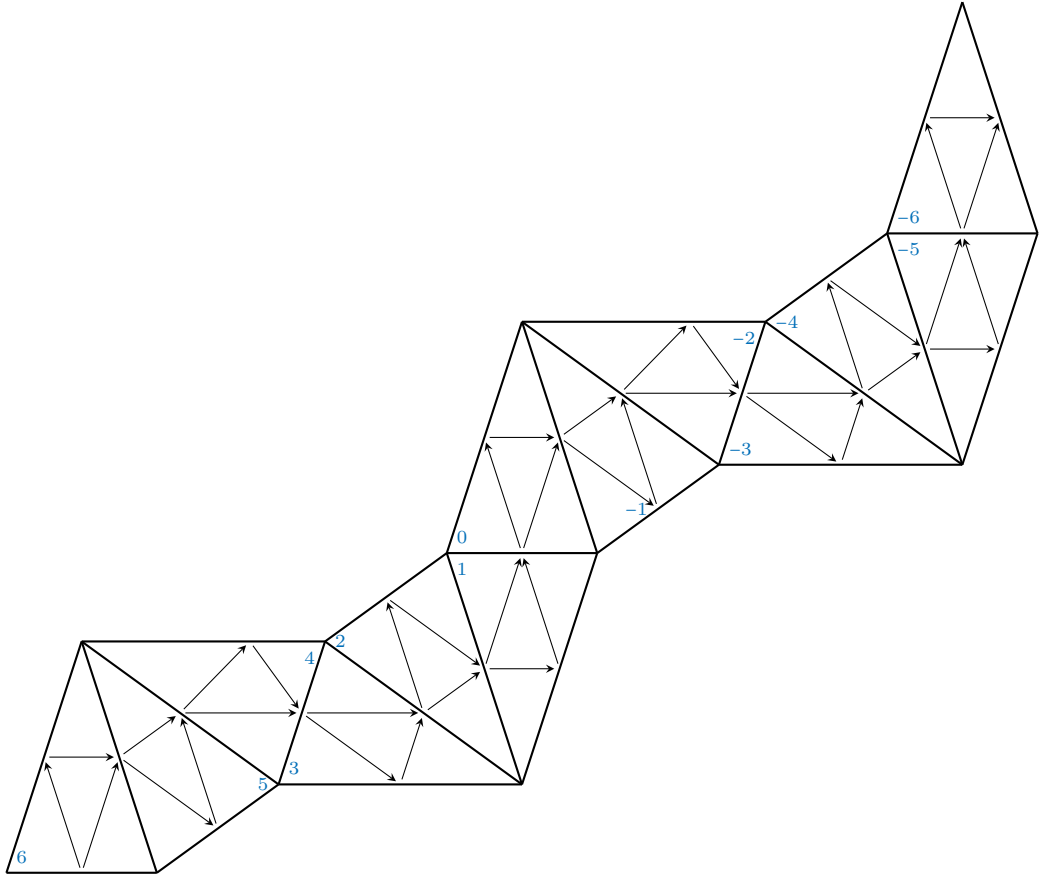

\subsection{A conserved quantity of mutation} Mutations of geometric realisations admit an interesting invariant. 

\begin{defn}[Mutation invariant]
\label{Def:MutInv}
Fix a triangle $\Delta=A_1A_2A_3$. For $i\in\{1,2,3\}$ we denote by $a_i$ the side of $\Delta$ opposite to $A_i$. We put
\begin{align*}
T(\Delta)=a_1\sin(A_2)\sin(A_3).
\end{align*}
Moreover, suppose now that $\Delta$ is an infinite region bounded by two parallel lines $a_1$ and $a_2$ and a finite side $a_3$ that intersects $a_1$ and $a_2$ at angles $\alpha$ and $\pi-\alpha$. Then we define 
\begin{align*}
T(\Delta)=a_3\sin(\alpha)\sin(\pi-\alpha)=a_3\sin^2(\alpha).
\end{align*}
\end{defn}

\begin{prop} The quantity $T$ is well-defined, that is, it is independent of the numbering of the vertices of the triangle $\Delta$. 
\end{prop}

\begin{proof}
Suppose that $\Delta=A_1A_2A_3$ a triangle as in Definition \ref{Def:MutInv}. We have to show 
\begin{align*}
a_1\sin(A_2)\sin(A_3)=a_2\sin(A_3)\sin(A_1)=a_3\sin(A_1)\sin(A_2).
\end{align*}
This is an immediate consequence of the law of sines.
\end{proof}

\begin{prop}[$T$ is mutation-invariant]
\label{Prop:MutationInvariance}
Suppose that the geometric realisation $(\mathbf{v},B)$ is represented by a triangle or an infinite region $\Delta$. Let $k\in\{1,2,3\}$ such that the mutation $\mu_k(\mathbf{v},B)=(v',B')$ is represented by triangle or an infinite region $\Delta'$. Then $T(\Delta)=T(\Delta')$. 
\end{prop}

\begin{proof} The statement is true when $k$ is a sink or a source. Assume that $k$ is neither a sink or a source. Without loss of generality we may assume $k=1$. Consider the formula $T(\Delta)=a_1\sin(\angle A_1A_2A_3)\sin(\angle A_2A_3A_1)$. The mutation $\mu_k$ leaves $a_1$ and one of the angles invariant and replaces the other angle with its supplementary angle. This means $T(\Delta)=T(\Delta')$.
\end{proof}

\begin{coro} Suppose that the triangles $\Delta$ and $\Delta'$ are similar (i.e. they have the same angles) and are related by a sequence of mutations as in the previous proposition. Then $\Delta$ and $\Delta'$ are congruent (i.e. they have the same side lengths).
\end{coro}

\begin{prop}
Let $\Delta=A_1A_2A_3$ be an acute-angled triangle and let $H_1$, $H_2$ and $H_3$ be the feet of the altitudes of $\Delta$. Then the invariant $T(\Delta)$ admits a geometric interpretation as
\begin{align*}
T(\Delta)=\tfrac{1}{2}\left(\lvert H_1H_2\rvert+\lvert H_2H_3\rvert+\lvert H_3H_1\rvert\right).
\end{align*}
\end{prop}

\begin{proof}
The semiperimeter of the triangle $H_1H_2H_3$ is equal to $\mathcal{A}/R$ where $\mathcal{A}$ is the area and $R$ the circumradius of $\Delta$, see Honsberger \cite[Section 5]{H}. On the other hand, $\mathcal{A}=\tfrac{1}{2} a_1a_2\sin(A_3)$. The relation $a_2=2R\sin(A_2)$, see Coxeter--Greitzer \cite[Theorem 1.1]{CG}, establishes the claim.
\end{proof}

Note that the points $H_1$, $H_2$ and $H_3$ from the previous proposition constitute the solution to Fagnano's problem. Hence the term $2T(\Delta)$ is the minimum perimeter of a triangle inscribed in a fixed triangle $A_1A_2A_3$. 

\begin{defn}[Translation vector] Suppose that a seed $(\mathbf{v},B)$ is represented by an acute-angled triangle $\Delta$. The vector that relates every sixth triangle in the acyclic belt of $(\mathbf{v},B)$ is called \emph{translation vector}.
\end{defn}

\begin{coro}
Suppose that the seed $(\mathbf{v},B)$ is represented by an acute-angled triangle. Then the translation vectors of $(\mathbf{v},B)$ and the initial seed have the same length $4T(\Delta)=4T(\Delta_0)$.
\end{coro}

\begin{proof}
The translation vector is equal to $2\left(\lvert H_1H_2\rvert+\lvert H_2H_3\rvert+\lvert H_3H_1\rvert\right)$, see Figure \ref{Figure:Line}. 
\end{proof}

\subsection{The feet of altitudes}

\begin{prop}[Feet of altitudes lie on $b$]
\label{Prop:Feet}
Suppose that the geometric realisation $(\mathbf{v},B)$ is obtained from the initial geometric realisation by a sequence of mutations and that it is given by a triangle $\Delta$ and a quiver $Q$.
\begin{enumerate}
\item If $\Delta$ is acute-angled so that the associated quiver $Q$ is acyclic, then the feet of the two altitudes opposite to the sink and the source lie on $b$.
\item If $\Delta$ is obtuse-angled, then the feet of the altitudes on the two sides of the obtuse angle lie on $b$.  
\end{enumerate}
\end{prop}

\begin{proof}
We prove this property by induction on the length of a shortest mutation sequence from  $(\mathbf{v}_0,B_0)$ to $(\mathbf{v},B)$. There are two types of mutations, namely mutations inside an acyclic belt, see Figure \ref{Figure:Line}, and mutations involving obtuse-angled triangles. For mutations inside an acyclic belt we can use the same arguments as in the proof of Proposition \ref{Prop:InitialBelt} because the reference point $u$ lies on the line $b$.

Let us consider mutations involving obtuse-angled triangles, see Figure \ref{Figure:ObtuseMutation}, whose quivers are all cyclic. Hence a mutation at a side $a$ keeps one of the other sides and reflects the other across $s$, compare also to Figure \ref{Figure:GeometricMutations}. Hence a generic situation is given locally as follows. Points $A_2,A_3,\ldots$ lie on a common line $l$, and  another point $A_1$ does not lie on the line. The points form triangles $A_1A_{k}A_{k+1}$ for $k\geq 2$. These triangles all have angles of the same size $\psi$ at $A_1$ and are obtained from each other by successive mutations. Now we denote by $H_3,H_4,\ldots$ the feet of the altitudes from $A_3,A_4,\dots$ on the lines $A_1A_2,A_1A_3,\ldots$ We want to prove that $H_3,H_4,\dots$ all lie on the line $b$. We apply the induction hypothesis to the triangle $A_1A_{k}A_{k+1}$ such that $A_1A_kA_{k+1}$ has the shortest mutation distance from the origin among all choices of $k$. The induction hypothesis implies that the feet $P$ of the perpendicular from $A_1$ on $l$ lies on $b$. 

The similarity transformation around $A_1$, which rotates by the angle $\psi$ and stretches with factor $\cos(\psi)$, maps $A_k$ to $H_k$ for all $k \geq 2$. This transformation maps the line $l$ to a line $b'$ so that $H_3,H_4,\dots$ all lie on a common line $b'$. By construction the line $b'$ intersects $l$ at the angle $\psi$. The induction hypothesis, applied to the triangle $A_1A_{k}A_{k+1}$,  implies that the feet of the altitudes $P$ and $H_k$ lie on $b$ so that $b=b'$.
\end{proof}

\begin{center}
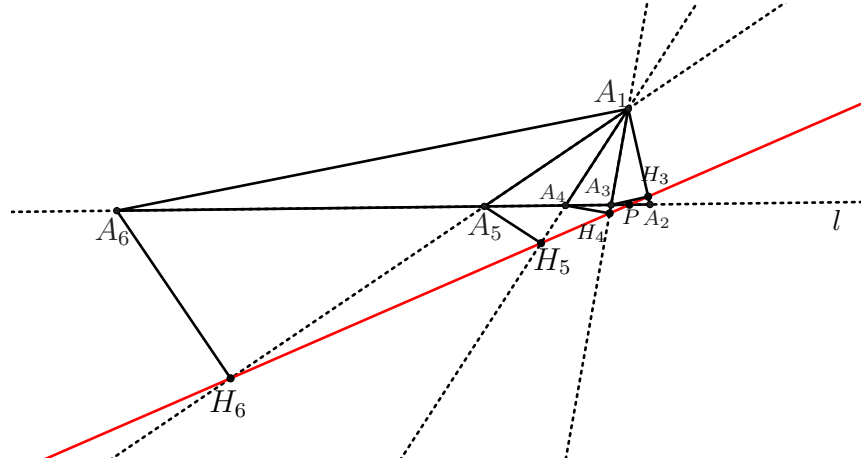
\begin{figure}
\definecolor{ffqqqq}{rgb}{1,0,0}\definecolor{sqsqsq}{rgb}{0.12549019607843137,0.12549019607843137,0.12549019607843137}\begin{tikzpicture}[scale=0.3,line cap=round,line join=round,>=triangle 45,x=1cm,y=1cm]\clip(-14.6,-12.54) rectangle (23.04,7.6);\draw [line width=1pt] (11.98,-1.32)-- (13.7,-1.3);\draw [line width=1pt] (12.74,2.94)-- (11.98,-1.32);\draw [line width=1pt] (12.74,2.94)-- (13.7,-1.3);\draw [line width=1pt,dotted,domain=-14.6:23.04] plot(\x,{(-2.51--0.02*\x)/1.72});
\node at (22,-2) {$l$};
\draw [line width=1pt,dotted,domain=-14.6:23.04] plot(\x,{(-39.49694127699572--3.646455471770662*\x)/2.3669732766539173});\draw [line width=1pt,dotted,domain=-14.6:23.04] plot(\x,{(-20.40872618586097--2.4284612471540146*\x)/3.5815884703677465});\draw [line width=1pt,dotted,domain=-14.6:23.04] plot(\x,{(--52.038-4.26*\x)/-0.76});\draw [line width=1pt] (12.74,2.94)-- (9.959514692228558,-1.3434940152066444);\draw [line width=1pt] (12.74,2.94)-- (6.360822531818554,-1.3853392728858305);\draw [line width=1pt] (6.360822531818554,-1.3853392728858305)-- (11.98,-1.32);\draw [line width=1pt,color=ffqqqq,domain=-14.6:23.04] plot(\x,{(-5.718970377196406--0.36200697228505374*\x)/0.8310080147537864});\draw [line width=1pt] (6.360822531818554,-1.3853392728858305)-- (8.873592179081012,-3.0164187284013626);\draw [line width=1pt] (9.959514692228558,-1.3434940152066444)-- (11.91361371107375,-1.6921126195076661);\draw [line width=1pt] (11.98,-1.32)-- (13.620433457500846,-0.948581103962073);\draw [line width=1pt] (12.74,2.94)-- (-9.92392109926941,-1.574696756968249);\draw [line width=1pt] (-9.92392109926941,-1.574696756968249)-- (6.360822531818554,-1.3853392728858305);\draw [line width=1pt] (-9.92392109926941,-1.574696756968249)-- (-4.88306414904689,-9.009147339553325);\begin{scriptsize}\draw [fill=sqsqsq] (11.98,-1.32) circle (4pt);\draw[color=sqsqsq] (11.4,-0.55) node {$A_3$};\draw [fill=sqsqsq] (13.7,-1.3) circle (4pt);\draw[color=sqsqsq] (14.0,-1.95) node {$A_2$};\draw [fill=sqsqsq] (12.74,2.94) circle (4pt);\draw[color=sqsqsq] (12.0,3.63) node {\large{$A_1$}};\draw [fill=sqsqsq] (9.959514692228558,-1.3434940152066444) circle (4pt);\draw[color=sqsqsq] (9.4,-0.8) node {$A_4$};\draw [fill=sqsqsq] (6.360822531818554,-1.3853392728858305) circle (4pt);\draw[color=sqsqsq] (6.4,-2.3) node {\large{$A_5$}};\draw [fill=black] (12.78942544274706,-1.3105880762471267) circle (4.5pt);\draw[color=black] (12.86,-1.85) node {$P$};\draw [fill=black] (13.620433457500846,-0.948581103962073) circle (4.5pt);\draw[color=black] (13.94,0.01) node {$H_3$};\draw [fill=black] (8.873592179081012,-3.0164187284013626) circle (4.5pt);\draw[color=black] (9.38,-3.81) node {\large{$H_5$}};\draw [fill=black] (11.91361371107375,-1.6921126195076661) circle (4.5pt);\draw[color=black] (11.1,-2.59) node {$H_4$};\draw [fill=sqsqsq] (-9.92392109926941,-1.574696756968249) circle (4pt);\draw[color=sqsqsq] (-10.1,-2.49) node {\large{$A_6$}};\draw [fill=black] (-4.88306414904689,-9.009147339553325) circle (4.5pt);\draw[color=black] (-4.99,-10.15) node {\large{$H_6$}};\end{scriptsize}\end{tikzpicture}
\caption{Mutations of an obtuse-angled triangle}
\label{Figure:ObtuseMutation}
\end{figure}
\end{center}

\begin{center}
\begin{figure}
\definecolor{ffqqqq}{rgb}{1,0,0}\definecolor{uuuuuu}{rgb}{0.26666666666666666,0.26666666666666666,0.26666666666666666}\begin{tikzpicture}[scale=0.4,line cap=round,line join=round,>=triangle 45,x=1cm,y=1cm]\draw [line width=0.8pt] (-13.00764033782241,-2.59)-- (-18.072083175656662,-4.844835226509175);\draw [line width=0.8pt] (-28.680248874204427,-2.59)-- (-18.072083175656662,-4.844835226509175);\draw [line width=0.8pt] (-13.00764033782241,-2.59)-- (-15.525319931504756,-5.38616646507369);\draw [line width=0.8pt,color=ffqqqq,domain=-36.82:0.44] plot(\x,{(--86.0255022656606-6.374624184329129*\x)/-29.990248874204426});\draw [line width=0.8pt] (-28.680248874204427,-2.59)-- (-13.00764033782241,-2.59);\draw [line width=0.8pt] (-18.072083175656662,-4.844835226509175)-- (-15.525319931504756,-5.38616646507369);\draw [line width=0.8pt,dotted,domain=-36.82:0.44] plot(\x,{(--16.213178697492957--2.2548352265091745*\x)/5.064442837834251});\draw [line width=0.8pt,dotted,domain=-36.82:0.44] plot(\x,{(-22.121631447765097-0.5413312385645153*\x)/2.546763244151906});\begin{scriptsize}\draw [fill=uuuuuu] (-28.680248874204427,-2.59) circle (3pt);\draw[color=uuuuuu] (-28.8,-3.1) node {$A_{k+1}$};\draw [fill=uuuuuu] (-18.072083175656662,-4.844835226509175) circle (3pt);\draw[color=uuuuuu] (-18.06,-5.42) node {$A_k$};\draw [fill=uuuuuu] (-15.525319931504756,-5.38616646507369) circle (3pt);\draw[color=uuuuuu] (-13.8,-4.94) node {$A_{k-1}$};\draw [fill=uuuuuu] (-13.00764033782241,-2.59) circle (3pt);\draw[color=uuuuuu] (-12.96,-1.68) node {$A_1$};\draw [fill=uuuuuu] (-26.0874556326005,-8.413508967772529) circle (3pt);\draw[color=uuuuuu] (-26.08,-9.46) node {$H_{k+1}$};\draw [fill=uuuuuu] (-13.685124437102195,-5.777312092164559) circle (3pt);\draw[color=uuuuuu] (-13.56,-6.48) node {$P$};\end{scriptsize}\end{tikzpicture}
\caption{Mutation of an obtuse-angled triangle}
\label{Figure:ObtuseMutation2}
\end{figure}
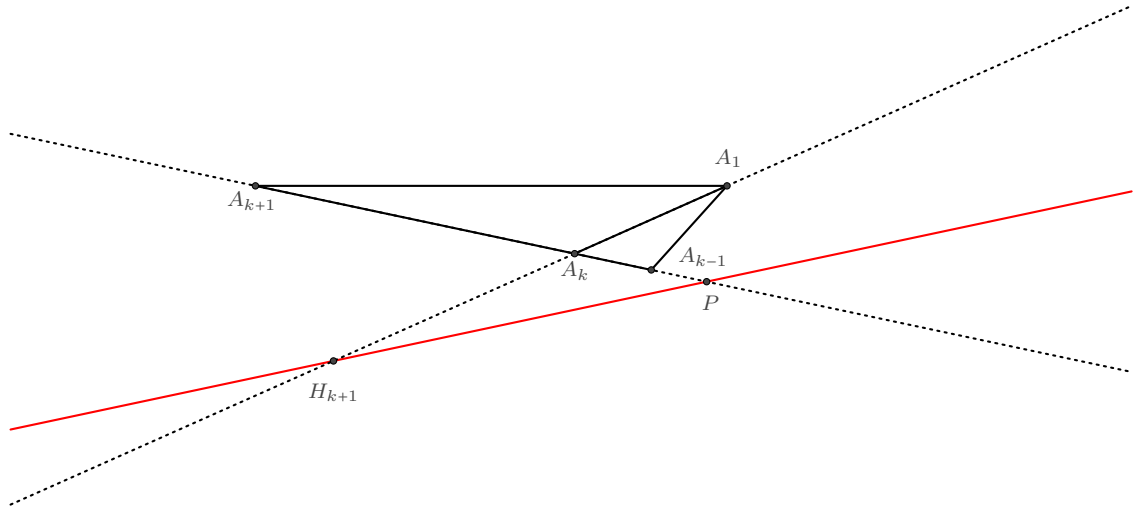
\end{center}

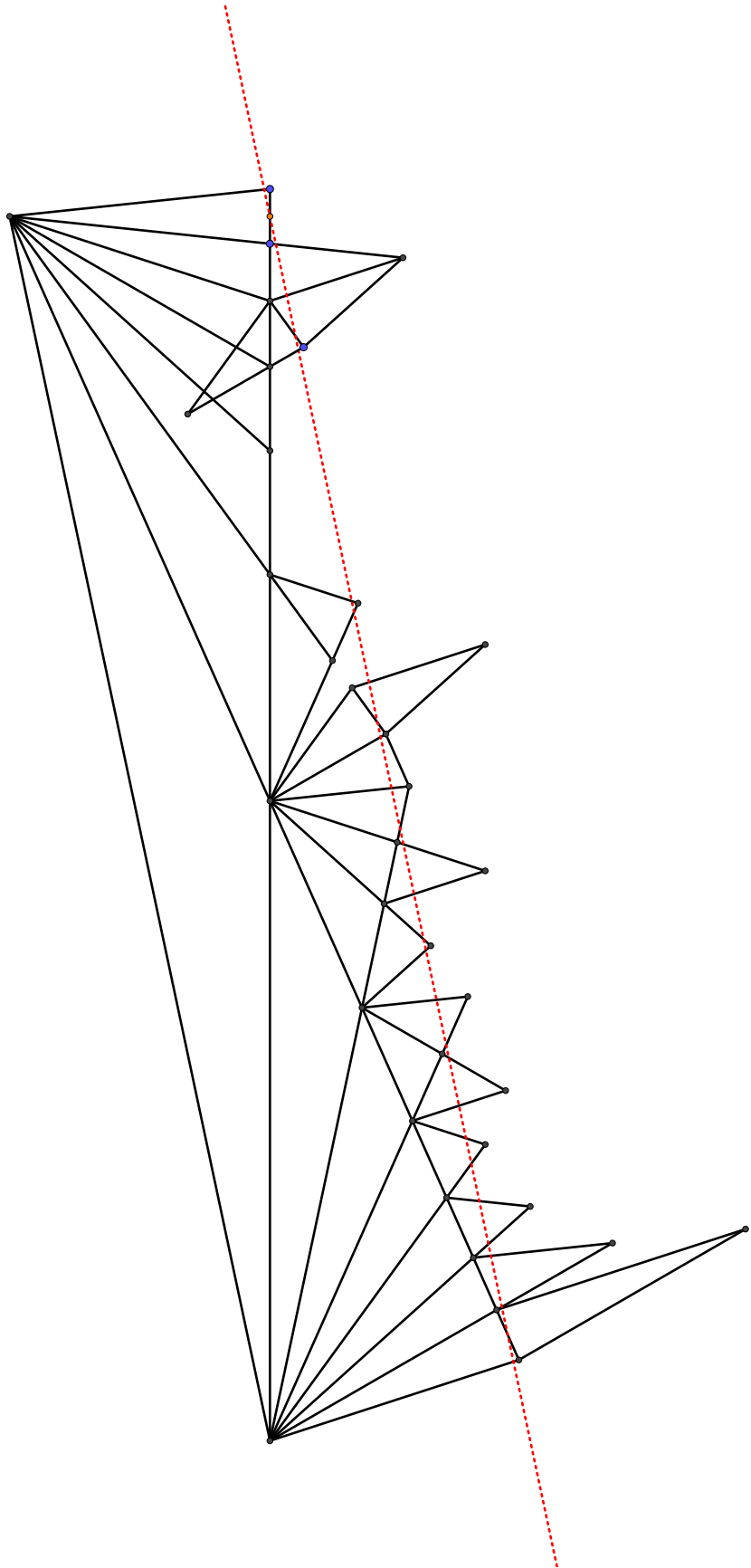
\begin{figure}[p]
\rotatebox{90}{
\definecolor{ffqqqq}{rgb}{1,0,0}\definecolor{ffxfqq}{rgb}{1,0.4980392156862745,0}\definecolor{uuuuuu}{rgb}{0.26666666666666666,0.26666666666666666,0.26666666666666666}\definecolor{ududff}{rgb}{0.30196078431372547,0.30196078431372547,1}\begin{tikzpicture}[scale=0.6,line cap=round,line join=round,>=triangle 45,x=1cm,y=1cm]\clip(-31.76,-15.49) rectangle (6.48,4.65);\draw [line width=1pt] (0.64,-2.59)-- (1.98,-2.59);\draw [line width=1pt] (1.31,3.7846241843291293)-- (0.64,-2.59);\draw [line width=1pt] (1.98,-2.59)-- (1.31,3.7846241843291293);\draw [line width=1pt] (1.31,3.7846241843291293)-- (-2.3703909888051244,-2.59);\draw [line width=1pt] (1.31,3.7846241843291293)-- (-4.429737396448772,-2.59);\draw [line width=1pt] (1.31,3.7846241843291293)-- (-7.463917477108924,-2.59);\draw [line width=1pt] (1.31,3.7846241843291293)-- (-13.00764033782241,-2.59);\draw [line width=1pt] (1.31,3.7846241843291293)-- (-28.680248874204427,-2.59);\draw [line width=1pt] (-28.680248874204427,-2.59)-- (0.64,-2.59);\draw [line width=1pt] (-13.00764033782241,-2.59)-- (-18.072083175656662,-4.844835226509175);\draw [line width=1pt] (-28.680248874204427,-2.59)-- (-26.699531089221104,-8.686022518633342);\draw [line width=1pt] (-28.680248874204427,-2.59)-- (-25.47538017598002,-8.140995416911771);\draw [line width=1pt] (-28.680248874204427,-2.59)-- (-24.195282867782296,-7.571059375315902);\draw [line width=1pt] (-28.680248874204427,-2.59)-- (-22.72525116242888,-6.9165590917795345);\draw [line width=1pt] (-28.680248874204427,-2.59)-- (-20.843944606013412,-6.0789474470043485);\draw [line width=1pt] (-26.699531089221104,-8.686022518633342)-- (-18.072083175656662,-4.844835226509175);\draw [line width=1pt] (-28.680248874204427,-2.59)-- (-18.072083175656662,-4.844835226509175);\draw [line width=1pt] (-13.00764033782241,-2.59)-- (-15.525319931504756,-5.38616646507369);\draw [line width=1pt] (-13.00764033782241,-2.59)-- (-14.020124437102195,-5.706105644005107);\draw [line width=1pt] (-13.00764033782241,-2.59)-- (-12.649503960316618,-5.997440019905175);\draw [line width=1pt] (-18.072083175656662,-4.844835226509175)-- (-12.649503960316618,-5.997440019905175);\draw [line width=1pt] (-18.072083175656662,-4.844835226509175)-- (-19.205710920309894,-6.808336077118282);\draw [line width=1pt] (-18.072083175656662,-4.844835226509175)-- (-17.79992665447517,-7.4342315576233);\draw [line width=1pt] (-20.843944606013412,-6.0789474470043485)-- (-17.79992665447517,-7.4342315576233);\draw [line width=1pt] (-22.72525116242888,-6.9165590917795345)-- (-21.423421437425556,-7.862393751215969);\draw [line width=1pt] (-20.843944606013412,-6.0789474470043485)-- (-21.423421437425556,-7.862393751215969);\draw [line width=1pt] (-24.195282867782296,-7.571059375315902)-- (-22.94051147775563,-8.964624184329125);\draw [line width=1pt] (-22.94051147775563,-8.964624184329125)-- (-22.72525116242888,-6.9165590917795345);\draw [line width=1pt] (-25.47538017598002,-8.140995416911771)-- (-23.83714649027652,-10.978499395221053);\draw [line width=1pt] (-23.83714649027652,-10.978499395221053)-- (-24.195282867782296,-7.571059375315902);\draw [line width=1pt] (-25.47538017598002,-8.140995416911771)-- (-23.494662390996734,-14.237017935545124);\draw [line width=1pt] (-23.494662390996734,-14.237017935545124)-- (-26.699531089221104,-8.686022518633342);\draw [line width=1pt] (-13.00764033782241,-2.59)-- (-9.57032221972923,-4.120392626705836);\draw [line width=1pt] (-9.57032221972923,-4.120392626705836)-- (-8.164537953894506,-4.746288107210854);\draw [line width=1pt] (-8.164537953894506,-4.746288107210854)-- (-7.463917477108924,-2.59);\draw [line width=1pt] (-9.57032221972923,-4.120392626705836)-- (-7.463917477108924,-2.59);\draw [line width=1pt] (1.31,3.7846241843291293)-- (-0.7612409535711604,-2.59);\draw [line width=1pt] (-15.525319931504756,-5.38616646507369)-- (-14.720744913887772,-7.86239375121595);\draw [line width=1pt] (-14.720744913887772,-7.86239375121595)-- (-14.020124437102195,-5.706105644005107);\draw [line width=1pt] (-18.072083175656662,-4.844835226509175)-- (-16.55499313532658,-6.529734411422479);\draw [line width=1pt] (-16.55499313532658,-6.529734411422479)-- (-15.525319931504756,-5.38616646507369);\draw [line width=1pt] (-20.843944606013412,-6.0789474470043485)-- (-20.10234593283077,-8.361353474649578);\draw [line width=1pt] (-20.10234593283077,-8.361353474649578)-- (-19.205710920309894,-6.808336077118285);\draw [line width=1pt] (-11.369406652118906,-5.427503978309312)-- (-12.649503960316618,-5.997440019905175);\draw [line width=1pt] (-13.00764033782241,-2.59)-- (-10.235778907465686,-4.603875210891957);\draw [line width=1pt] (-10.235778907465686,-4.603875210891957)-- (-11.369406652118906,-5.427503978309312);\draw [line width=1pt] (0.64,-2.59)-- (0.29751590072021533,-5.848518540324023);\draw [line width=1pt] (0.29751590072021533,-5.848518540324023)-- (-0.7612409535711604,-2.59);\draw [line width=1pt] (-0.7612409535711604,-2.59)-- (-1.8948686982243848,-3.4136287674173706);\draw [line width=1pt] (-1.8948686982243848,-3.4136287674173706)-- (0.29751590072021533,-5.848518540324023);\draw [line width=1pt] (-13.00764033782241,-2.59)-- (-11.369406652118906,-5.427503978309312);\draw [line width=1pt] (-11.369406652118906,-5.427503978309312)-- (-9.17702205317433,-7.8623937512159845);\draw [line width=1pt] (-9.17702205317433,-7.8623937512159845)-- (-10.235778907465686,-4.603875210891957);\draw [line width=1pt,dotted,color=ffqqqq,domain=-31.76:6.48] plot(\x,{(--86.0255022656606-6.374624184329129*\x)/-29.990248874204426});\draw [line width=1pt] (-3.533102383927903,-0.5761247891080642)-- (-1.8948686982243848,-3.4136287674173706);\draw [line width=1pt] (-3.533102383927903,-0.5761247891080642)-- (-0.7612409535711604,-2.59);\begin{scriptsize}\draw [fill=ududff] (0.64,-2.59) circle (2.5pt);\draw [fill=ududff] (1.98,-2.59) circle (2.5pt);\draw [fill=uuuuuu] (1.31,3.7846241843291293) circle (2pt);\draw [fill=uuuuuu] (-28.680248874204427,-2.59) circle (2pt);\draw [fill=uuuuuu] (-13.00764033782241,-2.59) circle (2pt);\draw [fill=uuuuuu] (-7.463917477108924,-2.59) circle (2pt);\draw [fill=uuuuuu] (-4.429737396448772,-2.59) circle (2pt);\draw [fill=uuuuuu] (-2.3703909888051244,-2.59) circle (2pt);\draw [fill=uuuuuu] (-18.072083175656662,-4.844835226509175) circle (2pt);\draw [fill=uuuuuu] (-26.699531089221104,-8.686022518633342) circle (2pt);\draw [fill=uuuuuu] (-25.47538017598002,-8.140995416911771) circle (2pt);\draw [fill=uuuuuu] (-24.195282867782296,-7.571059375315902) circle (2pt);\draw [fill=uuuuuu] (-22.72525116242888,-6.9165590917795345) circle (2pt);\draw [fill=uuuuuu] (-20.843944606013412,-6.0789474470043485) circle (2pt);\draw [fill=uuuuuu] (-15.525319931504756,-5.38616646507369) circle (2pt);\draw [fill=uuuuuu] (-14.020124437102195,-5.706105644005107) circle (2pt);\draw [fill=uuuuuu] (-12.649503960316618,-5.997440019905175) circle (2pt);\draw [fill=uuuuuu] (-23.494662390996734,-14.237017935545124) circle (2pt);\draw [fill=uuuuuu] (-17.79992665447517,-7.4342315576233) circle (2pt);\draw [fill=uuuuuu] (-21.423421437425556,-7.862393751215969) circle (2pt);\draw [fill=uuuuuu] (-19.205710920309894,-6.808336077118285) circle (2pt);\draw [fill=uuuuuu] (-22.94051147775563,-8.964624184329125) circle (2pt);\draw [fill=uuuuuu] (-23.83714649027652,-10.978499395221053) circle (2pt);\draw [fill=uuuuuu] (-9.57032221972923,-4.120392626705836) circle (2pt);\draw [fill=uuuuuu] (-8.164537953894506,-4.746288107210854) circle (2pt);\draw [fill=uuuuuu] (-0.7612409535711604,-2.59) circle (2pt);\draw [fill=uuuuuu] (-14.720744913887772,-7.86239375121595) circle (2pt);\draw [fill=uuuuuu] (-16.55499313532658,-6.529734411422479) circle (2pt);\draw [fill=uuuuuu] (-20.10234593283077,-8.361353474649578) circle (2pt);\draw [fill=uuuuuu] (-11.369406652118906,-5.427503978309312) circle (2pt);\draw [fill=uuuuuu] (-10.235778907465686,-4.603875210891957) circle (2pt);\draw [fill=uuuuuu] (0.29751590072021533,-5.848518540324023) circle (2pt);\draw [fill=ududff] (-1.8948686982243848,-3.4136287674173706) circle (2.5pt);\draw [fill=uuuuuu] (-9.17702205317433,-7.8623937512159845) circle (2pt);\draw [fill=ffxfqq] (1.31,-2.59) circle (2pt);\draw [fill=uuuuuu] (-3.533102383927903,-0.5761247891080642) circle (2pt);\end{scriptsize}\end{tikzpicture}}
\caption{Geometric mutations}
\label{Figure:GeometricMutations}
\end{figure}

\begin{rem}[Triangles are oriented towards the reference point]
\label{Rem:Orientation}
Let us look at the consequences of Proposition~\ref{Prop:Feet} for the orientation of the arrows in the associated quivers. To describe the orientations, we refer to the half-planes bounded by $b$ as $\Pi_1$ and $\Pi_2$. We divide each of the two cases of Proposition~\ref{Prop:Feet} into two subcases.
\begin{enumerate}
\item First assume that the triangle $\Delta=A_1A_2A_3$ is acute-angled, so that the associated quiver $Q$ is acyclic. Then the belt line $b$ intersects $\Delta$ in two points $H_1$ and $H_3$, and without loss of generality let us say the side $A_1A_2$ (containing $H_3$) corresponds to the source of the quiver; the side $A_2A_3$ (containing $H_1$) corresponds to the sink. There are two possibilities.
\begin{enumerate}
\item[(1a)] The third side of $\Delta$, namely $A_1A_3$, lies in $\Pi_1$.
\item[(1b)] The third side of $\Delta$, namely $A_1A_3$, lies in $\Pi_2$.
\end{enumerate}
\item Second assume that the triangle $\Delta$ is obtuse-angled so that the associated quiver $Q$ is cyclic. Then the belt line $b$ does not intersect $\Delta$, because the two feet of the altitudes are outside the triangle. Again, there are two possibilities.
\begin{enumerate}
\item[(2a)] The triangle $\Delta$ lies in $\Pi_1$.
\item[(2b)] The triangle $\Delta$ lies in $\Pi_2$.
\end{enumerate}
Let us label the vertices $\Delta=A_1A_2A_3$, so that the belt line $b$ intersects the extensions of the sides $A_1A_2$ and $A_2A_3$ in $H_3$ and $H_1$, respectively, and $H_3$ lies between $H_1$ and $u$. 
\end{enumerate}
Now it is easy to see that a sink/source mutation of an acute-angled seed takes case (1a) over to case (1b), and vice versa. Moreover, a mutation of an acute-angled seed at a vertex which is either sink nor source transports (1a) $\to$ (2a) and (1b) $\to$ (2b). What is more, mutation of a seed in case (2a) produces a seed in cases (1a) or (2a), and mutation of a seed in case (2b) produces a seed in cases (1b) or (2b). All these mutations leave the following structure intact.
\begin{enumerate}
\item If $\Delta=A_1A_2A_3$ is acute-angled, then the triangle is oriented towards the reference point $u$, that is, the sides are oriented $A_1A_2\to A_1A_3 \to A_2A_3$.
\item If $\Delta=A_1A_2A_3$ is obtuse-angled, then the sides are oriented $A_1A_2\to A_1A_3 \to A_2A_3 \to A_1A_2$. In particular, if two obtuse-angled seeds are located on different sides of the line $b$, see cases (2a) and (2b), then they have different orientations, compare Figure~\ref{Figure:Orientations}.
\end{enumerate}
\end{rem}

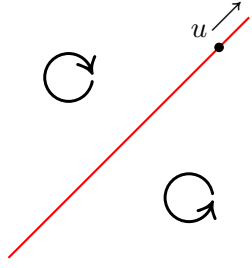
\begin{figure}
\begin{center}
\begin{tikzpicture}[scale=0.8]
\draw[thick,red] (-2,-2) to (2,2);
\filldraw[black] (1.5,1.5) circle (2pt) node[anchor=south] {$u$\rotatebox{45}{$\longrightarrow$}};
\node at (-1,1) {\textbf{\Huge{$\lcirclearrowright$}}};
\node at (1,-1) {\textbf{\Huge{$\rcirclearrowright$}}};
\end{tikzpicture}
\end{center}
\caption{The orientation of obtuse-angled triangles}
\label{Figure:Orientations}
\end{figure}

\subsection{Translational symmetries of the exchange graph}
\label{Section:AcyclicBelts}

\begin{defn}[Translated seeds] Let $w\in \mathbb{E}^2$.
\begin{enumerate}
\item Suppose that the seed $(\mathbf{v},B)$ is represented by a triangle $\Delta$ together with a quiver $Q$. The \emph{translated seed} $w+(\mathbf{v},B)$ is given by the translation of $\Delta$ in $\mathbb{E}^2$ by the vector $w$, together with the same quiver $Q$. In this case, we call $w$ the \emph{translation vector} between the seeds.
\item Suppose that $J$ is an acyclic belt. We denote its vertices by $J_k$ with $k\in\mathbb{Z}$. The \emph{translated acyclic belt} $w+J$ is the infinite path graph with vertices $w+J_k$ with $k\in \mathbb{Z}$. 
\end{enumerate}
\end{defn}

\begin{prop}
\label{Prop:ParallelTob}
All possible translation vectors between seeds in the mutation class of $(\mathbf{v}_0,B_0)$ are parallel to $b$.
\end{prop}

\begin{proof}
Suppose that two seeds are related by a translation $w$. We assume that the seeds are represented by triangles (or infinite regions) $\Delta$ and $\Delta'$ together with quivers $Q$ and $Q'$. By assumption $\Delta'=w+\Delta$. We define $H$ to be the foot of the altitude on the side of $\Delta$ that corresponds to the source in $Q$ if $\Delta$ is acute-angled, and define $H$ to be the foot of the altitude on the side of $\Delta$ that corresponds to the source of the arrow between the two sides of the obtuse angle otherwise. We define $H'$ similarly. Since the triangles $\Delta$ and $\Delta'$ are congruent and $Q$ is isomorphic to $Q'$, we must have $H'=w+H$. By Proposition \ref{Prop:Feet} both $H$ and $H'$ lie on $b$. 
\end{proof}

By construction, the translation of the initial seed $(\mathbf{v}_0,B_0)$ by a translation vector $w$ parallel to $b$ is again an acyclic seed in the acyclic belt $w+(\mathbf{v}_0,B_0)$. For some vectors $w$, for example integer multiples of the translation vector $T(\Delta_0)$, this coincides with the acyclic belt of $(\mathbf{v}_0,B_0)$.

\begin{prop}
\label{Prop:FinitelyManyTriangles}
Let $Q$ be a quiver with three vertices and let $n_1,n_2,n_3\in [1,2n]$ be natural numbers such that $n_1\alpha+n_2\alpha+n_3\alpha=\pi$ and $\gcd(n_1,n_2,n_3)=1$. Then, up to translation by vectors parallel to $b$, the mutation class of $(\mathbf{v}_0,B_0)$ contains exactly two seeds with quiver $Q$ whose triangles have angles $n_1\alpha$, $n_2\alpha$ and $n_3\alpha$ (corresponding to the vertices of $Q$ in this order).
\end{prop}

\begin{proof}
The mutation class of $(\mathbf{v}_0,B_0)$ must contain a triangles with angles $n_1\alpha$, $n_2\alpha$ and $n_3\alpha$ by Proposition~\ref{Prop:PossibleTriangles}. The conserved quantity $T(\Delta)=T(\Delta_0)$ from Proposition~\ref{Prop:MutationInvariance} is homogeneous of degree $1$, which implies that the angles of the triangle $\Delta$ suffice to determine its size. In particular, we can reconstruct the triangle up to congruence uniquely when given its angles. 

Proposition~\ref{Prop:Feet} shows that the belt line $b$ intersects every triangle corresponding to a seed in $\Gamma$ in the feet of two altitudes. Moreover, the given quiver $Q$ and the given angles tell us which altitudes we have to consider. Namely, for acyclic quivers (with acute-angled triangles as in case (1) of Proposition~\ref{Prop:Feet}) we have to consider the altitudes on the sides that correspond to the source and the sink. The quiver must be oriented towards the reference point by Remark~\ref{Rem:Orientation}, which determines the location of the triangles up to translations by vectors parallel to $b$ and reflection across $b$.

For cyclic quivers (with obtuse-angled triangles as in case (2) of Proposition~\ref{Prop:Feet}) we have to consider the altitudes on the sides of the obtuse angle. The triangle must lie either in $\Pi_1$ or in $\Pi_2$, where $\Pi_1$ and $\Pi_2$ are the half-planes bounded by $b$ as in Remark \ref{Rem:Orientation}. This location determines the orientation of the triangle by Remark~\ref{Rem:Orientation} as well as the relative position of $H_1$ and $H_3$ on $b$. 

Hence in all cases, up to translation by vectors parallel to $b$, there are exactly two triangles with prescribed angles in the mutation class of $(\mathbf{v}_0,B_0)$.
\end{proof} 

\begin{prop}
\label{Prop:GroupAction}
Let $w\in\mathbb{E}^2$. Assume that there exist two seeds $(\mathbf{v}_1,B_1)$ and $(\mathbf{v}_2,B_2)$ in $\Gamma$ such that $(\mathbf{v}_2,B_2)=w+(\mathbf{v}_1,B_1)$. Then the map $T_w$
\begin{align*}
\Gamma&\longrightarrow\Gamma\\
(\mathbf{v},B)&\longmapsto w+(\mathbf{v},B)
\end{align*}
is an automorphism of graphs.
\end{prop}

\begin{proof}
Suppose that $(\mathbf{v},B)$ is a seed in $\Gamma$ and that the triangles of $(\mathbf{v},B)$ and $w+(\mathbf{v},B)$ are given by $\Delta$ and $w+\Delta$. Since $w$ is parallel to $b$, a side of $\Delta$ has the same sign as the corresponding side of $w+\Delta$. Hence, a mutation at these sides yield triangles which are related to each other by a translation by $w$.     
\end{proof}

\begin{prop} The choice of $V^{+}\subseteq V$ is compatible with every seed. 
\end{prop}

\begin{proof}
We can verify the statement by drawing the sectors of non-compatible reference points, separately for acute-angled and obtuse-angled triangles, like we draw them for the initial triangle in Figure \ref{Figure:Forbidden}. An inspection shows that the choice of the reference point $u$ infinitely far on the line $b$ is compatible for every triangle. 
\end{proof}

\subsection{On the lengths of sides and translation vectors} 
\label{Section:Lengths}

We denote the length of the larger side in the initial triangle by $d_1$, see Figure \ref{Figure:Initial}.

Suppose that the seed $(\mathbf{v},B)$ is obtained from $(\mathbf{v}_0,B_0)$ by a sequence of mutations. Assume that $(\mathbf{v},B)$ is represented by a triangle $\Delta$. %We claim that the side lengths of $\Delta$ belong to the field $d_1K$ where $K=\mathbb{Q}(2\cos(2\alpha))$. To verify the claim 
Recall that $T(\Delta)=T(\Delta_0)$. Hence, if $a$ is a side of $\Delta$ and $\beta$ and $\gamma$ are angles in $\Delta$ incident with $a$, then $\beta$ and $\gamma$ are rational multiples of $\alpha=\pi/(2n+1)$ and
\begin{align}
\label{Eqn:Sides}
a = d_1\frac{\sin(\alpha)\sin(n\alpha)}{\sin(\beta)\sin(\gamma)}.
\end{align}

%\begin{proof}
%We prove the statement by induction on the length of a shortest mutation sequence from $(v_0,B_0)$ to $(\mathbf{v},B)$. The statement is true for the initial triangle $\Delta_0$ which has sides $d_1$, $d_1$ and $2d_1\cos(\frac{\pi}{2n+1})$.
%For the induction step note that the new side length is obtained from a previous side length by multiplication by $\sin(\frac{k_1\pi}{2n+1})/\sin(\frac{k_2\pi}{2n+1})$ which is equal to $U_{k_1}(\cos(\frac{\pi}{2n+1}))/U_{k_2}(\cos(\frac{\pi}{2n+1}))$. This term is contained in $d_1R$ by the previous proposition.  
%\end{proof} 

\begin{defn}[Translation vector lengths]
\label{Defn:TVectors}
For every $k\in[1,n]$ with $\gcd(k,2n+1)=1$ we put
\begin{align*}
s_k=d_1\frac{\sin(\alpha)\sin(n\alpha)}{\sin^2(k\alpha)}.
\end{align*} 
\end{defn}

\begin{prop}
\label{Prop:InfiniteRegionTranslationVectors}
Suppose that $k\in[1,n]$ is coprime to $2n+1$. Furthermore, let $w_k\in\mathbb{E}^2$ be a vector parallel to the line $b$ whose length is equal to $s_k$, where $s_k$ is as in Definition \ref{Defn:TVectors}. Then the exchange graph of $(\mathbf{v}_0,B_0)$ contains the translated acyclic belt $I+w_k$.
\end{prop}

\begin{proof}
We consider an infinite region $\Delta$ that is given by two parallel half-lines that intersect the third (finite) side $a$ of the region at angles $k\alpha$ and $\pi-k\alpha$. By Proposition \ref{Prop:PossibleTriangles} the mutation class of $(\mathbf{v}_0,B_0)$ contains a seed $(\mathbf{v},B)$ whose domain is equal to $\Delta$. Its quiver must be oriented cyclically since $\pi-k\alpha$ is obtuse. The mutation of $\Delta$ at one of the parallel sides yields an infinite region $\Delta'$ which is obtained from $\Delta$ by translation along $a$. Note that $a$ is parallel to the line $b$ by Proposition \ref{Prop:ParallelTob}. The equation $T(\Delta)=T(\Delta_0)$ implies that the length of $a$ is equal to $s_k$.

Then there exists a mutation sequence that takes $\Delta$ to a triangle $I_k$ in the initial acyclic belt $I$ in the same way as $\Delta'$ is mapped to the translation of $I_k$ along $a$.
\end{proof}

\section{The algebraic number theory of sines and cosines of occuring angles}
\label{Section:NumberTheory}

\subsection{Cyclotomic fields}

In this section we study the number theoretic properties of the cosines of the angles that occur in the triangles. The material in this section is classical and can be found for example in Lang's books about cyclotomic fields.

\begin{defn}[Cyclotomic polynomials] For every $d\geq 1$ the polynomial
\begin{align*}
\Phi_d=\prod_{\substack{k \in [1,d] \\ \gcd(k,d)=1}}(x-e^{2\pi k i/d})\in\mathbb{C}[x].
\end{align*}
is called the $d$-th \emph{cyclotomic polynomial}.
\end{defn}
Cyclotomic polynomials satisfy many properties. In particular, for any given $d\geq 1$, the polynomial $\Phi_d\in \mathbb{Z}[x]$ is a monic polynomial with integer coefficients. In fact, it is the minimal polynomial of any primitive $d$-th root of unity $e^{2\pi k i/d}$ with $\gcd(k,d)=1$. Moreover, $\Phi_d$ is irreducible over $\mathbb{Z}$ and for every $m\geq 1$ we have
\begin{align*}
x^m-1=\prod_{d \mid m}\Phi_d.
\end{align*}

The degree of the cyclotomic polynomial $\deg(\Phi_d)=\varphi(d)$ is given by Euler's totient function. We consider the primitive root of unity $\zeta=e^{2\alpha i}$ with $\alpha=\frac{\pi}{2n+1}$. Since $2n+1$ is odd, the zeros of $\Phi_{2n+1}$ can be grouped into pairs of complex conjugates. In other words the set of complex zeros of $\Phi_{2n+1}$ is equal to $\left\lbrace \,\zeta^k,\zeta^{-k} \mid k \in U\, \right\rbrace$ where
\begin{align*}
U=\left\lbrace \, k\in[1,n] \mid \gcd(k,2n+1)=1\,\right\rbrace.
\end{align*}
%In particular, the product of the zeros is equal to $1$
In particular, the polynomial $\Phi_{2n+1}$ is palindromic, i.e. $\Phi_{2n+1}(1/x)=x^{-\varphi(2n+1)}\Phi_{2n+1}(x)$. Hence there is a polynomial $\Psi_{2n+1}\in\mathbb{Z}[x]$ of degree $\varphi(2n+1)/2$ such that 
\begin{align*}
\Psi_{2n+1}(x+x^{-1})=x^{-\varphi(2n+1)/2}\Phi_{2n+1}(x).
\end{align*}
This polynomial must be irreducible over $\mathbb{Z}$, because any non-trivial factorization would yield a non-trivial factorization of $\Phi_{2n+1}$. Furthermore, $2\cos(2k \alpha)$ is a root of $\Psi_{2n+1}$ for every $k \in U$. As the cardinality of the set $U$ is equal to the degree of $\Psi_{2n+1}$, the polynomial $\Psi_{2n+1}$ cannot have roots other than the ones from above. Moreover, $\Psi_{2n+1}$ is monic.

\begin{defn}[Chebyshev polynomials of the first kind] Define $T_m\in\mathbb{Z}[x]$ recursively by $T_0=1$, $T_1=x$ and $T_{m+1}=2xT_m-T_{m-1}$ for $m\geq 1$. 
\end{defn}
It is well-known that $T_m(\cos(x))=\cos(mx)$ for all $m\in\mathbb{N}$ and $x\in\mathbb{R}$. Note that every $T_m$ is a monic integer polynomial in $t=2x$. Moreover, a result of Watkins--Zeitlin \cite[Equation (2)]{WZ} asserts that if $m=2n+1$ is odd, then
\begin{align*}
T_{n+1}-T_n=2^n\prod_{d \mid m}\Psi_d.
\end{align*}
%In particular, $2\cos(2k\alpha)$ is a monic integer polynomial in $2\cos(2\alpha)$ for every $k \in U$.
%This implies that $\mathcal{O}_K$ is closed under multiplication by $2\cos(2k\alpha)$ for every $k\geq 1$.

%\begin{defn}[Chebyshev polynomials of the second kind] Define $U_m\in\mathbb{Z}[x]$ recursively by $U_0=1$, $U_1=2x$ and $U_{k+1}=2xU_k-U_{k-1}$ for $k\geq 1$.
%\end{defn}
%
%It is well-known that $U_k(\cos(x))=\sin(kx)/\sin(x)$ for all $k\in\mathbb{N}$ and $x\in\mathbb{R}$. Note that every $U_k$ is a monic integer polynomial in $t=2x$.The Chebyshev polynomials $U_k\in\mathbb{Z}[x]$ of the second kind are defined recursively by $U_0=1$, $U_1=2x$ and $U_{k+1}=2xU_k-U_{k-1}$ for $k\geq 1$. It is well-known that $U_k(\cos(x))=\sin(kx)/\sin(x)$ for all $k\in\mathbb{N}$ and $x\in\mathbb{R}$. Note that every $U_k$ is a monic integer polynomial in $t=2x$.

The field $\mathbb{Q}(\zeta)$ is known as the \emph{cyclotomic field}. We also consider the following number field.  

\begin{defn}[Maximal real subfield] Put $K=\mathbb{Q}(2\cos(2\alpha))=\mathbb{Q}(\zeta+\zeta^{-1})$. 
\end{defn}

Then $K\subseteq \mathbb{Q}(\zeta)$ is the maximal real subfield of the cyclotomic field. Notice that $K$ contains the element $2\cos(k \alpha)$ for every $k\in U$.

Recall that an element $x\in \mathbb{C}$ is called an \emph{algebraic integer} if it is a root of a monic polynomial with integer coefficients. The set of algebraic integers in a field extension $F/\mathbb{Q}$ is a subring of $F$ and is denoted $\mathcal{O}_F$. It is known that $\mathcal{O}_{\mathbb{Q}(\zeta)}=\mathbb{Z}[\zeta]$. Note that $2\cos(2k \alpha)\in \mathcal{O}_K$ for every $k\in U$ because the elements are roots of the monic integer polynomial $\Psi_{2n+1}$. In fact, the set $\left\lbrace \,2\cos(2k\alpha) \mid k \in U\, \right\rbrace$ is a basis of the $\mathbb{Z}$-module $\mathcal{O}_K$.

Note that for every $k\in U$ the equality $\cos(2k\alpha)=-\cos((2n+1-2k)\alpha)$ holds. If $k>n/2$, then the number $2n+1-2k$ is an odd integer in the interval $[1,n]$. In particular,
\begin{align}
\label{Eq:CyclotomicInt}
\mathcal{O}_K=\left\langle\, 2\cos(2k\alpha) \mid k \in U\,\right\rangle_{\mathbb{Z}}=\left\langle\, 2\cos(k\alpha) \mid k \in U\,\right\rangle_{\mathbb{Z}}.
\end{align}

We are also interested in the groups of units $\mathcal{O}_K^{\times}\subseteq \mathcal{O}_{\mathbb{Q}(\zeta)}^{\times}$ because they contain elements which are relevant for our geometric discussions. Elements of these groups are known as \emph{cyclotomic units} in the literature.

\begin{prop}
\label{Prop:RingIntegers}
Let $k\in [1,n]$.
\begin{enumerate}
\item We have
\begin{align*}
(a)\quad\frac{1-\zeta^k}{1-\zeta},\frac{1-\zeta^k}{1-\zeta^n}\in\mathcal{O}_{\mathbb{Q}(\zeta)},&& (b)\quad \frac{\sin(k\alpha)}{\sin(\alpha)},\frac{\sin(k\alpha)}{\sin(n\alpha)}\in\mathcal{O}_K.
\end{align*}
\item If in addition $k$ is coprime to $2n+1$, then
\begin{align*}
(a)\quad \frac{1-\zeta^k}{1-\zeta},\frac{1-\zeta^k}{1-\zeta^n}\in\mathcal{O}_{\mathbb{Q}(\zeta)}^{\times},&&(b)\quad\frac{\sin(k\alpha)}{\sin(\alpha)},\frac{\sin(k\alpha)}{\sin(n\alpha)}\in\mathcal{O}_K^{\times}.
\end{align*}
\end{enumerate}
\end{prop}

\begin{proof}
(1a) Using the geometric series we see that the element $(1-\zeta^k)/(1-\zeta)=\sum_{l=0}^{k-1}\zeta^l$ belongs to $\mathcal{O}_{\mathbb{Q}(\zeta)}=\mathbb{Z}[\zeta]$. For the second claim note that $n$ is coprime to $2n+1$. Hence $\zeta^n$ is a generator of the cyclic group $\{ \zeta^k \mid k\in[0,2n]\}$. Hence there is an integer $l$ such that $\zeta^k=\zeta^{nl}$. Then $(1-\zeta^{nl})/(1-\zeta^n)\in \mathcal{O}_{\mathbb{Q}(\zeta)}=\mathbb{Z}[\zeta]$ using the geometric series as before.

(1b) By the previous part of the proposition $(1-\zeta^k)/(1-\zeta)$ is an algebraic integer. The same is true for the $(4n+2)$-th root of unity $\zeta^{\frac{1-k}{2}}$. Hence the product of the two numbers is an algebraic integer as well. The identities $2\sin(k\alpha)=\zeta^{k/2}-\zeta^{-k/2}$  and $2\sin(\alpha)=\zeta^{1/2}-\zeta^{-1/2}$ imply
\begin{align*}
\zeta^{\frac{1-k}{2}}\cdot\frac{1-\zeta^k}{1-\zeta} =\frac{\sin(k\alpha)}{\sin(\alpha)}.
\end{align*}
Now $\sin(k\alpha)/\sin(\alpha)\in\mathcal{O}_K$ because it is an algebraic integer and it belongs to $K=\mathbb{Q}(\zeta)\cap \mathbb{R}$. 

(2) Now suppose that $\gcd(k,2n+1)=1$. First we want to show that $(1-\zeta)/(1-\zeta^k)\in\mathcal{O}_{\mathbb{Q}(\zeta)}$. The assumption implies that $\zeta^k$ is a generator of the cyclic group $\{ \zeta^l \mid l\in[0,2n]\}$ and $(1-\zeta^{kl})/(1-\zeta^k)\in \mathcal{O}_{\mathbb{Q}(\zeta)}=\mathbb{Z}[\zeta]$ using the geometric series as above. The claim $(1-\zeta^n)/(1-\zeta^k)\in\mathcal{O}_{\mathbb{Q}(\zeta)}$ can be shown analogously. The claim $(1-\zeta)/(1-\zeta^k)\in\mathcal{O}_{\mathbb{Q}(\zeta)}$ also implies that $\sin(\alpha)/\sin(k\alpha)\in\mathcal{O}_K$. Similarly $\sin(n\alpha)/\sin(k\alpha)\in\mathcal{O}_K$.
\end{proof}

\subsection{The Galois group}
\label{Section:Galois}

The field extension $\mathbb{Q}(\zeta)/\mathbb{Q}$ is Galois. Its Galois group is isomorphic to the multiplicative group $(\mathbb{Z}/(2n+1)\mathbb{Z})^{\times}$. Explicitly, an isomorphism is given by
\begin{align*}
(\mathbb{Z}/(2n+1)\mathbb{Z})^{\times} &\to\operatorname{Gal}(\mathbb{Q}(\zeta),\mathbb{Q})\\
l+(2n+1)\mathbb{Z}&\mapsto m_l
\end{align*} 
where $m_l\colon \mathbb{Q}(\zeta)\to\mathbb{Q}(\zeta)$ maps $\zeta^r\mapsto \zeta^{rl}$ for every $r\in [0,2n]$.

Notice that every $m_l$ (with $l$ coprime to $2n+1$) leaves the subfield $K$ invariant. Hence the restriction defines a homomorphism of fields 
\begin{align}
\label{Eqn:FieldHom}\nonumber
\sigma_l\colon K&\to K\\
2\cos(2r\alpha)&\mapsto 2\cos(2rl\alpha)
\end{align}
for every $r\in [0,2n]$. These homomorphisms are actually automorphisms since they define a permutation on the basis $\{2\cos(2k\alpha)\mid k\in U\}$ thanks to the coprimality of $l$ and $2n+1$. However, they are not pairwise different. If $l\in [1,n]$ is coprime to $2n+1$, then $2n+1-l$ is coprime to $2n+1$ as well and $m_l=m_{2n+1-l}$ because $(2l+2(2n+1-l))\alpha=2\pi$. In fact, the map
\begin{align*}
(\mathbb{Z}/(2n+1)\mathbb{Z})^{\times}/\{\pm 1\} &\to \operatorname{Aut}(K)\\
l&\mapsto m_l
\end{align*}
is an isomorphism of groups. In particular, the automorphism group of $K$ is abelian.

\begin{prop}
Suppose that $r,l\in [1,n]$ such that $\gcd(l,2n+1)=1$. Then
\begin{align*}
\sigma_l\left(\frac{1}{\sin^2(r\alpha)}\right)=\frac{1}{\sin^2(rl\alpha)}
\end{align*}
where $\sigma$ is defined as in equation (\ref{Eqn:FieldHom}).
\end{prop}

\begin{proof}
The claim follows from the identity
\begin{align*}
\frac{1}{\sin^2(r\alpha)}=\frac{2}{1-\cos(2r\alpha)}.
\end{align*}
because $\sigma_l$ is a map of fields.
\end{proof}

\section{Linear independence of translation vectors}
\label{Sec:LinInd}

\subsection{An estimate of the coefficients}

In this section we want to show that the set  
\begin{align*}
\left\lbrace\, \frac{1}{\sin^2(k\alpha)} \mid k \in U\,\right\rbrace
\end{align*}
is linearly independent over $\mathbb{Q}$. This set of equal to the set of all translation vectors from Definition~ \ref{Defn:TVectors}, scaled by inverse of the common factor $d_1\sin(\alpha)\sin(n\alpha)$. Since the cardinality of the set $U=\left\lbrace \, k\in[1,n] \mid \gcd(k,2n+1)=1\,\right\rbrace$ is equal to the degree of the field extension $K/\mathbb{Q}$, this implies that this set of translation vectors is a basis of $K$ as a vector space over $\mathbb{Q}$. 

The following result is a special case of a Verlinde formula. A more general formula and a proof can be found in the article by Zagier, see \cite{Z}. Our formula can be obtained from Zagier's $D(g,k)$ by putting $k=2$.

\begin{prop}
\label{Prop:Verlinde}
For every $n\geq 1$ we have
\begin{align*}
\sum_{k=1}^n\frac{1}{\sin^2(k\alpha)}=\frac{2}{3}n(n+1).
\end{align*}
\end{prop}

Clearly, the first term in the sum is larger than every other summand. In fact, the first summand is larger than the sum of the remaining terms as the following lemma shows.

\begin{lemma}
\label{Prop:Estimate}
For every $n\geq 1$ we have
\begin{align*}
\sum_{k=2}^n\frac{1}{\sin^2(k\alpha)}<\frac{1}{\sin^2(\alpha)}.
\end{align*}
\end{lemma}

\begin{proof}
By Proposition \ref{Prop:Verlinde} the claim is equivalent to the inequalities
\begin{align*}
\frac{2}{3}n(n+1)<\frac{2}{\sin^2(\alpha)} \,\Leftrightarrow\, \sin^2\left(\frac{\pi}{2n+1}\right)<\frac{3}{n(n+1)}.
\end{align*}
It is well known that $\sin(x)< x$ for all $x>0$. Moreover, using Archimedes' bound $\pi<22/7$ it is easy to see that $\pi<2\sqrt{3}$ or $\pi^2<12$. We conclude that
\begin{align*}
\sin^2\left(\frac{\pi}{2n+1}\right)<\frac{\pi^2}{4n^2+4n+1}<\frac{12}{4n^2+4n}=\frac{3}{n(n+1)},
\end{align*}
which finishes the proof of the claim.
\end{proof}

\subsection{Determinants of group characters}

In this subsection we recall the definition of a group character and present a classical result about a determinant constructed from group characters. The determinant will play a crucial role in the proof of the linear independence of the translation vectors.

\begin{defn}[Group character] Let $G$ be a finite group. A \emph{character} is a group homomorphism from $G$ to $\mathbb{C}^{\times}$.
\end{defn}

Since every element $g$ in a finite group $G$ has finite order, the image of every character $\chi$ must lie in the unit circle $S^1\subseteq \mathbb{C}^{\times}$. 

The following theorem has a long and colourful history. In the special case of cyclic groups the theorem yields a factorisation of the determinant of a circulant matrix which was first proved by Catalan. Our formulation is due to Dedekind although Burnside proved a related statement. The theorem was generalised to all finite groups by Frobenius.

\begin{theorem}[Dedekind] 
\label{Thm:CharacterDets}
Let $G=\{g_1,\ldots,g_t\}$ be a finite abelian group of order $t$. Suppose that $R= \mathbb{C}[X_g\mid g\in G]$ is the polynomial ring in $t$ variables indexed by the elements of $G$. We define a matrix $M\in\operatorname{Mat}_{t\times t}(R)$ by putting $M_{ij}=X_{g_ig_j^{-1}}$. Then the determinant is given by
\begin{align*}
\operatorname{det}(M)=\prod_{\chi\colon G\to S^1}\left(\sum_{g\in G}\chi(g)X_g\right). 
\end{align*}
where the sum runs over all characters of the group $G$. In particular, every factor on the right hand side is an irreducible homogeneous polynomial of degree $1$.
\end{theorem}

\subsection{Linear independence of translation vectors}

Recall that a $\mathbb{Q}$-basis $B$ of a Galois extension $F/\mathbb{Q}$ is called \emph{normal} if there exists an element $a\in F$ such that 
\begin{align*}
B=\left\lbrace\, \sigma(a)\mid \sigma\in\operatorname{Gal}(F,\mathbb{Q})\,\right\rbrace.
\end{align*}
The Normal Basis Theorem asserts that every Galois extension has a normal basis. The following statement is a variation of a well-known argument which plays a role in the proof of the Normal Basis Theorem. Note that we do not assume that the field extension is Galois.

\begin{lemma}
\label{Prop:LinearIndependence}
Let $F/\mathbb{Q}$ be a field extension. Assume that $\phi_1,\ldots,\phi_r\colon F\to F$  are isomorphisms of fields. Let $a\in K$. We define a matrix $M\in \operatorname{Mat}_{r\times r}(K)$ by $M_{ij}=\phi_i^{-1}(\phi_j(a))$. Suppose that $\operatorname{det}(M)\neq 0$. Then $\phi_1(a),\ldots,\phi_r(a)$ are linearly independent over $\mathbb{Q}$.
\end{lemma}

\begin{proof}
Suppose that $\lambda_1,\ldots,\lambda_r\in \mathbb{Q}$ such that $\lambda_1\phi_1(a)+\ldots+\lambda_r\phi_r(a)=0$. Note that every field automorphism $\phi\in\operatorname{Aut}(K)$ restricts to the identity on $\mathbb{Q}\subseteq F$ since $\phi(1)=1$. We apply $\phi_1^{-1},\ldots,\phi_r^{-1}$ to the above equation and obtain
\begin{align*}
M\left(\begin{matrix}\lambda_1\\\vdots\\\lambda_r\end{matrix}\right)=0.
\end{align*}
As $M$ is invertible, we can conclude that $\lambda_1=\ldots=\lambda_r=0$.
\end{proof}

\begin{theorem}
\label{Thm:Independence}
Let $n\geq 1$ and $\alpha=\frac{\pi}{2n+1}$. The set
\begin{align*}
\left\lbrace\, \frac{1}{\sin^2(k\alpha)} \mid k \in [1,n], \gcd(k,2n+1)=1\,\right\rbrace
\end{align*}
is linearly independent over $\mathbb{Q}$.
\end{theorem}

\begin{proof}
 Recall the abbreviation $U=\left\lbrace\, k \in [1,n]\mid \gcd(k,2n+1)=1\,\right\rbrace$. We apply Lemma \ref{Prop:LinearIndependence} to $a=1/\sin^2(\alpha)$ and the field automorphisms $\sigma_l\in\operatorname{Aut}(K)$ with $l\in U$ from Section \ref{Section:Galois}. It is sufficient to prove that the determinant of the matrix $M\in \operatorname{Mat}_{U\times U}(K)$ with
\begin{align*}
M_{rs}=(\sigma_r^{-1}\circ \sigma_s)(a)&&(r,s\in U)
\end{align*}
is not equal to zero. We compute the determinant by Theorem \ref{Thm:CharacterDets} of Dedekind. More precisely, we apply the theorem to $G=\operatorname{Aut}(K)$ and substitute $X_{m}=\sigma(a)$ for all $m\in G$. We obtain
\begin{align*}
\operatorname{det}(M)=\prod_{\chi\colon G\to S^1} \left(\sum_{l\in U}\chi(\sigma_l)\sigma_l(a)\right).
\end{align*}
We show that every factor in the right hand side of the equation is non-zero. Let $\chi\colon G\to S^1$ be a character. Notice that $1\in U$ because it is always coprime to $2n+1$. We have $\sigma_1(a)=a$ because $\sigma_1$ is the neutral element in $G$. The triangle inequality, the fact that the image $\operatorname{im}(\chi)$ is contained in the unit circle, and Lemma \ref{Prop:Estimate} yield the chain
\begin{align*}
\left\lvert\sum_{\substack{l\in U \\ l \neq 1}}\chi(\sigma_l)\sigma_l(a)\right\rvert\leq \sum_{\substack{l\in U \\ l \neq 1}}\sigma_l(a)\lvert \chi(\sigma_l)\rvert=\sum_{\substack{l\in U \\ l \neq 1}}\sigma_l(a)\leq \sum_{l=2}^{n}\sigma_l(a)<a=\lvert\chi(\sigma_1)\sigma_1(a)\rvert
\end{align*}
of inequalities. We can conclude that
\begin{align*}
\sum_{\substack{l\in U \\ l \neq 1}}\chi(\sigma_l)\sigma_l(a)\neq -\chi(\sigma_1)\sigma_1(a)
\end{align*}
because those are two complex numbers with different absolute values.
\end{proof}

\section{The structure of the exchange graph}
\label{Section:MainThm}

\subsection{The lattice of translation vectors}

Let us introduce the following lattices.

\begin{defn}[Length lattices] \begin{enumerate}
\item (Translation lattice from infinite regions) Put 
\begin{align*}
R=\left\langle\, d_1\frac{\sin(\alpha)\sin(n\alpha)}{\sin^2(k\alpha)} \mid k\in U\,\right\rangle_{\mathbb{Z}}.
\end{align*}
\item (Translation lattice) Let $L$ be the $\mathbb{Z}$-module spanned by the Euclidean lengths of all vectors $w\in \mathbb{E}^2$ such that the exchange graph of $(\mathbf{v}_0,B_0)$ contains two seeds $(\mathbf{v}_1,B_1)$ and $(\mathbf{v}_2,B_2)$ with $(\mathbf{v}_1,B_1)=(\mathbf{v}_2,B_2)+w$.
\item (Side lattice) Let $L'$ be the $\mathbb{Z}$-module spanned by the Euclidean lengths of sides in triangles $\Delta$ such that the exchange graph of $(\mathbf{v}_0,B_0)$ contains a seed $(\mathbf{v},B)$ that is represented by the triangle $\Delta$.
\end{enumerate}
\end{defn}

\begin{prop}
\label{Prop:LatticeInclusions}
\begin{enumerate}
\item 
There exists a natural number $d$ such that there are inclusions:
\begin{align*}
\xymatrix@C=0.0pc@R=0.0pc{
L&\subseteq&\frac{1}{d}\mathbb{Z}[2\cos(\alpha)]d_1\\
\rotatebox[origin=c]{90}{$\subseteq$}&&\rotatebox[origin=c]{90}{$\subseteq$}\\
R&\subseteq&\mathbb{Z}[2\cos(\alpha)]d_1
}
\end{align*}
\item The ranks of the lattices are equal to $\operatorname{rk}_{\mathbb{Z}}(R)=\operatorname{rk}_{\mathbb{Z}}(L)=\operatorname{rk}_{\mathbb{Z}}(\mathbb{Z}[2\cos(\alpha)]d_1)=\varphi(2n+1)/2$. 
\end{enumerate}
\end{prop}

\begin{proof} The inclusion $R\subseteq \mathbb{Z}[2\cos(\alpha)]d_1$ follows from Proposition \ref{Prop:RingIntegers} (2b) and closure of $\mathbb{Z}[2\cos(\alpha)]=\mathcal{O}_{K}$ under multiplication. The inclusion $R\subseteq L$ was established in Proposition \ref{Prop:InfiniteRegionTranslationVectors}.

To construct the integer $d$ note that 
\begin{align*}
L'\subseteq \left\langle\, d_1\frac{\sin(\alpha)\sin(n\alpha)}{\sin(k_1\alpha)\sin(k_2\alpha)} \mid k_1,k_2\in [1,n]\,\right\rangle_{\mathbb{Z}}
\end{align*}
by virtue of formula (\ref{Eqn:Sides}), see Subsection \ref{Section:Lengths}. The formula $2\sin(x)\sin(y)=\cos(x-y)-\cos(x+y)$ for $x,y\in\mathbb{R}$ implies that the numerators and the denominators of the generators belong to the field $K=\left\langle\, 2\cos(k\alpha)\mid k\in U\,\right\rangle_{\mathbb{Q}}$. Hence $L'\subseteq K d_1$. Let $e\in\mathbb{N}$ be the common denominator of all the fractions that occur as coefficients in $\mathbb{Q}$-linear combinations of the generators of $L'$ in the basis $\left\lbrace\, 2\cos(k\alpha)d_1\mid k\in U\,\right\rbrace$ of $d_1 K$. Then $L'\subseteq \frac{1}{e}\mathbb{Z}[2\cos(\alpha)]d_1$.

Let $w\in L$ be a translation vector. Hence, there exists a sequence $(v_i,B_i)$, $i\in [1,k]$ of seeds such that $(v_1,B_1)=(v_k,B_k)+w$ and two adjacent elements in the sequence are related to each other by a single mutation. We denote by $\Delta_1$ and $\Delta_k$ the triangles associated with the first and the last seed, and by $w_1,w_k\in\mathbb{E}^2$ two vertices of $\Delta_1$ and $\Delta_k$ with $w_1=w_k+w$. From this we can conclude that there are vectors $u_i$ with $i\in [1,l]$ with $w=\sum_{i=1}^{l}u_i$ such that $\lvert u_i\rvert\in L'$ is a side length in a triangle given by one of the seeds in the mutation sequence and angle between $u_i$ and $b$ is an integer multiple of $\alpha$ for every $i\in[1,l]$. So we may write $w=\sum_{i=1}^l \lvert u_i\rvert \cos(m_i\alpha)$ with $m_i\in\mathbb{Z}$. The formula $2\cos(x)\cos(y)=\cos(x+y)+\cos(x-y)$ and the inclusion $L'\subseteq \frac{1}{e}\mathbb{Z}[2\cos(\alpha)]d_1$ imply that $L\subseteq \frac{1}{d}\mathbb{Z}[2\cos(\alpha)]d_1$ for $d=2e$. The inclusion  $\mathbb{Z}[2\cos(\alpha)]d_1\subseteq \frac{1}{d}\mathbb{Z}[2\cos(\alpha)]d_1$ is automatic. 

%The first statement follows because $\sin(\frac{k\pi}{2n+1})\neq 0$. Note that $P_n$ and $U_k$ are monic integer polynomials in $\mathbb{Z}[t]$ (where $t=2x$). The irreducibility of $P_n$ and the fact that $t=2\cos(\frac{\pi}{2n+1})$ is a zero of $P_n$ but not of $U_k$ imply that the polynomials are coprime over $\mathbb{Q}$. Hence there exist polynomials $f,g\in \mathbb{Q}[t]$ (depending on $n$ and $k$) such that $fP_n+gU_k=1$. We plug in $t=2\cos(\frac{\pi}{2n+1})$ to obtain the inverse. 

The rank of the lattice $R$ is equal to $\varphi(2n+1)/2$ by Theorem \ref{Thm:Independence}. The lattice $\frac{1}{d}\mathbb{Z}[2\cos(\alpha)]d_1$ has the same rank. From this we can conclude that all intermediate lattices must have the same rank as well.
\end{proof}

\begin{que} Does the equality $R=L$ hold?
\end{que}

\subsection{The exchange graph}

In this subsection we describe the structure of the exchange graph of $(\mathbf{v}_0,B_0)$. We can summarise the previous discussion as follows. 

\begin{theorem}
\label{Thm:ExchangeGraph}
The exchange graph of $(\mathbf{v}_0,B_0)$ has the following structure. 
\begin{enumerate}
\item The acyclic belt $I+w$ is a full subgraph of the exchange graph for every $w\in L$.
\item Suppose that $n_1,n_2,n_3\in [0,2n+1]$ are natural numbers such that $\gcd(n_1,n_2,n_3)=1$, so that there exists a Euclidean triangle $\Delta$ with angles $n_1\alpha$, $n_2\alpha$, $n_3\alpha$. Furthermore, let $Q$ be a quiver corresponding to $\Delta$, which is acyclic if $\Delta$ is acute-angled and cyclic if $\Delta$ is obtuse-angled. Then there exist exactly two seeds in the exchange graph whose triangle is congruent to $\Delta$ and whose quiver is $Q$. Furthermore, two such seeds are related to each other by a translation by a vector in $L$ or a reflection across the belt line $b$. 
%\item Suppose that $n_1,n_2,n_3\in [0,2n+1]$ are natural numbers such that $\gcd(n_1,n_2,n_3)=1$ and the triangle $\Delta$ with angles $n_1\alpha$, $n_2\alpha$, $n_3\alpha$ is obtuse-angled. Suppose that $Q$ is a cyclic orientation of the triangle's sides. Then the exchange graph contains a seed whose triangle is congruent to $\Delta$ and whose quiver is $Q$. Two such seeds are related to each other by a translation by a vector in $L$. 
\end{enumerate}
\end{theorem}

%% obsolete
%\begin{defn}[Coefficient ring]
%We put $R=\mathbb{Q}\left[2\cos\left(\frac{\pi}{2n+1}\right)\right]$.
%\end{defn}
%%%%

%\begin{prop}
%Suppose that $k_1,k_2\in[1,n]$ are integers such that $\operatorname{gcd}(k_1,k_2,2n+1)=1$. Then $\sin(k_2\alpha)\neq 0$ and
%\begin{align*}
%\frac{\sin\left(k_1\alpha\right)}{\sin\left(k_2\alpha\right)}\in \mathcal{O}_K.
%\end{align*}
%\end{prop}
%
%
%
%\begin{proof}
%The first statement is obvious. For the second statement first note that
%\begin{align*}
%\frac{\sin\left(k_1\alpha\right)}{\sin\left(k_2\alpha\right)}=\frac{U_{k_1}(2\cos(\alpha))}{U_{k_2}(2\cos(\alpha))}\in K.
%\end{align*}
%Hence, to show that the element belongs to the ring of integers $\mathcal{O}_K$ it suffices to show that it is an algebraic integer.
%\end{proof}
%
%
%
%\begin{prop}
%Every translation vector between two acyclic belts is parallel to $b$ and its length belongs to $d_1R$. 
%\end{prop}
%
%\begin{proof}
%Such a translation vector is a sum of vectors whose lengths belong to $d_1R$ by the previous proposition. These vectors intersect $b$ under an angle of the form $\cos(\frac{k\pi}{2n+1})$ with $k\in [0,2n]$. The claim follows from the fact that $d_1R$ is closed under multiplication by these terms.
%\end{proof}

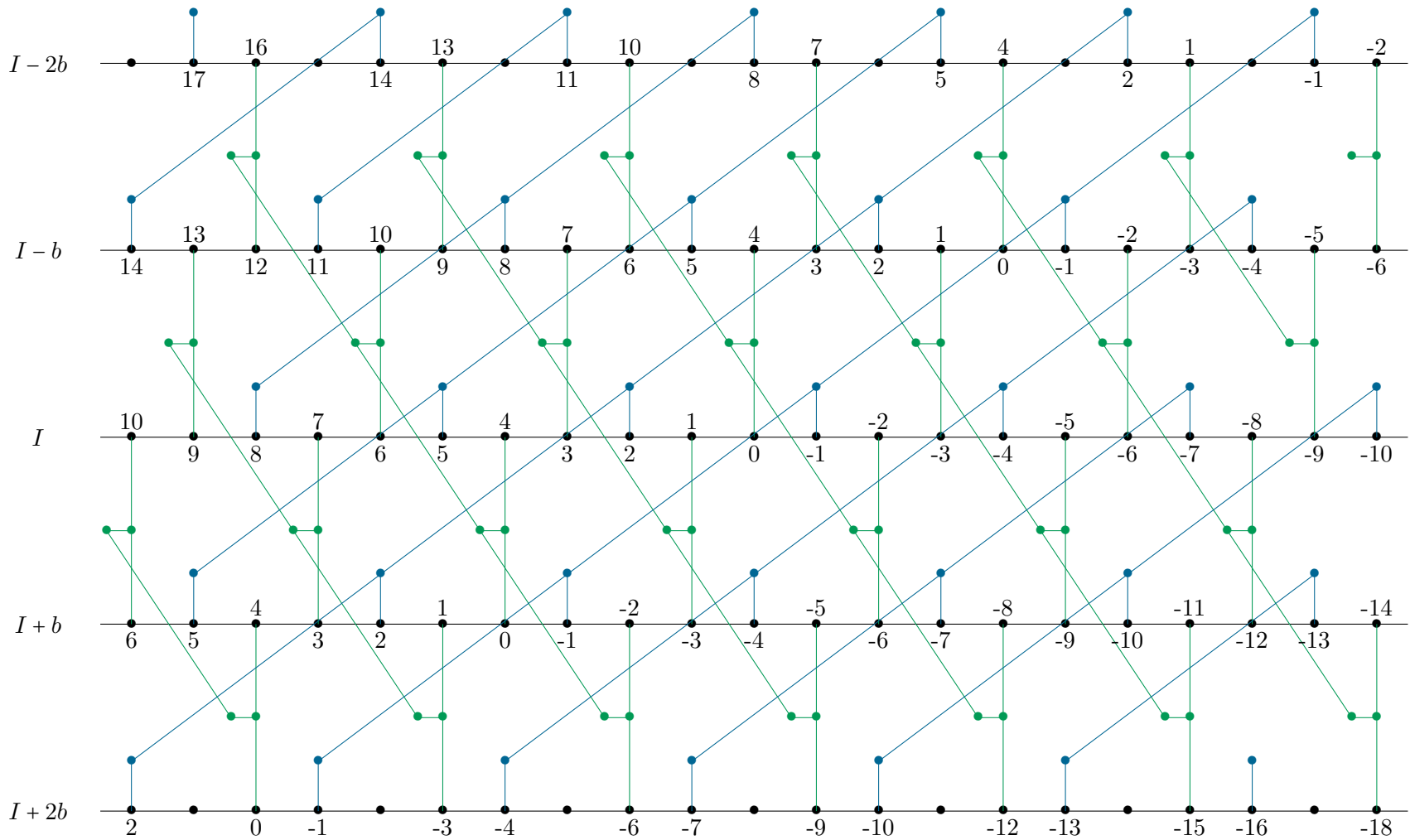
\begin{figure}[p]
\rotatebox{90}{
\begin{tikzpicture}

\newcommand{\x}{1cm} %% distance between two points in acyclic belt
\newcommand{\y}{3cm} %% vertical distance between two belts
\newcommand{\n}{10} %% we draw points -n to n in initial acyclic belt
\newcommand{\q}{8} %% useful parameter m=n-2
\newcommand{\m}{7} %% useful parameter m=n-3
\newcommand{\p}{6} %% useful parameter m=n-3

\newcommand{\s}{0.4cm} %% horizontal shift of the infinity-gons of first type
\newcommand{\ColorFirstInfinityGon}{ForestGreen} %% color of the infinity gons of type 1
\newcommand{\BendOne}{0} %% curviness of infinity gons of first type

\renewcommand{\t}{0.8cm} %% vertical shift of the infinity-gons of second type
\newcommand{\ColorSecondInfinityGon}{MidnightBlue} %% color of the infinity gons of type 2
\newcommand{\BendTwo}{0} %% curviness of infinity gons of second type

\pgfmathsetmacro{\label}{123} %% variable to store labels of points
\pgfmathsetmacro{\remainder}{123} %% variable to store values mod 3

%%% initial acyclic belt
\foreach \k in {-\n,...,\n} {
   \node at (\k*\x,0) {$\bullet$}; 
 } 
\path[draw] (-\n*\x-0.5*\x,0) to (0,0);
\path[draw] (0,0) to (3*\x,0);
\path[draw] (3*\x,0) to (\n*\x+0.5*\x,0);
\node at (-\n*\x-1.5*\x,0) {$I$};

%%% acyclic belt I-b
\foreach \k in {-\n,...,\n} {
   \node at (\k*\x,\y) {$\bullet$};
 } 
\path[draw] (-\n*\x-0.5*\x,\y) to (2*\x,\y);
\path[draw] (2*\x,\y) to (5*\x,\y);
\path[draw] (5*\x,\y) to (\n*\x+0.5*\x,\y);
\node at (-\n*\x-1.5*\x,\y) {$I-b$};

%%% acyclic belt I-2b
\foreach \k in {-\n,...,\n} {
   \node at (\k*\x,2*\y) {$\bullet$};
 } 
\path[draw] (-\n*\x-0.5*\x,2*\y) to (4*\x,2*\y);
\path[draw] (4*\x,2*\y) to (7*\x,2*\y);
\path[draw] (7*\x,2*\y) to (\n*\x+0.5*\x,2*\y);
\node at (-\n*\x-1.5*\x,2*\y) {$I-2b$};

%%% acyclic belt I+b
\foreach \k in {-\n,...,\n} {
   \node at (\k*\x,-\y) {$\bullet$};
 } 
\path[draw] (-\n*\x-0.5*\x,-\y) to (-3*\x,-\y);
\path[draw] (-3*\x,-\y) to (0*\x,-\y);
\path[draw] (0*\x,-\y) to (\n*\x+0.5*\x,-\y);
\node at (-\n*\x-1.5*\x,-\y) {$I+b$};

%%% acyclic belt I+2b
\foreach \k in {-\n,...,\n} {
   \node at (\k*\x,-2*\y) {$\bullet$};
 } 
\path[draw] (-\n*\x-0.5*\x,-2*\y) to (-5*\x,-2*\y);
\path[draw] (-5*\x,-2*\y) to (-2*\x,-2*\y);
\path[draw] (-2*\x,-2*\y) to (\n*\x+0.5*\x,-2*\y);
\node at (-\n*\x-1.5*\x,-2*\y) {$I+2b$};

%%% vertices between I and I-b
\foreach \k in {-\n,...,\n} {
   \pgfmathtruncatemacro{\remainder}{mod(\k,3)};
   \ifthenelse{\remainder=0}{
      \node at (\k*\x,0.5*\y) {\textcolor{\ColorFirstInfinityGon}{$\bullet$}};
      \path[draw,color=\ColorFirstInfinityGon] (\k*\x,0) to (\k*\x,\y);
      \pgfmathtruncatemacro{\label}{-\k+4};
      \node[above] at (\k*\x,\y) {\label};
      \pgfmathtruncatemacro{\label}{-\k};
      \node[below] at (\k*\x,0) {\label};
      \node at (\k*\x-\s,0.5*\y) {\textcolor{\ColorFirstInfinityGon}{$\bullet$}};
      \path[draw,color=\ColorFirstInfinityGon] (\k*\x-\s,0.5*\y) to (\k*\x,0.5*\y);
    }{
    }
 }

%%% vertices between I-b and I-2b
\foreach \k in {-\n,...,\n} {
   \pgfmathtruncatemacro{\remainder}{mod(\k,3)};
   \ifthenelse{\equal{\remainder}{1} \OR \equal{\remainder}{-2}}{
      \node at (\k*\x,1.5*\y) {\textcolor{\ColorFirstInfinityGon}{$\bullet$}};
      \path[draw,color=\ColorFirstInfinityGon] (\k*\x,\y) to (\k*\x,2*\y);
      \pgfmathtruncatemacro{\label}{-\k+8};
      \node[above] at (\k*\x,2*\y) {\label};
      \pgfmathtruncatemacro{\label}{-\k+4};
      \node[below] at (\k*\x,\y) {\label};
      \node at (\k*\x-\s,1.5*\y) {\textcolor{\ColorFirstInfinityGon}{$\bullet$}};
      \path[draw,color=\ColorFirstInfinityGon] (\k*\x-\s,1.5*\y) to (\k*\x,1.5*\y);
    }{
    }
 } 
 
 %%% vertices between I and I+b
\foreach \k in {-\n,...,\n} {
   \pgfmathtruncatemacro{\remainder}{mod(\k,3)};
   \ifthenelse{\equal{\remainder}{-1} \OR \equal{\remainder}{2}}{
      \node at (\k*\x,-0.5*\y) {\textcolor{\ColorFirstInfinityGon}{$\bullet$}};
      \path[draw,color=\ColorFirstInfinityGon] (\k*\x,0) to (\k*\x,-\y);
      \pgfmathtruncatemacro{\label}{-\k-4};
      \node[below] at (\k*\x,-\y) {\label};
      \pgfmathtruncatemacro{\label}{-\k};
      \node[above] at (\k*\x,0) {\label};
      \node at (\k*\x-\s,-0.5*\y) {\textcolor{\ColorFirstInfinityGon}{$\bullet$}};
      \path[draw,color=\ColorFirstInfinityGon] (\k*\x-\s,-0.5*\y) to (\k*\x,-0.5*\y);
    }{
    }
 } 

 %%% vertices between I+b and I+2b
\foreach \k in {-\n,...,\n} {
   \pgfmathtruncatemacro{\remainder}{mod(\k,3)};
   \ifthenelse{\equal{\remainder}{1} \OR \equal{\remainder}{-2}}{
      \node at (\k*\x,-1.5*\y) {\textcolor{\ColorFirstInfinityGon}{$\bullet$}};
      \path[draw,color=\ColorFirstInfinityGon] (\k*\x,-\y) to (\k*\x,-2*\y);
      \pgfmathtruncatemacro{\label}{-\k-8};
      \node[below] at (\k*\x,-2*\y) {\label};
      \pgfmathtruncatemacro{\label}{-\k-4};
      \node[above] at (\k*\x,-\y) {\label};
      \node at (\k*\x-\s,-1.5*\y) {\textcolor{\ColorFirstInfinityGon}{$\bullet$}};
      \path[draw,color=\ColorFirstInfinityGon] (\k*\x-\s,-1.5*\y) to (\k*\x,-1.5*\y);
    }{
    }
 }

%%% infinity gons type 1
\foreach \k in {-\n,...,\q} {
   \pgfmathtruncatemacro{\remainder}{mod(\k,3)};
   
   \ifthenelse{\equal{\remainder}{1} \OR \equal{\remainder}{-2}}{
      \path[draw,color=\ColorFirstInfinityGon,bend right=\BendOne] (\k*\x-\s,1.5*\y) to (\k*\x+2*\x-\s,0.5*\y);
   }{}
      
   \ifthenelse{\equal{\remainder}{0}}{
      \path[draw,color=\ColorFirstInfinityGon,bend right=\BendOne] (\k*\x-\s,0.5*\y) to (\k*\x+2*\x-\s,-0.5*\y);  
    }{}
    
    \ifthenelse{\equal{\remainder}{-1} \OR \equal{\remainder}{2}}{
      \path[draw,color=\ColorFirstInfinityGon,bend right=\BendOne] (\k*\x-\s,-0.5*\y) to (\k*\x+2*\x-\s,-1.5*\y);
   }{}
 }

 %%% infinity gons of type 2
 \foreach \k in {-\n,...,\n}{
   \pgfmathtruncatemacro{\remainder}{mod(\k,3)};
   
   \ifthenelse{\equal{\remainder}{1} \OR \equal{\remainder}{-2}}{
      \node at (\k*\x,\t) {\textcolor{\ColorSecondInfinityGon}{$\bullet$}};
      \pgfmathtruncatemacro{\label}{-\k};
      \node[below] at (\k*\x,0) {\label};
      \path[draw,color=\ColorSecondInfinityGon] (\k*\x,\t) to (\k*\x,0);
    }{}     
    
    \ifthenelse{\equal{\remainder}{2} \OR \equal{\remainder}{-1}}{
      \node at (\k*\x,\y+\t) {\textcolor{\ColorSecondInfinityGon}{$\bullet$}};
       \pgfmathtruncatemacro{\label}{-\k+4};
       \node[below] at (\k*\x,\y) {\label};
       \path[draw,color=\ColorSecondInfinityGon] (\k*\x,\y+\t) to (\k*\x,\y);
    }{}   
    
     \ifthenelse{\equal{\remainder}{0}}{
       \node at (\k*\x,-\y+\t) {\textcolor{\ColorSecondInfinityGon}{$\bullet$}};
       \pgfmathtruncatemacro{\label}{-\k-4};
       \node[below] at (\k*\x,-\y) {\label};
       \path[draw,color=\ColorSecondInfinityGon] (\k*\x,-\y+\t) to (\k*\x,-\y);
    }{}  
    
    \ifthenelse{\equal{\remainder}{-1} \OR \equal{\remainder}{2}}{
       \node at (\k*\x,-2*\y+\t) {\textcolor{\ColorSecondInfinityGon}{$\bullet$}};
       \pgfmathtruncatemacro{\label}{-\k-8};
       \node[below] at (\k*\x,-2*\y) {\label};
       \path[draw,color=\ColorSecondInfinityGon] (\k*\x,-2*\y+\t) to (\k*\x,-2*\y);
    }{}  
    
     \ifthenelse{\equal{\remainder}{0}}{
       \node at (\k*\x,2*\y+\t) {\textcolor{\ColorSecondInfinityGon}{$\bullet$}};
       \pgfmathtruncatemacro{\label}{-\k+8};
       \node[below] at (\k*\x,2*\y) {\label};
       \path[draw,color=\ColorSecondInfinityGon] (\k*\x,2*\y+\t) to (\k*\x,2*\y);
    }{}  
 }
 
  \foreach \k in {-\p,...,\n}{
   \pgfmathtruncatemacro{\remainder}{mod(\k,3)};
   
     \ifthenelse{\equal{\remainder}{0}}{
       \path[draw,color=\ColorSecondInfinityGon,bend right=\BendTwo] (\k*\x,2*\y+\t) to (\k*\x-4*\x,\y+\t);
       \path[draw,color=\ColorSecondInfinityGon,bend right=\BendTwo] (\k*\x,-\y+\t) to (\k*\x-4*\x,-2*\y+\t);
    }{}   
    
       \ifthenelse{\equal{\remainder}{-1} \OR \equal{\remainder}{2}}{
       \path[draw,color=\ColorSecondInfinityGon,bend right=\BendTwo] (\k*\x,\y+\t) to (\k*\x-4*\x,\t);
    }{}   
    
       \ifthenelse{\equal{\remainder}{1} \OR \equal{\remainder}{-2}}{
       \path[draw,color=\ColorSecondInfinityGon,bend right=\BendTwo] (\k*\x,\t) to (\k*\x-4*\x,-\y+\t);
    }{}   
    
}

\end{tikzpicture}}
\caption{The exchange graph}
\label{Figure:ExchangeGraph}
\end{figure}

Let us illustrate the result by an example.

\begin{ex} Suppose that $n=2$. The exchange graph of $(\mathbf{v}_0,B_0)$ is shown in Figure \ref{Figure:ExchangeGraph}. We denote by $w\in\mathbb{E}^2$ the generator of the $1$-dimensional lattice $L$. The exchange graph has the following properties. 
\begin{enumerate}
\item The acyclic belt $I+nw$ is a full subgraph of the exchange graph for every $n\in\mathbb{Z}$.
\item (Coloured green) For every $(n,m)\in\mathbb{Z}\times 3\mathbb{Z}$ the exchange graph contains additional vertices $R_{n,m}$ and $S_{n,m}$ such that the following conditions hold:
\begin{itemize}
\item The vertex $R_{n,m}$ is adjacent to $(I+nw)_m$, $(I+(n-1)w)_{m+4}$ and $S_{n,m}$.
\item The vertex $S_{n,m}$ is adjacent to $R_{n,m}$, $S_{n-1,m+6}$ and $S_{n+1,m-6}$.
\end{itemize}
% In particular, the exchange graph contains a heptagon with vertices
%\begin{align*}
%(I+nw)_{m},R_{n,m},S_{n,m},S_{n+1,m-6},R_{n+1,m-6},(I+nw)_{m-2},(I+nw)_{m-1}
%\end{align*}
%for every $(n,m)\in\mathbb{Z}\times 3\mathbb{Z}$ and an infinity-gon $\{S_{n,m-6n}\colon n\in\mathbb{Z}\}$ for every $m\in3\mathbb{Z}$.
\item (Coloured blue) For every $(n,m)\in\mathbb{Z}\times (2+3\mathbb{Z})$ the exchange graph contains an additional vertex $T_{n,m}$ such that $T_{n,m}$ adjacent to $(I+nw)_m$, $T_{n-1,m}$ and $T_{n+1,m}$. 
%In particular, the exchange graph contains a nonagon with vertices
%\begin{align*}
%&(I+nw)_{m+2},(I+nw)_{m+1},(I+nw)_{m},T_{n,m},T_{n+1,m},\\
%&(I+(n+1)w)_{m},(I+(n+1)w)_{m-1},(I+(n+1)w)_{m-2},R_{n+1,m-2}
%\end{align*}
%for every $(n,m)\in\mathbb{Z}\times (2+3\mathbb{Z})$ and an infinity-gon $\{\,T_{n,m}\mid n\in\mathbb{Z}\,\}$ for every $m\in(2+3\mathbb{Z})$.
\end{enumerate}
\end{ex}

\subsection{Quasi-isometries}

\begin{defn}[Quasi-isometries] Suppose that $(M_1,d_1)$ and $(M_2,d_2)$ are metric spaces. A map $f\colon M_1\to M_2$ is called a \emph{quasi-isometry} if the following conditions hold.
\begin{enumerate}
\item There exist real numbers $a\geq 1$ and $b\geq 0$ such that
\begin{align*}
\frac{1}{a} d_1\left(x,y\right)-b\leq d_2\left(f(x),f(y)\right) \leq a d_1\left(x,y\right)+b
\end{align*}
for all $x,y\in M_1$.
\item There exists a real number $c\geq 0$ such that for every $z\in M_2$ there exists $x\in M_1$ with $d_2\left(z,f(x)\right)\leq c$.
\end{enumerate}
We say that $(M_1,d_1)$ and $(M_2,d_2)$ are \emph{quasi-isometric} if there exists a quasi-isometry $f\colon M_1\to M_2$.
\end{defn}

It is known that the composition of two quasi-isometries is again a quasi-isometry, and that if $f\colon M_1\to M_2$ is a quasi-isometry with constants $a,b,c$, then the map $g\colon M_2\to M_1$, where $g(z)=x$ for an arbitrary $x\in M_1$ such that $d_2\left(z,f(x)\right)\leq c$, is again a quasi-isometry. It follows that quasi-isometry is an equivalence relation on metric spaces. 

Every isometry is a quasi-isometry, but the converse is false in general. Other examples of quasi-isometries are the inclusions $\mathbb{Z}^k\hookrightarrow\mathbb{R}^k$ and $k\mathbb{Z}\hookrightarrow\mathbb{Z}$ for each $k\geq 1$.

Let $\Gamma=(V_0,V_1)$ be a simple graph with vertex set $V_0$ and edge set $V_1$. Recall that the distance $d$ between two vertices in $\Gamma$ is the number of edges in a shortest path between the two vertices. In this way, $(V_0,d)$ becomes a metric space. Also we may view $(V_0,V_1)$ as a $1$-dimensional cell complex, and the distance $d$ induces a metric on $(V_0,V_1)$. Notice that the embedding $(V_0,d)\hookrightarrow ((V_0,V_1),d)$ is a quasi-isometry.

Another source of quasi-isometries is group theory. Recall that a generating set $S$ of a group $G$ is called \emph{symmetric} if $S=S^{-1}$ and $S$ does not contain the identity.

\begin{nota}[Cayley graph] Assume that a group $G$ is generated by a finite symmetric set $S\subseteq G$. By $\operatorname{Cay}_S(G)$ we denote the \emph{Cayley graph} of $G$ (i.e. the graph with vertex set $G$ and an edge between $g,h\in G$ if $gh^{-1}\in S$).  
\end{nota}

It is well known that if a graph has two finite, symmetric generating sets $S,S'\subseteq G$, then $\operatorname{id}\colon G\to G$ induces a quasi-isomorphism between $\operatorname{Cay}_S(G)$ and $\operatorname{Cay}_{S'}(G)$. 

Suppose that $G$ is a group with identity $e$ and $M$ is a set. Recall that a \emph{group action} of $G$ on $M$ is a map $G\times M\to M$, $(g,m)\mapsto gm$ such that $ex=x$ for all $x\in M$ and $(gh)x=g(hx)$ for all $g,h\in M$ and $x\in X$.  

\begin{defn}[Group actions on metric spaces] Assume that the group $G$ acts on a metric space $(M,d)$. 
\begin{enumerate}
\item We say the group action is \emph{isometric} if $d(gx,gy)=d(x,y)$ for all $g\in G$ and $x,y\in M$.
\item We say the group action is \emph{properly discontinuous} if the set $\{g\in G\mid d(x,gx)\leq r\}$ is finite for every $x\in M$ and $r\geq 0$.
\item We say the group action is \emph{cocompact} if the orbit space $M/G$ is compact with respect to the quotient topology.
\end{enumerate}
\end{defn}

Recall that a metric space $(M,d)$ is called \emph{proper} if the closed ball $B_r(x)=\{y\in M\mid d(x,y)\leq r\}$ is compact for every $x\in M$ and $r>0$. It is called \emph{geodesic} if for all $x,y\in M$ there exists a \emph{geodesic} between $x$ and $y$, that is, an isometric embedding $p\colon [0,d]\to M$ such that $p(0)=x$ and $p(d)=y$ where $d=d(x,y)$.

\begin{theorem}[Schwarz \cite{Sch} and Milnor \cite{M}] Suppose that a group $G$ acts on a proper, geodesic metric space $(M,d)$ and that the group action is isometric, properly discontinuous and cocompact. Then $G$ is finitely generated. Moreover, $(M,d)$ is quasi-isometric to $\operatorname{Cay}_S(G)$ for every finite, symmetric generating set $S\subseteq G$. To be concrete, fix $x\in M$. Then the map 
\begin{align*}
G\to M,\,g\mapsto gx
\end{align*} 
is a quasi-isomorphism. 
\end{theorem}

\begin{coro}
\label{Coro:QItoLattice}
The exchange graph $\Gamma$ is quasi-isometric to the Cayley graph of the lattice $L$, and a quasi-isomorphism is given by the translation of the initial seed i.e. by the map
\begin{align*}
L&\to \Gamma,\,w\mapsto w+(\mathbf{v}_0,B_0).
\end{align*}
\end{coro}

\begin{proof}
The graph $\Gamma$ (viewed as a $1$-dimensional cell complex) is a geodesic metric space. It is proper because every ball $B_r(x)$ in $\Gamma$ is sequentially compact and hence compact.

We consider the translation map 
\begin{align*}
L\times \Gamma&\to \Gamma\\
(w,(\mathbf{v},B))&\mapsto w+(\mathbf{v},B).
\end{align*}
The definition of $L$ implies that the map is well-defined, that is, the image $w+(\mathbf{v},B)$ belongs to $\Gamma$ for all $w\in L$ and all seeds $(\mathbf{v},B)$. It is a group action by definition. Proposition \ref{Prop:GroupAction} implies that the group action is isometric. Let $x=(\mathbf{v},B)$ be a vertex of $\Gamma$ and $r\geq 0$. Then the number of vertices $y$ of $\Gamma$ with $d(x,y)\leq r$ is finite. For every such $y$ there is at most one $w\in L$ such that $y=w+x$. We see that the group action is properly discontinuous. By construction every seed $x=(\mathbf{v},B)$ is given by a quiver and a triangle such that the angles are given by $n_1\alpha$, $n_2\alpha$ and $n_3\alpha$ for some natural numbers satisfying $n_1+n_2+n_3=2n+1$. By Proposition \ref{Prop:FinitelyManyTriangles} the exchange graph contains at most $2$ triangles of that kind for every quiver $Q$ and every triple $(n_1,n_2,n_3)$ up to translation. Hence $\Gamma /L$ is a finite graph. In particular, it is compact so that the group action is cocompact. The claim follows from the theorem of Schwarz and Milnor.  
\end{proof}

Choose a basis $B=\{l_1,\ldots,l_r\}$ of the lattice $L$ where $r=\operatorname{rk}_{\mathbb{Z}}(L)=\varphi(2n+1)/2$ denotes the rank. Then the Cayley graph of $L$ with respect to $B$ is isomorphic to $\mathbb{Z}^r$ with two lattice points $x,y\in\mathbb{Z}^r$ being connected by an edge if and only if their Euclidean distance is equal to $1$. In particular, $\Gamma$ is quasi-isometric to $\mathbb{Z}^r$.

\subsection{The growth rate of the exchange graph}

We use the abbreviation $x_0=(\mathbf{v}_0,B_0)$ for the initial seed. Recall that $\Gamma=(V_0,V_1)$ is the exchange graph of $x_0$. 

\begin{defn}[Growth function] The \emph{growth function} $gr\colon \mathbb{N}\to\mathbb{N}$ is defined by 
\begin{align*}
gr(n)=\lvert \left\lbrace\, x\in V_0\mid d(x_0,x)\leq n\right\rbrace\rvert.
\end{align*}
\end{defn}

\begin{defn}[Polynomial growth] We say that $\Gamma$ has \emph{polynomial growth} if $gr(n)=\mathcal{O}(n^r)$ for some $r\geq 0$. If this happens to be the case, then we call the smallest natural number $r$ such that $gr(n)=\mathcal{O}(n^r)$ the \emph{polynomial growth rate} of $\Gamma$. 
\end{defn}

For example, let us consider $\mathbb{Z}^r$. Then $f(n)$ is equal to the number of points $\mathbf{x}\in\mathbb{Z}^r$ such that $\sum_{i=1}^n\lvert x_i\rvert\leq n$. The sequence $gr(n)$ is also known as the \emph{crystal ball sequence} in the literature. It is well known that $gr(n)=\lambda n^r+\mathcal{O}(n^{r-1})$ where $\lambda=2^r/r!$ is a constant (depending on $r$ but not on $n$). In particular, $\mathbb{Z}^r$ has polynomial growth with growth rate $r$.

\begin{coro}
\label{Theo:GrowthRate}
The exchange graph $\Gamma$ has polynomial growth and its polynomial growth rate is equal to $\varphi(2n+1)/2$.
\end{coro}

\begin{proof} By Corollary \ref{Coro:QItoLattice}, $\Gamma$ is quasi-isometric to the Cayley graph of $L$ where $L$ is a lattice of rank $\varphi(2n+1)/2$. The Cayley graph of $L$ is isomorphic to $\mathbb{Z}^{\varphi(2n+1)/2}$ and therefore has polynomial growth with growth rate $r=\varphi(2n+1)/2$. 

Recall that $x_0=(\mathbf{v}_0,B_0)$ denotes the initial seed. We consider the map $f\colon L\to \Gamma$ given by $f(w)=w+x_0$. Corollary \ref{Coro:QItoLattice} asserts that $f$ is a quasi-isomorphism. Note that $f(0)=x_0$.

According to the definition of a quasi-isomorphism we can pick $a\geq 1$ and $b\geq 0$ such that $\frac{1}{a}d_L(u,w)-b\leq d_{\Gamma}(f(u),f(w))\leq ad_L(u,w)+b$ for all $u,w\in L$. Moreover, we can pick $c\geq 0$ such that for all $x\in V_0$ there exists $w\in L$ such that $d_{\Gamma}\left(x,f(w)\right)\leq c$. Here, the subscripts indicate the metric spaces of the distance functions. Notice that for every $w\in L$ there are only finitely many $x\in V_0$ satisfying $d_{\Gamma}\left(x,f(w)\right)\leq c$. In fact, since $\Gamma$ is a $3$-regular graph, the number of such $x$ can be bounded by $3^c$. Notice that this constant does not depend on $x$. 

Suppose that $x\in V_0$ is a seed in $\Gamma$. We choose a $w\in L$ such that $d_{\Gamma}\left(x,f(w)\right)\leq c$. The triangle inequality implies
\begin{align}
\label{Ineq1}
&d_L\left(w,0\right)\leq a\left[d_{\Gamma}(f(w),f(0))+b\right]\leq a\left[d_{\Gamma}(f(w),x)+d_{\Gamma}(x,x_0)+b\right]\leq ad_{\Gamma}(x,x_0)+a(b+c);\\
&d_L\left(w,0\right)\geq \frac{1}{a}\left[d_{\Gamma}(f(w),f(0))-b\right]\geq \frac{1}{a}\left[-d_{\Gamma}(f(w),x)+d_{\Gamma}(x,x_0)-b\right]\geq \frac{1}{a}d_{\Gamma}(x,x_0)-\frac{b+c}{a}.
\end{align}
From inequality (\ref{Ineq1}) we can conclude that
\begin{align*}
gr^{\Gamma}(n)\leq 3^cgr^{L}\left(an+a(b+c)\right)=\mathcal{O}(n^r).
\end{align*}
Here the superscripts indicate the metric space of the growth function. In particular, $\Gamma$ has polynomial growth, and its growth rate is at most $r$. For any $n\gg 0$ there are $\lambda n^r+\mathcal{O}(n^{r-1})$ pairwise distinct elements $w\in L$ such that $d_L(w,0)\leq n$. Notice that the quasi-isomorphism $f$ is injective by construction. Application of $f$ yields $\lambda n^r+\mathcal{O}(n^{r-1})$ pairwise distinct seeds $x$ in $\Gamma$ such that $d_{\Gamma}(x,x_0)\leq an+b+c=\mathcal{O}(n)$. 
\end{proof}

\section{Exchange graphs for even least common denominators}
\label{Section:FurtherExchangeGraphs}

\subsection{The structure of the exchange graphs}
\label{Section:EvenDenominators}

In this section we consider geometric mutations of seeds that are given by triangles in the Euclidean plane whose angles are rational multiples of $\pi$ where the least common denominator of the three rational multiples is \emph{even}.

Fix a natural number $n\geq 2$, and put $\alpha=\pi/2n$.  Suppose that the initial seed $(\mathbf{v}_0,B_0)$ is given by a triangle $\Delta_0=A_1^{(0)}A_2^{(0)}A_3^{(0)}$ with angles $A_1^{(0)}=\alpha$, $A_2^{(0)}=(n-1)\alpha$, and $A_3^{(0)}=n\alpha$. For $i\in\{1,2,3\}$ we denote the side of $\Delta$ opposite to $A_i^{(0)}$ by $a_i^{(0)}$. To construct a seed we introduce a quiver with vertices $a_1^{(0)}$, $a_2^{(0)}$, and $a_3^{(0)}$, and arrows $a_1^{(0)}\to a_3^{(0)}$ and $a_3^{(0)}\to a_2^{(0)}$. As before, we assume that the reference point lies infinitely far away on the line $b$, which is constructed from the initial triangle by the billiard geometry. It is easy to check that this is the unique compatible choice of it.

The geometric considerations in Section \ref{Section:GeoMut} do not depend on the parity of the common denominator. In particular, there is a line $b\subseteq \mathbb{E}^2$ that contains the feet of two altitudes of every triangle $\Delta$ in the exchange graph. In this case, the line $b$ contains the vertex with the right angle of the initial triangle. Moreover, the quantity $T(\Delta)=a_1\sin(A_2)\sin(A_3)$ is conserved and hence the same for every triangle $\Delta=A_1A_2A_3$ in $\Gamma$.  

Some statements in Section \ref{Section:NumberTheory} undergo a slight change when we switch to an even common denominator. Quintessentially, equation (\ref{Eq:CyclotomicInt}) does not hold anymore. We put $K=\mathbb{Q}(2\cos(2\alpha))$ and consider the ring of integers $\mathcal{O}_{K}$. Then $2\cos(\alpha)\notin K$. (For example, when $n=2$, then $K=\mathbb{Q}(2\cos(2\alpha))=\mathbb{Q}$ but $2\cos(\alpha)=\sqrt{2}\notin\mathbb{Q}$.) In particular,  
\begin{align*}
\mathcal{O}_K=\left\langle\, 2\cos(2k\alpha) \mid k \in [1,n],\,\operatorname{gcd}(k,2n)=1\,\right\rangle_{\mathbb{Z}}\subsetneqq\left\langle\, 2\cos(k\alpha) \mid k \in [1,n],\,\operatorname{gcd}(k,2n)=1\,\right\rangle_{\mathbb{Z}}.
\end{align*}
Notice that $\operatorname{rk}_{\mathbb{Z}}(\mathbb{Z}[2\cos(\alpha)])=\varphi(2n)$ is twice as large as $\operatorname{rk}_{\mathbb{Z}}(\mathcal{O}_K)=\varphi(2n)/2$. The lattice 
\begin{align*}
R=\left\langle\, d_1\frac{\sin(\alpha)\sin(n\alpha)}{\sin^2(k\alpha)} \mid k\in [1,n],\,\operatorname{gcd}(k,2n)=1\,\right\rangle_{\mathbb{Z}}
\end{align*}
 is generated by the lengths of all the finite sides bounding infinite regions in $\Gamma$. For every element $w\in R$ the translation map $(\mathbf{v},B)\mapsto w+(\mathbf{v},B)$ induces a symmetry of the exchange graph. We can show that $\operatorname{rk}_{\mathbb{Z}}(R)=\varphi(2n)/2$ similar to Section \ref{Section:Galois}. As before, let $L$ denote the lattice of all elements $w\in \mathbb{E}^2$ such that the translation map $(\mathbf{v},B)\mapsto w+(\mathbf{v},B)$ induces a symmetry of the exchange graph. As in Proposition \ref{Prop:LatticeInclusions} there exists a natural number $d$ such that there are inclusions:
\begin{align*}
\xymatrix@C=0.0pc@R=0.0pc{
L&\subseteq&\frac{1}{d}\mathbb{Z}[2\cos(\alpha)]d_1\\
\rotatebox[origin=c]{90}{$\subseteq$}&&\rotatebox[origin=c]{90}{$\subseteq$}\\
R&\subseteq&\mathbb{Z}[2\cos(\alpha)]d_1
}
\end{align*}
However, in this situation the lower left corner of the diagram does not have the same rank as the upper right corner as before. We conclude with the following theorem.

\begin{theorem}
\phantomsection\label{thm:even}\begin{enumerate}
\item For every $(n_1,n_2,n_3)\in[0,2n+1]$ with $\gcd(n_1,n_2,n_3)=1$ and every $w\in L$ the exchange graph contains at most two additional vertices represented by (finite or infinite) triangles with angles $n_1\alpha$, $n_2\alpha$, $n_3\alpha$ (with two different orientations). For a fixed triple $(n_1,n_2,n_3)$ together with a fixed orientation all these triangles are related to each other by translations by vectors in $L$.   
\item We have $\operatorname{rk}_{\mathbb{Z}}(L)=\varphi(2n)/2$ or $\operatorname{rk}_{\mathbb{Z}}(L)=\varphi(2n)$.
\item The exchange graph $\Gamma$ is quasi-isometric to the Cayley graph of the lattice $L$, and a quasi-isomorphism is given by the map $L\to \Gamma$ with $w\mapsto w+(\mathbf{v}_0,B_0)$ where $(\mathbf{v}_0,B_0)$ is the initial seed of $\Gamma$.
\end{enumerate}
\end{theorem}

\section*{Acknowledgments}
The authors were partially supported by EPSRC grant EP/N005457/1. The second author was partially supported by EPSRC grant EP/M004333/1.

\bibliographystyle{hyperalphaabbr}
\bibliography{geomutation}

\end{document}